\documentclass[10pt]{amsart}
\usepackage[T1]{fontenc}
\usepackage[utf8]{inputenc}
\usepackage[english]{babel}
\usepackage{csquotes} 
\usepackage[a4paper,margin=3cm]{geometry}
\allowdisplaybreaks
\usepackage[draft]{graphicx}
\usepackage[backend=biber,style=alphabetic,bibencoding=utf8,language=auto,autolang=other,giveninits=true,doi=false,isbn=false,url=false,maxnames=10,sorting=nty]{biblatex}
\DeclareNolabel{
  \nolabel{\regexp{[\p{P}\p{S}\p{C}\p{Z}]+}}
}
\usepackage{caption} 
\usepackage[hyperfootnotes=false]{hyperref} 
\hypersetup{
	colorlinks = true,
	linkcolor = {blue},
	urlcolor = {red},
	citecolor = {blue}
}
\usepackage[nameinlink,capitalise,sort]{cleveref}
\crefname{equation}{}{} 

\crefname{figure}{Figure}{Figures}

\usepackage{amsmath}
\usepackage{amsthm}
\usepackage{amssymb}
\usepackage{esint} 
\usepackage{mathrsfs} 
\usepackage{faktor} 
\usepackage{dsfont} 
\usepackage[dvipsnames]{xcolor}
\usepackage{array}
\usepackage{hhline}
\usepackage[normalem]{ulem}
\usepackage{enumitem} 
\usepackage{comment} 
\usepackage[toc,page]{appendix} 
\usepackage{imakeidx} 
\usepackage{xparse} 
\usepackage{mathtools}
\usepackage{fancyhdr} 
\usepackage{ifthen} 
\usepackage{forloop} 
\usepackage{xstring}
\usepackage{tikz}
\usetikzlibrary{positioning, arrows.meta}
\usetikzlibrary{babel} 
\usetikzlibrary{cd} 
\usepackage{tikz-cd} 
\usepackage{emptypage} 
\usepackage{listings} 
\usepackage{mathabx} 
\usepackage{subcaption}

\theoremstyle{plain}
\newtheorem{lemma}{Lemma}[section]

\newtheorem{theorem}[lemma]{Theorem}

\newtheorem{remark}[lemma]{Remark}

\theoremstyle{definition}

\theoremstyle{remark}

\numberwithin{equation}{section}

\newcommand{\R}{\mathbb{R}}

\newcommand{\T}{\mathbb{T}}

\newcommand{\sgn}[1]{\mathrm{sign}\left(#1\right)}

\newcommand{\p}{\partial}

\begin{document}


\title[Stability and instability of 1D MHD]{Stability and instability of a one-dimensional MHD model}

\author[N.~De Nitti]{Nicola De Nitti}
\address[N.~De Nitti]{Università di Pisa, Dipartimento di Matematica, Largo Bruno Pontecorvo 5, 56127 Pisa, Italy.}
\email[]{nicola.denitti@unipi.it}

\author[J.~Guo]{Jie Guo}
\address[J.~Guo]{Capital Normal University, School of Mathematical Sciences, West Third Ring North Road 105, Haidian District, 100048 Beijing, P.\,R.~China.}
\email[]{2230501023@cnu.edu.cn}

\author[Q.~Jiu]{Quansen Jiu}
\address[Q.~Jiu]{Capital Normal University, School of Mathematical Sciences, West Third Ring North Road 105, Haidian District, 100048 Beijing, P.\,R.~China.}
\email[]{jiuqs@cnu.edu.cn}



\begin{abstract}

We consider a one-dimensional magnetohydrodynamics model introduced by Dai \textit{et al.}~(2023), in a parameter regime where, in the absence of a magnetic field, the system reduces to the De Gregorio model for the Euler equations. We analyze stability and instability near the first excited state on the torus, thus generalizing the recent results obtained by Guo and Jiu~(2025) for the De Gregorio model. Specifically, we establish global well-posedness of the linearized system, local well-posedness for the nonlinear system, and demonstrate both linear and nonlinear instability for a broad class of initial data in the weighted Sobolev space introduced by Lai \textit{et al.}~(2020). We identify the principal linearized operator, which is structurally equivalent to that of the De Gregorio model, as the primary mechanism of instability. Moreover, we prove global well-posedness and stability of both linear and nonlinear systems for initial data in a particular subspace of the aforementioned weighted Sobolev space.

\end{abstract}

\maketitle

\section{Introduction}
\label{sec:intro}

\subsection{One-dimensional models in fluid dynamics and magnetohydrodynamics}
\label{ssec:models}

We consider the following \emph{one-dimensional magnetohydrodynamics} (MHD) \emph{model} with mixed vortex stretching effects, introduced in \cite{dai2023}:
\begin{align}\label{eq:MHD1}
\begin{cases}
\partial_t \omega^{+} + a\, u^{-} \partial_\theta \omega^{+} = p\, \omega^{+} H \omega^{-} + q\, \omega^{-} H \omega^{+}, & t > 0, \ \theta \in \mathbb{T}, \\
\partial_t \omega^{-} + a\, u^{+} \partial_\theta \omega^{-} = p\, \omega^{-} H \omega^{+} + q\, \omega^{+} H \omega^{-}, & t > 0, \ \theta \in \mathbb{T},
\end{cases} 
\end{align}
posed on the torus \(\mathbb{T} \coloneqq [-\pi, \pi]\), where \(a,\, p,\, q \in \mathbb{R}\), \(H\) denotes the \emph{Hilbert transform} on \(\mathbb{T}\), defined by
\[
H f(\theta) \coloneqq \frac{1}{2\pi} \ \mathrm{p.\,v.\,} \int_{-\pi}^\pi \cot\left( \frac{\theta - \vartheta}{2} \right) f(\vartheta) \, \mathrm d\vartheta,
\]
and  \(u^{\pm}\) are determined by
\begin{align}\label{eq:MHDu}
\partial_\theta u^{\pm} = H \omega^{\pm}, \quad \text{with} \quad u^{\pm}(t,0) \equiv 0, \quad t > 0, \ \theta \in \mathbb{T}.
\end{align}

Let us briefly recall the motivation for \crefrange{eq:MHD1}{eq:MHDu}, following the derivation presented in \cite[Section~1.2]{dai2023}. The system of \emph{ideal incompressible magnetohydrodynamics} in three space dimensions is given by
\begin{equation}\label{eq:MHD3D}
\begin{cases}
\partial_t \boldsymbol{v} + (\boldsymbol{v} \cdot \nabla) \boldsymbol{v} - (\boldsymbol{b} \cdot \nabla) \boldsymbol{b} + \nabla P = 0, & t > 0, \ \theta \in \mathbb{T}^3,\\
\partial_t \boldsymbol{b} + (\boldsymbol{v} \cdot \nabla) \boldsymbol{b} - (\boldsymbol{b} \cdot \nabla) \boldsymbol{v} = 0, & t > 0, \ \theta \in \mathbb{T}^3, \\
\nabla \cdot \boldsymbol{v} = 0, \quad \nabla \cdot \boldsymbol{b} = 0, & t > 0, \ \theta \in \mathbb{T}^3,
\end{cases}
\end{equation}
where the vector fields $\boldsymbol{v}$ and $\boldsymbol{b}$ denote the fluid velocity and magnetic field, respectively; the scalar function $P$ is the pressure. 
Introducing the \emph{Els\"asser variables}~\cite{PhysRev},
\begin{align}\label{eq:ElV}
\boldsymbol{u}^+ \coloneqq \boldsymbol{v} + \boldsymbol{b}, \qquad \boldsymbol{u}^- \coloneqq \boldsymbol{v} - \boldsymbol{b},
\end{align}
the system \cref{eq:MHD3D} is symmetrized:
\begin{equation}\label{eq:EV}
\begin{cases}
\partial_t \boldsymbol{u}^+ + (\boldsymbol{u}^- \cdot \nabla) \boldsymbol{u}^+ + \nabla P = 0, & t > 0, \ \theta \in \mathbb{T}^3, \\
\partial_t \boldsymbol{u}^- + (\boldsymbol{u}^+ \cdot \nabla) \boldsymbol{u}^- + \nabla P = 0, & t > 0, \ \theta \in \mathbb{T}^3,\\
\nabla \cdot \boldsymbol{u}^+ = 0, \quad \nabla \cdot \boldsymbol{u}^- = 0, & t > 0, \ \theta \in \mathbb{T}^3.
\end{cases}
\end{equation}
Defining
\[
\boldsymbol{\omega}^+ \coloneqq \nabla \times \boldsymbol{u}^+ = \nabla \times \boldsymbol{v} + \nabla \times \boldsymbol{b}, \qquad \boldsymbol{\omega}^- \coloneqq \nabla \times \boldsymbol{u}^- = \nabla \times \boldsymbol{v} - \nabla \times \boldsymbol{b},
\]
(where $\nabla \times \boldsymbol{v}$ and $\nabla \times \boldsymbol{b}$ are the \emph{vorticity} and \emph{magnetic current}, respectively) we have, from \cref{eq:EV},
\begin{equation}\label{eq:VVMHD0}
\begin{cases}
\partial_t \boldsymbol{\omega}^+ + (\boldsymbol{u}^- \cdot \nabla) \boldsymbol{\omega}^+ - (\boldsymbol{\omega}^+ \cdot \nabla) \boldsymbol{u}^- + \nabla \times (\boldsymbol{u}^- \nabla \boldsymbol{u}^+) = 0, & t > 0, \ \theta \in \mathbb{T}^3,\\
\partial_t \boldsymbol{\omega}^- + (\boldsymbol{u}^+ \cdot \nabla) \boldsymbol{\omega}^- - (\boldsymbol{\omega}^- \cdot \nabla) \boldsymbol{u}^+ + \nabla \times (\boldsymbol{u}^+ \nabla \boldsymbol{u}^-) = 0, & t > 0, \ \theta \in \mathbb{T}^3, \\
\boldsymbol{u}^+ = \nabla \times (-\Delta)^{-1} \boldsymbol{\omega}^+, \qquad \boldsymbol{u}^- = \nabla \times (-\Delta)^{-1} \boldsymbol{\omega}^-, & t > 0, \ \theta \in \mathbb{T}^3,
\end{cases}
\end{equation}
where $(\boldsymbol{u}^- \nabla \boldsymbol{u}^+)_j = u^-_i \partial_j u^+_i$ and $(\boldsymbol{u}^+ \nabla \boldsymbol{u}^-)_j = u^+_i \partial_j u^-_i$ for $j \in \{1,2,3\}$. Owing to the differential identities
\begin{align*}
\nabla \times (\boldsymbol{u}^- \nabla \boldsymbol{u}^+) &= \tfrac{1}{2} \nabla \times (\boldsymbol{u}^- \nabla \boldsymbol{u}^+) - \tfrac{1}{2} \nabla \times (\boldsymbol{u}^+ \nabla \boldsymbol{u}^-), \\
\nabla \times (\boldsymbol{u}^+ \nabla \boldsymbol{u}^-) &= \tfrac{1}{2} \nabla \times (\boldsymbol{u}^+ \nabla \boldsymbol{u}^-) - \tfrac{1}{2} \nabla \times (\boldsymbol{u}^- \nabla \boldsymbol{u}^+),
\end{align*}
we can rewrite \cref{eq:VVMHD0} into
\begin{equation}\label{eq:VVMHD}
\left\{
\begin{aligned}
\partial_t \boldsymbol{\omega}^+ 
&+ (\boldsymbol{u}^- \cdot \nabla) \boldsymbol{\omega}^+ 
- (\boldsymbol{\omega}^+ \cdot \nabla) \boldsymbol{u}^- \\
&+ \tfrac{1}{2} \nabla \times (\boldsymbol{u}^- \nabla \boldsymbol{u}^+) 
- \tfrac{1}{2} \nabla \times (\boldsymbol{u}^+ \nabla \boldsymbol{u}^-) 
= 0, 
&& t > 0,\ \theta \in \mathbb{T}^3, \\
\partial_t \boldsymbol{\omega}^- 
&+ (\boldsymbol{u}^+ \cdot \nabla) \boldsymbol{\omega}^- 
- (\boldsymbol{\omega}^- \cdot \nabla) \boldsymbol{u}^+ \\
&+ \tfrac{1}{2} \nabla \times (\boldsymbol{u}^+ \nabla \boldsymbol{u}^-) 
- \tfrac{1}{2} \nabla \times (\boldsymbol{u}^- \nabla \boldsymbol{u}^+) 
= 0, 
&&t > 0,\ \theta \in \mathbb{T}^3, \\
\boldsymbol{u}^+ 
&= \nabla \times (-\Delta)^{-1} \boldsymbol{\omega}^+, \quad 
\boldsymbol{u}^- 
= \nabla \times (-\Delta)^{-1} \boldsymbol{\omega}^-, 
&& t > 0,\ \theta \in \mathbb{T}^3.
\end{aligned}
\right.
\end{equation}

As a one-dimensional model of \cref{eq:VVMHD}, the following system was proposed in \cite[Eq.~(1.11)]{dai2023}:
\begin{equation}\label{eq:MHD1-DVZ}
\begin{cases}
\partial_t \omega^+ + u^-\, \partial_\theta \omega^+ - \omega^+\, H \omega^- + \tfrac{1}{2} \omega^-\, H \omega^+ - \tfrac{1}{2} \omega^+\, H \omega^- = 0, & t > 0, \ \theta \in \mathbb{T}, \\
\partial_t \omega^- + u^+\, \partial_\theta \omega^- - \omega^-\, H \omega^+ + \tfrac{1}{2} \omega^+\, H \omega^- - \tfrac{1}{2} \omega^-\, H \omega^+ = 0, & t > 0, \ \theta \in \mathbb{T},\\ 
\partial_\theta u^\pm = H \omega^\pm & t > 0, \ \theta \in \mathbb{T}.
\end{cases}
\end{equation}
We note that \cref{eq:MHD1-DVZ} corresponds to \crefrange{eq:MHD1}{eq:MHDu} for the particular parameter choice $a = 1$, $p = \tfrac{3}{2}$, $q = -\tfrac{1}{2}$. More generally, one can introduce the parameters $a\in \R$ and $p-q=2$, with $p \ge 0$,  $- 1\le q \le 0$ to capture the anti-symmetry feature of \cref{eq:VVMHD}. 
On the other hand, starting from \cref{eq:VVMHD0}, the following system can be proposed (cf.~\cite[Eq.~(1.10)]{daiArXiv}, the ArXiv version of \cite{dai2023}): 
\begin{equation}\label{eq:MHD1-DVZ-arxiv}
\begin{cases}
\partial_t \omega^+ 
+ 2\, u^- \, \partial_\theta \omega^+
- \omega^+ \, H \omega^-
- \omega^- \, H \omega^+ = 0, & t > 0, \ \theta \in \mathbb{T},\\
\partial_t \omega^- 
+ 2\, u^+ \, \partial_\theta \omega^-
- \omega^- \, H \omega^+
- \omega^+ \, H \omega^- = 0, & t > 0, \ \theta \in \mathbb{T}, \\ 
\partial_\theta u^\pm = H \omega^\pm & t > 0, \ \theta \in \mathbb{T}, 
\end{cases}
\end{equation}
which corresponds to \crefrange{eq:MHD1}{eq:MHDu} for the particular parameter choice $a = 2$, $p = 1$, $q = 1$. More generally, one can introduce the parameters $a,\,p,\,q \in \mathbb R$, with $a= \frac{p}{2} = \frac{q}{2}$. 

We also mention that \textit{one-dimensional  electron-magnetohydrodynamics models} have been studied in \cite{zbMATH08048827,dai2025wellposednessblowup1delectron}.

In the absence of a magnetic field (that is, when $\boldsymbol{b} \equiv \boldsymbol{0}$, which is equivalent to \(\boldsymbol{u}^- \equiv \boldsymbol{u}^-\) and \(\boldsymbol{\omega}^+ \equiv \boldsymbol{\omega}^-\)), \cref{eq:MHD3D} reduces to the \emph{incompressible Euler system}:
\begin{equation}\label{eq:Euler}
\begin{cases}
\partial_t \boldsymbol{v} + (\boldsymbol{v} \cdot \nabla) \boldsymbol{v} + \nabla P = 0, & t > 0, \ \theta \in \mathbb{T}^3, \\
\nabla \cdot \boldsymbol{v} = 0, & t > 0, \ \theta \in \mathbb{T}^3.
\end{cases}
\end{equation}
The vorticity $\boldsymbol{\omega} \coloneqq \nabla \times \boldsymbol{v}$ of \cref{eq:Euler} satisfies
\begin{equation}\label{eq:EulerVorticity}
\begin{cases}
\partial_t \boldsymbol{\omega} + (\boldsymbol{v} \cdot \nabla) \boldsymbol{\omega} + (\boldsymbol{\omega} \cdot \nabla) \boldsymbol{v} = 0, & t > 0, \ \theta \in \mathbb{T}^3,\\
\boldsymbol{v} = \nabla \times (-\Delta)^{-1} \boldsymbol{\omega}, & t > 0, \ \theta \in \mathbb{T}^3.
\end{cases}
\end{equation}

Consequently, the one-dimensional MHD model in \crefrange{eq:MHD1}{eq:MHDu} is closely related to one-dimensional models of \cref{eq:EulerVorticity}. In particular, if \(\omega^+ \equiv \omega^-\), \(a,\,p,\,q \in \mathbb R\), and \(p+q=1\), \crefrange{eq:MHD1}{eq:MHDu} reduces to the \emph{Okamoto--Sakajo--Wunsch model} (also known as \emph{generalized Constantin--Lax--Majda model}) introduced in \cite{MR2439488}: 
\begin{align}\label{eq:OSW}
\partial_t \omega + a\, u\, \partial_\theta \omega &= \omega H \omega, &&  t > 0, \ \theta \in \mathbb{T},
\\
\label{eq:OSWu}  \partial_\theta u &= H \omega, \quad  \text{with} \quad u(t,0) \equiv 0, &&  t > 0, \ \theta \in \mathbb{T},
\end{align}
with a parameter \( a \in \mathbb{R} \) in front of the transport term $u \, \partial_\theta \omega$, to give it a different weight compared to the vortex stretching term $\omega H \omega$. Here,   \cref{eq:OSWu} is  a one-dimensional analogue of \emph{Biot--Savart's law} $\boldsymbol u = \nabla \times (-\Delta)^{-1} \boldsymbol{\omega}$.

When $a = 0$ (no transport effects), \crefrange{eq:OSW}{eq:OSWu} reduces to the \emph{Constantin--Lax--Majda model} (introduced in \cite{MR812343}):
\begin{align}\label{eq:CLM}
\partial_t \omega = \omega H \omega, \quad \text{where} \quad \partial_\theta u = H \omega, \quad u(t,0) \equiv 0, \qquad  t > 0, \ \theta \in \mathbb{T},
\end{align}

When $a = 1$ (equally weighted transport and stretching), \crefrange{eq:OSW}{eq:OSWu} reduces to the \emph{De Gregorio model} (introduced in \cite{DG1996,zbMATH04172622}):
\begin{align}\label{eq:DG1}
\partial_t \omega + u\, \partial_\theta \omega = \omega H \omega, \quad \text{where} \quad \partial_\theta u = H \omega, \quad u(t,0) \equiv 0, \qquad t > 0, \ \theta \in \mathbb{T}.
\end{align}

When $a = -1$, \crefrange{eq:OSW}{eq:OSWu} corresponds to the \emph{C\'ordoba--C\'ordoba--Fontelos model} introduced in \cite{MR2179734} for the \emph{2D quasi-geostrophic equation}:
\begin{align}\label{eq:CCF}
\partial_t \omega - u\, \partial_\theta \omega = \omega H \omega, \quad \text{where} \quad \partial_\theta u = H \omega, \quad u(t,0) \equiv 0, \qquad t > 0, \ \theta \in \mathbb{T}.
\end{align}

\subsubsection{Singularity formation vs global well-posedness for the Okamoto--Sakajo--Wunsch model}

These one-dimensional models are employed to investigate potential singularities in three-dimensional fluid equations, capturing the core mechanisms of 3D dynamics while remaining tractable for mathematical analysis.

For $a = 0$ (i.\,e., for the original Constantin--Lax--Majda model \cref{eq:CLM}), finite-time singularity was established already in \cite{MR812343} (cf.~also \cite{zbMATH08051718}). 

When $a < 0$, heuristically, the transport effect works together with the stretching effect to develop singularities. For the special case $a = -1$ (the C\'ordoba--C\'ordoba--Fontelos model \cref{eq:CCF}), finite-time singularity from smooth initial data was established in \cite{MR2179734}; the result was later extended to all $a < 0$ in \cite{MR2680191}. For a closely related nonlocal model, cusp formation was established in \cite{MR3621816}.

On the other hand, when $a > 0$, there are competing nonlocal stabilizing effects due to advection and destabilizing effects due to vortex stretching, which are of the same order in terms of scaling.\footnote{~The stabilizing effect of advection has also been studied in \cite{MR2388660} for an exact 1D model of the 3D axisymmetric Navier--Stokes equations along the symmetry axis, and in \cite{zbMATH05531021} for a 3D model of the axisymmetric Navier--Stokes equations.} 

For small positive $a$, it is expected that the vortex stretching term dominates the advection term. Indeed, in \cite{MR4065651}, an exact self-similar finite-time blow-up from smooth initial data was constructed for sufficiently small $|a|$. In \cite{MR4242826}, the stability of an approximate self-similar profile was established, along with the formation of asymptotically self-similar finite-time blow-up from smooth, compactly supported initial data. Similar results on the stability of the self-similar solutions constructed in \cite{MR4065651} were obtained in \cite{MR4259877}. For $a$ close to $1/2$, where the vortex stretching term is relatively stronger than the advection, self-similar blow-up was obtained in \cite{MR4300258,MR4284536}.

In \cite{MR4284536}, finite-time singularity formation for \crefrange{eq:OSW}{eq:OSWu} from smooth initial data on the torus was proven for $1 - \delta < a < 1$; on the other hand, for $1 < a < 1 + \delta$, for some $\delta > 0$, it was shown that the solution exists globally and satisfies the decay estimate
\(
\|\omega(t,\cdot)\|_{H^1} = {O}(t^{-1})\) as \( t \to +\infty\). In particular, when (formally) $a=\infty$ (i.\,e., the vortex stretching term $\omega H \omega$ is neglected and only transport effects remain),  global well-posedness of \crefrange{eq:OSW}{eq:OSWu} can be obtained using the conservation of $\|\omega\|_{L^\infty}$ (see, e.\,g., \cite{MR2439488}). 

For $a = 1$ (the De Gregorio model \cref{eq:DG1}), the analysis becomes more challenging, as advection and vortex stretching are comparable, and different behaviors arise on the real line and on the torus.\footnote{~For discussions of the Okamoto--Sakajo--Wunsch model \textit{with viscosity}, see, e.\,g., \cite{zbMATH07789595,MR840339,MR2865810,MR2397459,MR2662457,MR4105366}.} 

On the whole line $\mathbb{R}$, finite-time singularity was established in \cite{MR4242826}  for certain smooth odd initial data. Finite-time blow-up can also occur for initial data of lower regularity, for which the equations are nevertheless still locally well-posed, as shown in \cite{MR4065651}. Self-similar blow-ups for $a = 1$ were constructed in \cite{MR4242826,MR4630486} (more recently, the study of self-similar blow-ups for all $a \le 1$ was completed in \cite{MR4721709}).

On the {torus} $\mathbb{T}$, the nonlinear stability of the \emph{ground state}\footnote{~Using \cref{le:Heo}, we can compute that 
\[
\Omega_\kappa(\theta) \coloneqq -\sin(\kappa\,\theta), \quad U_\kappa(\theta) \coloneqq \frac{1}{\kappa}\,\sin(\kappa\,\theta),  \qquad \text{for $\kappa \ge 1$},
\]
are \emph{stationary solutions} (or \emph{steady states}) of \cref{eq:DG1}; see \cite[Lemma 2.7]{guo2025}. We call $(\Omega_1, U_1)$  \emph{ground state} and, if $\kappa \ge2$, we call $(\Omega_\kappa, U_\kappa)$ the \emph{($k-1$)-th excited state} of \cref{eq:DG1} on the torus $\mathbb T$. 

In what follows, we use similar terminology for the MHD model \crefrange{eq:MHD1}{eq:MHDu} as well.} 
\begin{align*}
\Omega_1^\pm(\theta) \coloneqq -\sin(\theta), \quad U_1^\pm(\theta) \coloneqq \sin(\theta),
\end{align*}
for \cref{eq:DG1} was proven in \cite{JSS2019} using spectral theory and complex-variable methods.\footnote{~This is consistent with numerical simulations suggesting that, for general smooth initial data, solutions of the De Gregorio model \cref{eq:DG1} converge to the ground states (see \cite[Section~4]{MR2439488}).} In \cite{LLR2020}, global well-posedness of solutions to \cref{eq:DG1} on the real line or the circle was established when the initial vorticity $\omega_0$  does not change sign and satisfies $|\omega_0|^{1/2} \in H^1(\mathbb{T})$; the analysis also revisits the proof of nonlinear stability of the ground state on the torus given in \cite{JSS2019}, and provides an alternative approach using a direct energy method in a suitably chosen energy space. In \cite{MR4875613}, finite time blow-up was established for a class of $C^\alpha$ initial data (for any $\alpha < 1$) that are odd and do not change sign only on one side of the origin.

\subsubsection{Stability for the one-dimensional MHD model}

Concerning the MHD model \crefrange{eq:MHD1}{eq:MHDu}, in \cite{dai2023}, \( H^1 \)-local-in-time well-posedness of the initial-value problem was established in the case \( p = 0 \), \( q = 1 \), and \( a \in \mathbb{R} \), along with a continuation criterion based on controlling \( \| H\omega^+ \|_{L^\infty} + \| H\omega^- \|_{L^\infty} \). In addition, global well-posedness was demonstrated in the case \( p = q = 0 \). An alternative proof of local existence for arbitrary \(a,\, p,\, q \in \mathbb{R}\), as well as a continuation criterion, is provided in \cite[Appendix~A]{sun2025}.

Moreover, the stability of the MHD model \crefrange{eq:MHD1}{eq:MHDu}, with parameters
\begin{align}\label{ass:parameters}
a = 1,\qquad p + q = 1,\qquad p,\,q\ge0,
\end{align}
i.\,e., in the setting that reduces to the De Gregorio model \cref{eq:DG1} when $\omega^+ \equiv \omega^-$, was studied in \cite{sun2025} using techniques similar to those developed in \cite{LLR2020} for \cref{eq:DG1}, yielding an analogous exponential stability result for the ground state. More precisely, in \cite[Theorem 1.1]{sun2025}, under the stronger assumption
\begin{align}\label{ass:q=0}
a = 1,\qquad p = 1,\qquad q = 0,
\end{align}
it is proven that for any mean‐zero \(H^2\) initial perturbation \(\omega_0^\pm\) sufficiently small in a suitable weighted Sobolev space (the one recalled in \cref{ssec:spaces}) around the \textit{ground state}
\[
\Omega_1^\pm (\theta) \coloneqq -\sin\theta, \qquad U_1^\pm(\theta) \coloneqq \sin(\theta),
\]
the corresponding solution \(\omega^\pm\) exists globally and converges exponentially fast back to \(\Omega_1\) in the same weighted Sobolev norm. In \cite[Theorem 1.2]{sun2025}, under the assumption
\begin{align}\label{ass:q_small}
a=1, \qquad p = 1 - q,\qquad 0 < q < \tfrac14
\end{align}
(a slight generalization of \cref{ass:q=0}, in which \(q\) is allowed to be non-zero but small), 
it is shown that any mean‐zero \(H^2\) perturbation yields a global solution which, up to a well-behaved time‐dependent modulation \(h=h(t)\), decays exponentially toward the ground state $\Omega_1$.

\subsection{Summary of the main results and outline of the paper}
\label{ssec:summary}

More recently, progress has been made in understanding the stability of \emph{excited states} for the De Gregorio model \cref{eq:DG1}:
\[
\Omega_\kappa(\theta) \coloneqq -\sin(\kappa\,\theta), \quad U_\kappa(\theta) \coloneqq \frac{1}{\kappa}\,\sin(\kappa\,\theta)  \qquad \text{for $\kappa \ge 2$.}
\]
For $\kappa=2$, numerical simulations mentioned in \cite{JSS2019} and \cite{LLR2020} highlight instability and, on the other hand, that perturbations of the type  $\sin(2\theta) + \varepsilon \cos \theta$ converge to some multiple of $- \sin (2\theta)$ instead of ground states. A rigorous result in this direction was recently established in \cite{guo2025}, where the authors analyzed the stability and instability of perturbations around the first excited state on the torus. More precisely, for $\kappa=2$, the linear and nonlinear instability are established for a broad class of initial data, while nonlinear stability is proved for another large class of initial data. 

We recall some key ideas of the strategy of \cite{guo2025} in later sections, as they are helpful for our study as well. In particular, we stress that one of the new elements in the instability analysis of \cite{guo2025} is the derivation of a second-order ordinary differential equation for the Fourier coefficients of solutions and the analysis of the spectral properties of a positive definite quadratic form arising from it. 

The goal of this paper is to extend the analysis of \cite{guo2025} to the one-dimensional MHD model given in \crefrange{eq:MHD1}{eq:MHDu}, with parameters as in \cref{ass:q=0}. In this paper,
we cannot address the general parameter range \cref{ass:parameters}, but some extensions to cases when $q \neq 0$ are also possible (see \cref{rk:qneq0} below).

We focus on the study of the stability and instability of  \crefrange{eq:MHD1}{eq:MHDu} near the first excited state 
\begin{align*}
\Omega_2^\pm(\theta) \coloneqq -\sin(2\,\theta), \quad U_2^\pm(\theta) \coloneqq \frac{1}{2}\,\sin(2\,\theta). 
\end{align*}
Moreover, the case of higher excited states (i.\,e., $\kappa \ge 3$) presents several further challenges, which are discussed in \cref{rk:arbitrarykappa} below. The paper is organized as follows.  

In \cref{sec:perturbations}, we introduce the notation and functional framework used throughout the paper. In particular, we write the equation satisfied by the perturbations $\eta^\pm$ of the first excited state: 
\begin{align*}
    \omega^{\pm}(t,\theta) &\coloneqq -\sin(2\theta) + \left(\eta^{+}(t,\theta) \pm \eta^{-}(t,\theta)\right), && t >0, \ \theta \in \T, \\
    u^{\pm}(t,\theta) &\coloneqq \frac{1}{2}\sin(2\theta) + \left(v^{+}(t,\theta) \pm v^{-}(t,\theta)\right), && t >0, \ \theta \in \T.
\end{align*}

We state the main results in \cref{sec:main}: global well-posedness for the linearized equations \crefrange{the:existence}{the:existence2}); instability for the linearized problem (\cref{the:instability_linearized}); local well-posedness for the nonlinear problem (\cref{le:nonlinear existence}); instability (\cref{the:instability_nonlinear}) or global well-posedness and stability (\cref{the:stability_nonlinear}) for the nonlinear problem. 

In \cref{sec:existence-linearized}, we prove \crefrange{the:existence}{the:existence2} by relying on Galerkin’s method using the orthonormal basis \(\{e_{2,k}\}_{k \geq 1}\) of the Hilbert space $\mathcal H_2$ introduced in \cref{ssec:spaces}. 

In \cref{sec:instability-linearized}, we prove \cref{the:instability_linearized}. The key observation is that the first linearized equation concerning \(\eta^{+}\) is responsible for the instability: indeed, the structure of the associated linear operator is the same as that of the De Gregorio model \cref{eq:DG1} studied in \cite{guo2025}. To avoid overestimation in the linear stability analysis, we expand the perturbation with respect to an orthonormal basis in the weighted \(L^{2}\)-norm defined in \cref{ssec:spaces}. 
	
In \cref{sec:stability-nonlinear}, we first establish the local-posedness of the nonlinear system: this is based on the linear analysis in \cref{sec:existence-linearized}, the main challenge being to obtain uniform estimates for the nonlinear terms. Moreover, since their derivatives exhibit stronger singularities, 
we are able to establish the nonlinear instability result (under the Lipschitz structure) in \cref{the:instability_nonlinear}.

In \cref{sec:instability-nonlinear}, we prove the nonlinear stability result \cref{the:stability_nonlinear}. In particular, we notice that the linearized equation for \(\eta^{+}\) exhibits exponential decay for initial data of the form \(\eta_{0}^{+} = \sum_{k\geq 1}\eta_{2,2k}^{+} (0) e_{2,2k}\), and the  equation for \(\eta^{-}\) also exhibits exponential decay for initial data \(\eta^{-}_{0} \in \mathcal{H}_{2}\). Combining these linear stability results with a continuity argument, we establish the global well-posedness of system \cref{eq:linearized}.

Finally, in \cref{app:tech}, we collect some preliminary lemmas used throughout the paper.

\section{Perturbations around the first excited state}
\label{sec:perturbations}

\subsection{Equations of the perturbations}
\label{ssec:perturbations}

In analyzing the behavior of \(\omega^{\pm}\) near the first excited state \(\Omega_2^\pm (\theta) \coloneqq -\sin(2\theta)\), we rewrite \crefrange{eq:MHD1}{eq:MHDu}, with parameters as in \cref{ass:parameters}, in terms the variables \(\eta^{\pm} (t,\theta)\) defined by
\begin{align*}
    \omega^{\pm}(t,\theta) &\coloneqq -\sin(2\theta) + \left(\eta^{+}(t,\theta) \pm \eta^{-}(t,\theta)\right), && t >0, \ \theta \in \T, \\
    u^{\pm}(t,\theta) &\coloneqq \frac{1}{2}\sin(2\theta) + \left(v^{+}(t,\theta) \pm v^{-}(t,\theta)\right), && t >0, \ \theta \in \T, 
\end{align*}
with \(v^{\pm}(t,0) = 0\). Then \crefrange{eq:MHD1}{eq:MHDu} can be rewritten as a partial differential equation in \(\eta^{\pm}\) as follows: 
\begin{align}\label{eq:perturbations}
\begin{cases}
\partial_t \eta^{+}  = \{\frac{1}{2} \eta^{+} + v^{+}, \sin(2\theta)\} +  N_{1}, & t >0, \ \theta \in \T, \\
\partial_t \eta^{-}  = \{\frac{1}{2} \eta^{-} - v^{-}, \sin(2\theta)\} - 2q\left(\sin(2\theta) H\eta^{-} + \cos(2\theta) \eta^{-}\right) +  N_{2}, & t >0, \ \theta \in \T,\\
\partial_{\theta}v^{\pm}  = H\eta^{\pm}, & t >0, \ \theta \in \T,
\end{cases} 
\end{align}
where \(N_{1}\) and \(N_{2}\) are nonlinear terms defined as
\begin{align}\label{def:nonlinear1}
     N_{1} &\coloneqq \{\eta^{+}, v^{+}\} - \{\eta^{-}, v^{-}\},\\
     N_{2} &\coloneqq \{\eta^{-}, v^{+}\} - \{\eta^{+}, v^{-}\} - 2q\left(\eta^{-} H \eta^{+} - \eta^{+} H \eta^{-}\right), \label{def:nonlinear2}
\end{align}
with \(\{\cdot, \cdot\}\) denoting the \emph{Lie bracket} \(\{f, g\} \coloneqq f\partial_{\theta}g - g\partial_{\theta}f\). 

To simplify the notation, we define the linear operators 
    \begin{align}\label{def:operator1}
        L f \coloneqq \{f, \sin(2\theta)\},\qquad A f \coloneqq \{v(f), \sin(2\theta)\},\qquad Q f \coloneqq \sin(2\theta) H\eta^{-} + \cos(2\theta) \eta^{-},
    \end{align}
    where \(v(f)\) satisfies \(\partial_{\theta} v = H f\) with \(v(0) = 0\), and let  
\begin{align} \label{def:operator2}
L^{+} \coloneqq L + A, \qquad L^{-} \coloneqq L - A - 2\,q\, Q.
\end{align}

With the notations in \crefrange{def:operator1}{def:operator2}, the system \crefrange{eq:perturbations}{def:nonlinear2} takes the form
    \begin{align}\label{eq:perturbations2}
        \begin{cases}
            \partial_{t}\eta^{+} = L^{+} \eta^{+} +  N_{1}, & t >0, \ \theta \in \T, \\
            \partial_{t}\eta^{-} = L^{-} \eta^{-} +  N_{2}, & t >0, \, \theta \in \T.
        \end{cases}
    \end{align}  

    The linear part of the one-dimensional MHD system in \cref{eq:perturbations2}, with the more restrictive parameter choice in \cref{ass:q=0},
    is given by
\begin{align}\label{eq:linearized}
\begin{cases}
    \partial_{t} \eta^{+} = L^{+} \eta^{+} = \left\{\frac{1}{2} \eta^{+} + v^{+}, \sin(2\theta)\right\}, & t >0, \ \theta \in \T, \\
    \partial_{t} \eta^{-} = L^{-} \eta^{-} = \left\{\frac{1}{2} \eta^{-} - v^{-}, \sin(2\theta)\right\}, & t >0, \ \theta \in \T.
\end{cases}
\end{align}

\subsection{Function spaces and notation}
\label{ssec:spaces}

We analyze the one-dimensional MHD in the following function space inspired by \cite{LLR2020,guo2025}:
\begin{align*}
    \mathcal{H}_{2} \coloneqq  \left\{\eta \in H^{1} (\mathbb{T}) : \ \eta \text{ is odd } \quad \text{and} \quad   \int_{\mathbb{T}} \frac{|\partial_{\theta}\eta|^{2}}{|\sin(\theta)|^{2}} \, \mathrm d \theta < \infty\right\},
\end{align*}
which is a Hilbert space with inner-product defined as
\begin{equation}
\langle \xi, \eta \rangle_{\mathcal H_2} \coloneqq \int_{-\pi}^{\pi} \rho_2 \, \partial_\theta \xi \, \partial_\theta \eta \, \mathrm d\theta, \qquad \text{with } \ \rho_2 \coloneqq \frac{1}{4\pi \sin^2 \theta}
\end{equation}
(see \cite[Eq.~(1.13)--(1.14)]{guo2025}).
The family \(\{e_{2,k}\}_{k \geq 1}\), with 
\begin{align*}
    e_{2,k} \coloneqq \frac{\sin((k+2)\theta)}{k+2} - \frac{\sin(k\theta)}{k}, \qquad k \geq 1,
\end{align*}
is a complete orthonormal basis for $\mathcal H_2$ (see \cite[Lemma 2.1]{guo2025}).

    The norm in the Sobolev space \(\mathcal H_2\) can thus be written as
	\begin{align*}
		\| \eta \|_{\mathcal H_2} \coloneqq \left( \int_{-\pi}^{\pi} \rho_2 \, | \partial_{\theta} \eta|^{2} \, \mathrm d\theta \right)^{1/2}  = \| \rho_2^{1/2}\, \partial_{\theta} \eta\|_{L^2}.
	\end{align*}

We denote
	\begin{align*}
		 L^{2} (\T) &\coloneqq \left\{ f: [-\pi,\pi] \to \R\,: \  f\in L^{2} (- \pi, \pi),  \ f \ \text{is periodic on}\ [- \pi, \pi]\right\},\\
		 H^{m} (\T) &\coloneqq \left\{f: [-\pi,\pi] \to \R\,: \ f^{(k)} \in L^{2} (- \pi, \pi), \ f^{(k)}\ \text{is periodic on}\ [- \pi, \pi],\ \text{for all}\ 0 \leq k \leq m \right\},
	\end{align*}
	where \(f^{(k)}\) denotes the \(k\)-th order derivative of \(f\).
	Throughout this paper, the norms \(\|\cdot\|_{L^{p}}\), \(\|\cdot\|_{H^{m}}\) and \(\|\cdot\|_{L^{\infty}}\) refers to those in \(L^{p} (\T)\), \(H^{m} (\T)\), and \(L^{\infty}(\T) \), respectively. We also denote by \(\left\langle \cdot, \cdot \right\rangle\) the standard inner product in \( L^{2} (\T) \), i.\,e.,
	\begin{align*}
		\left\langle f, g\right\rangle \coloneqq \int_{- \pi}^{\pi} f g \, \mathrm d\theta.
	\end{align*}

Finally, throughout the paper, we use \(C\) for an absolute constant and \(C(A, B, \ldots, Z)\) for constants depending on the parameters \(A, B, \ldots, Z\). These constants may vary from line to line unless otherwise specified. The notation \(A \lesssim B\) indicates that \(A \leq C B\) for some positive constant \(C\), and \(A \sim B\) means that both \(A \lesssim B\) and \(B \lesssim A\) hold, with implicit constants that may also vary from line to line.

\section{Main results}
\label{sec:main}

Our first result establishes the global existence and uniqueness of solutions to \cref{eq:linearized} with odd initial data.

\begin{theorem}[Global existence of solutions to the linearized problem] \label{the:existence}
    Let \(\rho_2^{1/2} \partial_{\theta} \eta^{\pm}_{0} \in H^{m}(\mathbb T)\) for some integer \(m > 3\). Then, for any \(T > 0\), there exists a unique classical odd solution \(\eta^{\pm} \) to \cref{eq:linearized} with initial condition \(\eta^{\pm} (0, \cdot) = \eta^{\pm}_{0} (\cdot)\), satisfying 
    \begin{align*}
        \rho_2^{1/2}\, \partial_{\theta} \eta^{\pm} (t,\theta) \in C([0, T]; H^{m}(\T)) \cap C^{2}([0, T]; H^{m - 2}(\T)).
    \end{align*}
\end{theorem}
To prove \cref{the:existence},  we employ Galerkin's method through the basis \( \{ e_{2,k}\}_{k \geq 1} \) (introduced in \cref{ssec:spaces}) and construct the approximate solutions 
\[ \eta_{n}^{\pm} (t, \theta) = \sum_{k=1}^{n} \eta_{2,k}^{\pm} (t) e_{2,k}. \]
During the proof, we also obtain uniform estimates for the Fourier coefficients of the solutions as a byproduct (see \crefrange{eq:decay-coeff}{eq:decay-coeff-ddt}), which guarantee the rigor of our subsequent arguments. 

For more general initial data, we can construct approximate solutions using the general basis \(\{\sin(k\theta)\}_{k \geq 1} \cup \{\cos(k\theta)\}_{k \geq 0}\). We can also establish the following result (whose proof is analogous to that of \cref{the:existence} and is therefore omitted).

\begin{theorem}[Global well-posedness of the  linearized problem] \label{the:existence2}
Let us assume that \( \eta_{0}^{\pm} \in H^{m}\) with \(m>3\)  an integer.  Then, for any \(T > 0\), there exits a unique classical solution \( \eta^{\pm} \) to \cref{eq:linearized} with initial condition \(\eta^{\pm}(0,\cdot) = \eta_{0}^{\pm} (\cdot) \), satisfying 
		\begin{align*}
			\eta^{\pm} (t,\theta) \in C([0,T];H^{m})\cap C^{2}([0,T];H^{m-2}). 
		\end{align*}
\end{theorem} 

Having established the global existence result, we turn to the study of the  instability of the linearized model \cref{eq:linearized}. More precisely, we show that  \(\| \eta^{+} (t) \|_{\mathcal{H}_{2}} + \| \eta^{-} (t) \|_{\mathcal{H}_{2}}\) grows exponentially for a suitable class of initial data.

\begin{theorem}[Instability for the linearized problem]\label{the:instability_linearized}
    Let us assume that the conditions in \cref{the:existence} are satisfied and that the initial data satisfies 
    \[\|\eta^{+}_{0}\|_{\mathcal{H}_{2}}\neq 0 \qquad  \text{ and } \qquad  \left\langle L^{+} \eta_{0}^{+}, \eta_{0}^{+} \right\rangle_{\mathcal{H}_{2}} \geq 0.
    \]
    Then there exists an absolute constants \( \lambda_{1} > 0 \) such that the corresponding solution of \cref{eq:linearized} given in \cref{the:existence} satisfies the following lower bound:
    \begin{align}\label{eq:est32}
        \| \eta^{+} (t) \|_{\mathcal{H}_{2}} + \| \eta^{-} (t) \|_{\mathcal{H}_{2}} > E_{1}^{1/2} (t) > 0, \qquad \text{for all $t >0$},
    \end{align}
 where
    \begin{align}\label{eq:est32b}
    \begin{aligned}
        E_{1} (t) &\coloneqq \frac{ \left\langle \eta_{0}^{+}, \eta_{0}^{+} \right\rangle_{\mathcal{H}_{2}} + \tfrac{1}{\sqrt{ \lambda_{1}}} \left\langle L^{+} \eta_{0}^{+}, \eta_{0}^{+} \right\rangle_{\mathcal{H}_{2}} }{2} \,e^{2 \sqrt{ \lambda_{1} } t } \\ &\qquad + \frac{\left\langle \eta_{0}^{+}, \eta_{0}^{+} \right\rangle_{\mathcal{H}_{2}} - \tfrac{1}{\sqrt{ \lambda_{1}}} \left\langle L^{+} \eta_{0}^{+}, \eta_{0}^{+} \right\rangle_{\mathcal{H}_{2}}}{2} \,e^{-{2 \sqrt{\lambda_{1} }t}}.
    \end{aligned}
    \end{align}
\end{theorem}

\begin{remark}[The constant $\lambda_1$ in \cref{the:instability_linearized}]\label{re:instability_linearized_remark}
    The constant \( \lambda_{1} \) in the statement of \cref{the:instability_linearized} is a positive and absolute constant satisfying \( \frac{1}{50}< \lambda_{1} < \frac{3}{5} \); see \cref{le:lambda}  for more details.
\end{remark}

\begin{remark}[Initial data satisfying the assumptions of \cref{the:instability_linearized}]
    There exist various initial data satisfying the assumptions in \cref{the:instability_linearized}. Let us suppose that \( \eta_{0}^{-}  \) fulfills \( \rho_2^{1/2}\, \partial_{\theta} \eta_{0}^{-} \in H^{m}(\T) \) for some integer \( m > 3\) and that either of the following conditions holds:
\begin{enumerate}[label=(\roman*)]
            \item \(\eta_{0}^{+} (\theta) = \eta_{2,1}^{+} (0) e_{2,1}\) with \(a_{1} \neq 0\);
            \item \(\eta_{0}^{+} (\theta) = \eta_{2,1}^{+} (0) e_{2,1} + \eta_{2,k}^{+} (0) e_{2,k} \) for some integer \( k \geq 2\), and
    \begin{align*}
        0 < \big(\eta_{2,k}^{+} (0)\big)^{2} \leq \frac{11}{18( d_{2,k+2}^{+} - d_{2,k}^{+} )} \big(\eta_{2,1}^{+} (0)\big)^{2}.
    \end{align*}
\end{enumerate}
Then the solution obtained in \cref{the:existence} satisfies \cref{eq:est32}.
\end{remark}

Let us briefly comment on the strategy of the proof of \cref{the:instability_linearized}. We start by expanding \(\eta^{\pm} (t,\theta)\) in orthonormal basis \(\{e_{2,k}\}_{k\ge 1}\) of the space \(\mathcal{H}_{2}\). Using this basis, \(\left\langle L^{\pm} \eta^{\pm}, \eta^{\pm}\right\rangle_{\mathcal{H}_{2}}\) formally reduces to a sum involving only squared coefficients
\begin{align*}
    \frac{1}{2} \frac{\mathrm{d}}{\mathrm{d} t} \|\eta^{\pm} (t)\|_{\mathcal{H}_{2}}^{2} = \left\langle L^{\pm} \eta^{\pm}, \eta^{\pm}\right\rangle_{\mathcal{H}_{2}} = \sum_{k\geq 1} c_{k}^{\pm} \left(\eta^{\pm}_{2,k} (t)\right)^{2}
\end{align*}
without any cross product terms between the coefficients (e.\,g., \(\eta^{\pm}_{2,k} (t) \eta^{\pm}_{2,k+2} (t)\)), where \(c_{k}^{\pm}\) are constants depending on \(k\). 

If there exists a constant \(C > 0\), such that \(c_{k}^{\pm} \leq - C\) for all \(k \geq 1\), then we obtain 
\begin{align*}
    \frac{1}{2} \frac{\mathrm{d}}{\mathrm{d} t} \|\eta^{\pm} (t)\|_{\mathcal{H}_{2}}^{2} \leq - C \|\eta^{\pm} (t)\|_{\mathcal{H}_{2}}^{2}.
\end{align*}
Regarding \(L^{-}\), there exists \(C > 0\) (see \cref{ssec:lineareta-}) such that \(c^{-}_{k} < - C\) for all \(k \geq 1\). In contrast, for \(L^{+} \), the coefficients \(c_{k}^{+}\) change signs (see \cref{ssec:lineareta+}): more specifically, a computation yields
\begin{align}\label{eq:ODE eta positive}
    \frac{1}{2} \frac{\mathrm{d}}{\mathrm{d} t} \sum_{k \geq 1} \left(\eta_{2,k}^{+} (t)\right)^{2} = \left\langle L^{+} \eta^{+}, \eta^{+} \right\rangle_{\mathcal{H}_{2}} = \sum_{k\geq 1} \left(d_{2,k}^{+} - d_{2,k+2}^{+} \right) \left( \eta^{+}_{2,k} (t)\right)^{2},
\end{align}
where
\begin{align*}
    d_{2,k}^{+} - d_{2,k+2}^{+} = - \frac{1}{2} - \frac{2 k^{2} - 4 k - 8}{(k + 2)^{2} k^{2}}, \qquad \text{for } k = 1,\, 2, \ldots.
\end{align*}
For \(k = 1\), we have \( d_{2,1}^{+} - d_{2,3}^{+} = \frac{11}{18} > 0\) and, for \(k \geq 2\), we have \( d_{2,k}^{+} - d_{2,k+2}^{+} < 0\). This suggests that the  linearized equation concerning \(\eta^{+}\) is responsible for the instability. To actually show that this is the case, we study a second-order ordinary differential equation (ODE) derived from \cref{eq:ODE eta positive}, which can be written as
\begin{align*}
        \frac12\frac{\mathrm{d}^2}{\mathrm{d} t^2}\sum_{k \geq 1} \left(\eta_{2,k}^{+} (t)\right)^{2} = \sum_{k\geq 1}  \frac{\mathrm{d}}{\mathrm{d} t} \left(d_{2,k}^{+} - d_{2,k+2}^{+} \right) \left( \eta^{+}_{2,k} (t)\right)^{2}.
\end{align*}
Similarly to the approach in \cite{guo2025}, we consider the following splitting:
\begin{align}\label{eq:part sum}
      \begin{aligned}
          S_{n}^{+} (t) & \coloneqq \sum_{k=1}^{n} \frac{\mathrm d}{\mathrm dt} \left( d_{2,k}^{+} - d_{2,k+2}^{+} \right) \left(\eta_{2,k}^{+} (t) \right)^{2}\\
      & = \left( d_{2,1}^{+} - d_{2,3}^{+} \right)^{2} \left( \eta_{2, 1}^{+} (t) \right)^{2}  + \left( d_{2,2}^{+} - d_{2,4}^{+} \right)^{2} \left(\eta_{2, 2}^{+} (t) \right)^{2} + \sum_{k=1}^{n-2} Q_{k} (t)  \\
      & \quad + \left( d_{2,n-1}^{+} - d_{2,n+1}^{+} \right)^{2} \left( \eta_{2, n-1}^{+} (t) \right)^{2}  + \left( d_{2,n} - d_{2,n+2}^{+} \right)^{2} \left( \eta_{2,n}^{+} (t) \right)^{2} \\
      & \quad + R_{n-1} (t) + R_{n} (t),
      \end{aligned}
  \end{align}
  where, the quadratic form \( Q_{k} (t) \) is defined as
  \begin{align*}
      Q_{k}(t) & \coloneqq \left( d_{2,k}^{+} - d_{2,k+2}^{+} \right)^{2} \left( \eta_{2, k}^{+} (t) \right)^{2} + \left( d_{2,k+2}^{+} - d_{2,k+4}^{+} \right)^{2} \left( \eta_{2, k+2}^{+} (t) \right)^{2}\\
      & \quad + 2 \left( d_{2,k}^{+} d_{2,k+2}^{+} + d_{2,k+2}^{+} d_{2,k+4}^{+} - 2 (d_{2,k+2}^{+})^{2} \right ) \eta_{2, k}^{+} (t) \eta_{2, k+2}^{+} (t), \quad \text{for \( k = 1,\, 2, \ldots,\, n - 2\),}
  \end{align*}
   and remainder term \( R_{k} (t) \) is given by
  \begin{align*}
      R_{k}(t) \coloneqq 2 d_{2,k+2}^{+} \left( d_{2,k}^{+} - d_{2,k+2}^{+} \right) \eta_{2,k}^{+} (t) \eta_{2, k+2}^{+} (t), \quad \text{for \( k = n - 1,\, n\).}
  \end{align*}
  Each quadratic form \( Q_{k} (t)\) is positive definite and admits uniform lower and upper bounds independent of \(k\ge 1\); moreover, by applying decay estimates to the remainder terms and letting \(n\to \infty\) in \cref{eq:part sum}, we obtain
	\begin{align}\label{eq:soo}
		4 \lambda_{1} \|\eta^{+} (t)\|_{\mathcal{H}_{2}}^{2} < \frac{\mathrm{d}^{2}}{\mathrm{d} t^{2}} \|\eta^{+} (t) \|_{\mathcal{H}_{2}}^{2} < 4 \lambda_{2} \|\eta^{+}(t)\|_{\mathcal{H}_{2}}^{2}, \qquad \text{for \( t >0\),}
	\end{align}
	 where \(\lambda_{i} > 0\) (with \(i = 1,\, 2\)) are absolute constants. Thus, the proof of \cref{the:instability_linearized} follows by applying a suitable comparison theorem  (see \cref{le:comparison the}) to the second-order ODE \cref{eq:soo}.
    
We now turn to the nonlinear problem. First, we establish the local well-posedness of the classical solution to \cref{eq:perturbations2} with odd initial data. 

\begin{theorem}[Local well-posedness of the nonlinear problem] \label{le:nonlinear existence}
    Let \(m > 3\) be an integer, and assume that $\eta^\pm_0 : \mathbb T \to \mathbb R$ satisfy \( I_{m}^{2} (0) \coloneqq \| \rho_2^{1/2} \partial_{\theta} \eta^{+}_{0} \|_{H^{m}}^{2} +  \| \rho_2^{1/2} \partial_{\theta} \eta^{-}_{0} \|_{H^{m}}^{2} = \delta^{2} \) for some $0 < \delta < +\infty$. Then there exist \(T_{0} > 0\) and unique classical odd solutions \(\eta^{\pm}\) to \cref{eq:perturbations2} with initial condition \(\eta^{\pm} (0, \cdot) = \eta_{0}^{\pm}\), such that 
    \begin{align*}
        \rho_2^{1/2} \partial_{\theta} \eta^{\pm} \in C ([0,T_{0}]; H^{m} (\T) ) \cap C^{2}([0,T_{0}]; H^{m-2} (\T)).
    \end{align*}
    Moreover, define \(u^{\pm} (t,\theta) \coloneqq \rho_2^{1/2} \partial_{\theta} \eta^{\pm} (t,\theta)\). Then the solutions satisfy 
    \begin{align*}
         \sup_{0 \leq t \leq T} \| u^{\pm} (t)\|_{H^{m}}, \ \sup_{0 \leq t \leq T} \| \partial_{t} u^{\pm} (t)\|_{H^{m}}, \ \sup_{0 \leq t \leq T} \| \partial_{t}^{2} u^{\pm} (t)\|_{H^{m}} \leq C (T_{0}) \delta,
     \end{align*}
     where the existence time \(T_{0}\) can be chosen as
     \[T_{0} = \frac{1}{C_{{1}}} \ln\left(1 + \frac{C_{{1}}}{2C_{{2}} I_{m} (0)}\right),\]
     with \(C_{1}, C_{2} > 0\) denoting absolute constants (cf.~\cref{eq:C12}).
\end{theorem}
The proof of \cref{le:nonlinear existence} is based on the linear analysis of \cref{the:existence}, the main challenge being to obtain suitable estimates for the nonlinear terms.

We then establish a nonlinear instability result. By \cref{the:instability_linearized}, the linearized evolution of \(\eta^{+}\)  exhibits instability in \(\mathcal{H}_{2}\), in the sense \(\|\eta^{+} (t)\|_{\mathcal{H}_{2}} > E_{1}^{1/2} (t)\) for \( t >0\). To analyze the corresponding nonlinear system and prove the following instability result, we derive some energy estimates for the system \cref{eq:perturbations2}, which enable us to rigorously pass the limit in a suitable scaled formulation and recover the linearized structure for a suitable class of initial data. 

\begin{theorem}[Instability for the nonlinear problem]\label{the:instability_nonlinear}
    Let us consider an integer \(m > 3\) and real constants \(\delta > 0\) and \(K > 0\). Let \(F : [0, + \infty) \to \mathbb{R}\) be a function satisfying 
    \begin{align}\label{eq:def_F}
        F(y) \leq K y, \quad  \text{for all} \quad  y\in[0,\infty).
    \end{align}
    There exist initial values \(\eta_{0}^{\pm}\) satisfying
    \begin{align*}
         I_{m} (0) \coloneqq \sqrt{\| \rho_2^{1/2} \partial_{\theta} \eta^{+}_{0} \|_{H^{m}}^{2} +  \| \rho_2^{1/2} \partial_{\theta} ,\eta^{-}_{0} \|_{H^{m}}^{2}} < \delta\qquad \text{and} \qquad 0 \leq \left\langle L_{1}^{+} u_{0}^{+}, u_{0}^{+} \right\rangle \leq \sqrt{\lambda_{1}} \|u_{0}^{+}\|_{L^{2}},
    \end{align*}
    where \( u_{0}^{+} \coloneqq \rho_2^{1/2} \partial_{\theta} \eta_{0}^{+}\), \( L_{1}^{+} u_{0}^{+} \coloneqq \rho_2^{1/2} \partial_{\theta} L^{+} \eta_{0}^{+} \), and \(\lambda_{1}\) is the same as in \cref{re:instability_linearized_remark}. 
    
    Then the corresponding solution \(\eta^{+}\) constructed in \cref{le:nonlinear existence} satisfies 
    \begin{align*}
        \| u^{+} (t_{K})\|_{L^{2}} > F(I_{m} (0)), \qquad \text{for some \(t_{K} \in (0,T_{0})\),}
    \end{align*}
     where \(u^{+} (t,\theta) \coloneqq \rho_2^{1/2} \partial_{\theta} \eta^{+} (t,\theta)\) and \( T_{0}\) is same as in \cref{le:nonlinear existence}. 
\end{theorem}

As a consequence of \cref{the:instability_nonlinear}, we say that the system \cref{eq:perturbations2} is \textit{unstable under the Lipschitz structure}. This terminology, motivated by the Lipschitz condition \cref{eq:def_F}, originates from  \cite{MR2963789} (where this type of instability was first discussed for compressible inviscid fluids) and \cite{JJN2013} (for inhomogeneous incompressible viscous fluids).

\begin{remark}[Lack of stability]
    From \cref{the:instability_nonlinear}, we deduce that \cref{eq:perturbations2} \emph{does not} possess the following property: there exists a constant $C>0$ such that
    \begin{align}\label{eq:Lip instability}
        \sup_{0\leq t\leq T} \|u^{+} (t) \|_{L^{2}}\leq C\|u_{0}^{+}\|_{H^{m}}, \qquad \text{for any \(T > 0\)}.
    \end{align}
\end{remark}

\begin{remark}[Initial data satisfying the assumptions of \cref{the:instability_nonlinear}]
    \cref{the:instability_nonlinear} holds for a broad class of initial data. As a simple example, consider initial data of the following form. Assume that \(\rho_2^{1/2} \partial_{\theta} \eta_{0}^{-} \in H^{m}\) for any integer \(m > 3\), and \(\eta_{0}^{+}\) is of the form
    \begin{align*}
        \eta_{0}^{+} (\theta) = \eta_{2,1}^{+} (0) e_{2,1} + \eta_{2,k}^{+} (0) e_{2,k},
    \end{align*}
    where \(\eta_{2,1}^{+} (0),\eta_{2,k}^{+} (0) \neq 0\) for some \(k\geq 2\). It is straightforward to verify that \(\eta_{0}^{+} \in \mathcal{H}_{2}\) and \(\rho_2^{1/2} \partial_{\theta} \eta_{0}^{+}\in H^{m}\) for \(m > 3\). If the coefficients of \(\eta_{0}^{+}\) satisfy
    \begin{align*}
        \frac{\frac{11}{18}-\sqrt{\lambda_{1}}}{\sqrt{\lambda_{1}} + (d_{2, k+2}^{+} - d_{2,k}^{+})} (\eta_{2,1}^{+} (0) )^{2} \leq (\eta_{2,k}^{+} (0) )^{2} \leq \frac{11}{18 (d_{2,k+2}^{+} - d_{2,k}^{+})} (\eta_{2,1}^{+} (0) )^{2},
		\end{align*}
        then it follows that
		\begin{align*}
			0 \leq \left\langle L^{+} \eta_{0}^{+}, \eta_{0}^{+} \right\rangle_{\mathcal{H}_{2}} \leq \sqrt{\lambda_{1}} \left\langle\eta_{0}^{+}, \eta_{0}^{+} \right\rangle_{\mathcal{H}_{2}}.
		\end{align*}
        There, \(\lambda_1\) is an absolute constant satisfying \( \frac{1}{50} < \lambda_1 < \frac{3}{5}\), and \(d_{2,k}^{+} - d_{2,k+2}^{+} < 0\) for all \(k \geq 2\); see \cref{le:lambda}  for more details.
\end{remark}

Finally, we establish the nonlinear stability result \cref{the:stability_nonlinear} for initial data \(\eta^{+}_{0}\) of the form \(\eta^{+}_{0} (\theta) = \sum_{k\geq 1} \eta_{2, 2k}^{+} (0) e_{2, 2k}\) (i.\,e., initial data such that only the even coefficients in the Fourier expansion with respect to the basis $\{e_{2,k}\}_{k \geq 1}$ are present). With this choice of initial data, the linearized system is stable in the space \(\mathcal{H}_{2}\); combined with a continuity argument and delicate estimates, this yields the nonlinear stability of \cref{eq:perturbations} in \(\mathcal{H}_{2}\).

\begin{theorem}[Global well-posedness and stability for the nonlinear problem]\label{the:stability_nonlinear}
    Let us suppose the initial value \(\eta_{0}^{\pm}\) satisfies \(\eta_{0}^{+} (\theta) = \sum_{k\geq 1} \eta_{2, 2k} (0) e_{2, 2k}\) and \(\eta_{0}^{-} \in \mathcal{H}_{2}\).  Let us define
    \begin{align*}
        I_{0}^{2} (t) \coloneqq \| \eta^{+} (t)\|_{\mathcal{H}_{2}}^{2} + \| \eta^{-} (t)\|_{\mathcal{H}_{2}}^{2}, \qquad \text{for $t \ge 0$}.
    \end{align*}
    Then there exists a constant \( \varepsilon > 0\) such that if \( I_{0} (0) \leq \varepsilon\), then the  system \cref{eq:linearized} with initial value \(\eta_{0}^{\pm}\) is globally well-posed. Moreover, the following estimate holds:
    \begin{align*}
        I_{0} (t) \lesssim e^{- \frac{1}{2} t} I_{0} (0), \qquad \text{for all \( t \geq 0\).}
    \end{align*}
\end{theorem}

The proof of \cref{the:stability_nonlinear} relies on the linear stability analysis, delicate estimates on the nonlinear terms, and a continuity argument.

\section{Global existence and uniqueness for the linearized problem}
\label{sec:existence-linearized}

In this section, we prove \cref{the:existence}, establishing the global existence and uniqueness of solutions to the system \cref{eq:linearized} using Galerkin's method. To obtain the decay estimates for the coefficients \(\eta_{2,k}^{\pm}\), \(\frac{\mathrm d}{\mathrm d t} \eta_{2,k}^{\pm}\), and \(\frac{\mathrm d^{2}}{\mathrm d t^{2}} \eta_{2,k}^{\pm}\), as required in \cref{sec:instability-linearized}, we construct approximate solutions using the basis \(\{e_{2,k}\}_{k \geq 1}\).

\begin{proof}[Proof of \cref{the:existence}]
    \uline{Step 1.} \emph{Construction of approximate solutions.} Let us fix a positive integer \(n\) and define
    \begin{align}
        & \eta_{n}^{\pm} (t, \theta) = \sum_{k = 1}^{n}\eta_{2, k}^{\pm} (t) e_{2, k},\label{eq:etan expression}\\
        & \partial_{\theta} v_{n}^{\pm} (t, \theta) = H \eta_{n}^{\pm} (t, \theta).
    \end{align}
    We aim to determine the coefficients \(\eta_{2, k}^{\pm} (t)\) such that
    \begin{align}
       \left\langle \partial_{t} \partial_{\theta} \eta_{n}^{+}, \partial_{\theta} e_{2, k} \rho_2 \right\rangle - \left\langle \partial_{\theta} L^{+} \eta_{n}^{+}, \partial_{\theta} e_{2, k} \rho_2 \right\rangle &= 0,\label{eq:etak positive}\\
        \left\langle \partial_{t} \partial_{\theta} \eta_{n}^{-}, \partial_{\theta} e_{2, k} \rho_2 \right\rangle - \left\langle \partial_{\theta} L^{-} \eta_{n}^{-}, \partial_{\theta} e_{2, k} \rho_2 \right\rangle &= 0,\label{eq:etak negative} \\ 
         \left\langle \partial_{\theta} \eta_{0}^{+}, \partial_{\theta} e_{2, k} \right\rangle &= \eta_{2, k}^{+} (0),\label{eq:etan positive initial data}\\
         \left\langle \partial_{\theta} \eta_{0}^{-}, \partial_{\theta} e_{2, k} \right\rangle &= \eta_{2, k}^{-} (0)\label{eq:etan negative initial data}
    \end{align}
    for \(0 \leq t \leq T\), and \(k = 1, 2, \ldots, n\). 
    
    Equations \crefrange{eq:etak positive}{eq:etak negative} can be rewritten as
\begin{align*}
    & \left\langle \partial_{t} \sum_{k = 1}^{n} \eta_{2, l}^{+} (t) \sin((l+1) \theta), \sin((k+1)\theta) \right\rangle - \left\langle -\frac{1}{2 \sin(\theta)} \partial_{\theta} L^{+} \eta_{n}^{+}, \sin((k+1)\theta) \right\rangle = 0,\\
    & \left\langle \partial_{t} \sum_{k = 1}^{n} \eta_{2, l}^{-} (t) \sin((l+1) \theta), \sin((k+1)\theta) \right\rangle - \left\langle -\frac{1}{2 \sin(\theta)} \partial_{\theta} L^{-} \eta_{n}^{-}, \sin((k+1)\theta) \right\rangle = 0,
\end{align*}
    where 
    \begin{align*}
       & -\frac{1}{2 \sin(\theta)} \partial_{\theta} L^{+} \eta_{n}^{+} = \frac{1}{2} \cos(\theta) \partial_{\theta}^{2} \eta_{n}^{+} + 2\cos(\theta) \eta_{n}^{+} + \cos(\theta) H (\partial_{\theta} \eta_{n}^{+}) + 4\cos(\theta) v_{n}^{+},\\
       & -\frac{1}{2 \sin(\theta)} \partial_{\theta} L^{-} \eta_{n}^{-} = \frac{1}{2} \cos(\theta) \partial_{\theta}^{2} \eta_{n}^{-} + 2 \cos(\theta) \eta_{n}^{-} - \cos(\theta) H (\partial_{\theta} \eta_{n}^{-}) - 4 \cos(\theta) v_{n}^{-}.
    \end{align*}
    To simplify the notation, we set \(u_{n}^{\pm} \coloneqq - \sqrt{\pi} \rho_2^{1/2} \partial_{\theta} \eta_{n}^{\pm}\) and define
    \begin{align*}
        L_{1}^{+} (u_{n}^{+}) &\coloneqq - \sqrt{\pi} \rho_2^{1/2} \partial_{\theta} L^{+} \eta_{n}^{+}\\
        & = - \frac{1}{2} \sin(2\theta) \partial_{\theta} u_{n}^{+} - \cos^{2}(\theta) u_{n}^{+} - 2 \cos(\theta) H (\sin(\theta) u_{n}^{+}) + 2 \cos(\theta) \eta_{n}^{+} + 4 \cos(\theta) v_{n}^{+}, \\
        L_{1}^{-} (u_{n}^{-}) & \coloneqq - \sqrt{\pi} \rho_2^{1/2} \partial_{\theta} L^{-} \eta_{n}^{-}\\
        & = - \frac{1}{2} \sin(2\theta) \partial_{\theta} u_{n}^{-} - \cos^{2}(\theta) u_{n}^{-} + 2 \cos(\theta) H (\sin(\theta) u_{n}^{-}) + 2 \cos(\theta) \eta_{n}^{-} - 4 \cos(\theta) v_{n}^{-}.
    \end{align*}
    Then ordinary differential equations \crefrange{eq:etak positive}{eq:etak negative} are equivalent to 
    \begin{align}
       & \big\langle \partial_{t} u_{n}^{+}, \sin((k+1)\theta) \big\rangle - \big\langle L_{1}^{+} (u_{n}^{+}), \sin((k+1)\theta) \big\rangle = 0,\label{eq:un positive}\\
       & \big\langle \partial_{t} u_{n}^{-}, \sin((k+1)\theta) \big\rangle - \big\langle L_{1}^{-} (u_{n}^{-}), \sin((k+1)\theta) \big\rangle = 0.\label{eq:un negative}
    \end{align}
    We may verify that \crefrange{eq:un positive}{eq:un negative} hold for \(k = 0\), using the convention \(\eta_{2,0}^{\pm} (t) = 0\). The standard existence theory for ordinary differential equations then guarantees the existence of a unique sequence of absolutely continuous functions \(\eta_{2,k}^{\pm} (t)\), with \(k = 1,\, 2,\, \ldots,\, n\)) satisfying \crefrange{eq:etak positive}{eq:etak negative} and the initial conditions \crefrange{eq:etan positive initial data}{eq:etan negative initial data} in \(0 \leq t \leq T\).

    \uline{Step 2.} \emph{Energy estimates.} We need several estimates on $u_n^\pm$ and their derivatives, which are proven below.
    
    \uline{Step 2\,a.} \emph{\(L^{2}\)-estimates of \(u_{n}^{\pm}\).} Multiplying \cref{eq:un positive} by \(\eta_{2,k}^{+} (t)\) and \cref{eq:un negative} by \(\eta_{2,k}^{-} (t)\), and then summing up over \(k = 1, 2, \ldots, n\), we obtain
    \begin{align*}
        \frac{1}{2} \frac{\mathrm d}{\mathrm dt} \left( \| u_{n}^{+} \|_{L^{2}}^{2} + \| u_{n}^{-} \|_{L^{2}}^{2} \right)  = \left\langle L_{1}^{+} (u_{n}^{+}), u_{n}^{+} \right\rangle + \left\langle L_{1}^{-} (u_{n}^{-}), u_{n}^{-} \right\rangle.
    \end{align*}
    Applying H\"{o}lder's inequality and \cref{le:hilbert estimate}, we obtain
    \begin{align*}
        \left\langle 2\cos(\theta) H ( \sin(\theta) u_{n}^{\pm}), u_{n}^{\pm} \right\rangle & \lesssim \ \| H( \sin(\theta) u_{n}^{\pm} )\|_{L^{2}} \| u_{n}^{\pm} \|_{L^{2}}\\ 
        &\lesssim\ \| u_{n}^{\pm} \|_{L^{2}}^{2}.
    \end{align*}
    Since \( \eta_{n}^{\pm} (0) = 0 \), we obtain
	\begin{align*}
	    \left\langle 2 \cos(\theta) \eta_{n}^{\pm}, u_{n}^{\pm} \right\rangle & \lesssim \| \eta_{n}^{\pm} \|_{L^{\infty}} \| u_{n}^{\pm} \|_{L^{2}}\\
		& \lesssim \| \partial_{\theta} \eta_{n}^{\pm} \|_{L^{2}} \| u_{n}^{\pm} \|_{L^{2}}\\
		&\lesssim\ \| u_{n}^{\pm} \|_{L^{2}}^{2}.
	\end{align*}
    Since \(\eta_{n}^{\pm}\) and \(v_{n}^{\pm}\) are odd functions, we apply Poincar\'{e}'s inequality together with H\"{o}lder's inequality to deduce that
    \begin{align*}
        \left\langle 4 \cos(\theta) v_{n}^{\pm},  u_{n}^{\pm} \right\rangle &\lesssim \| v_{n}^{\pm} \|_{L^{\infty}} \| u_{n}^{\pm} \|_{L^{2}}\\
		& \lesssim \| \eta_{n}^{\pm} \|_{L^{2}} \| u_{n}^{\pm} \|_{L^{2}}\\
		& \lesssim \| \partial_{\theta}\eta_{n}^{\pm} \|_{L^{2}} \| u_{n}^{\pm} \|_{L^{2}}\\
		&\lesssim \| u_{n}^{\pm} \|_{L^{2}}^{2}.
    \end{align*}
    Thus, we obtain
    \begin{align*}
	\frac{\mathrm d}{\mathrm dt} \left( \| u_{n}^{+} (t) \|_{L^{2}} + \| u_{n}^{-} (t)\|_{L^{2}} \right) \leq C \left( \| u_{n}^{+} (t) \|_{L^{2}} + \| u_{n}^{-} (t)\|_{L^{2}} \right) ,
    \end{align*}
    which means that
    \begin{align}\label{estimate:unL2}
        \sup_{0 \leq t \leq T} \left( \| u_{n}^{+} (t) \|_{L^{2}} + \| u_{n}^{-} (t)\|_{L^{2}} \right) \leq C
    \end{align}
    for any \(T > 0\)  where \(C\) is a positive constant independent of \(n\).
    
    \uline{Step 2\,b.} \emph{{\(H^{m}\)-estimates of \(u_{n}^{\pm}\) (for \(m \geq 1\)).}}
    Multiplying equation \cref{eq:un positive} by \( (k + 1)^{2m} \eta_{2,k}^{+} (t)\) and equation \cref{eq:un negative} by \( (k + 1)^{2m} \eta_{2,k}^{-} (t)\), summing up over \( k = 1, 2,\ldots, n \), and applying integration by parts, we obtain
    \begin{align*}
        \frac{1}{2} \frac{\mathrm d}{\mathrm dt} \left( \| \partial_{\theta} u_{n}^{+} (t) \|_{L^{2}} + \| \partial_{\theta}^{m} u_{n}^{-} (t)\|_{L^{2}} \right) = \left\langle \partial_{\theta}^{m} L_{1}^{+} (u_{n}^{+}), \partial_{\theta}^{m} u_{n}^{+} \right\rangle + \left\langle \partial_{\theta}^{m} L_{1}^{-} (u_{n}^{-}), \partial_{\theta}^{m} u_{n}^{-} \right\rangle.
    \end{align*}
 Thus,
    \begin{align*}
	  \partial_{\theta}^{m} L_{1}^{+} (u_{n}^{+}) &= -\frac{1}{2} \sin(2\theta) \partial_{\theta}^{m+1} u_{n}^{+} - (m \cos(2\theta) + \cos^{2}(\theta) ) \partial_{\theta}^{m} u_{n}^{+} - 2 \cos(\theta) H (\sin(\theta) \partial_{\theta}^{m} u_{n}^{+}) + \text{l.\,o.\,t.},\\
      \partial_{\theta}^{m} L_{1}^{-} (u_{n}^{-}) &= \frac{1}{2} \sin(2\theta) \partial_{\theta}^{m+1} u_{n}^{-} + (m \cos(2\theta) + \cos^{2}(\theta) ) \partial_{\theta}^{m} u_{n}^{-} + 2 \cos(\theta) H (\sin(\theta) \partial_{\theta}^{m} u_{n}^{-}) + \text{l.\,o.\,t.},
    \end{align*}
      where ``l.\,o.\,t.'' denotes  \textit{lower-order terms} whose \(L^{2}\)-norms can be bounded by \(\| u_{n}^{\pm}\|_{H^{m-1}}\).  Taking \(\eta^{+}_{n}\) as an example. For the term \(\cos(\theta) H (\sin(\theta) u_{n}^{+})\), using \cref{le:hilbert estimate}, we yield
      \begin{align*}
          \|\partial_{\theta}^{m-k} (\cos(\theta))  H (\partial_{\theta}^{k} ( \sin(\theta) u_{n}^{+}))\|_{L^{2}} & \lesssim \| H (\partial_{\theta}^{k} ( \sin(\theta) u_{n}^{+}))\|_{L^{2}}\\
          & \lesssim \| \partial_{\theta}^{k} ( \sin(\theta) u_{n}^{+})\|_{L^{2}} \lesssim \| u_{n}^{+}\|_{H^{k}}  
      \end{align*}
for \(0 \leq k \leq m-1\). For the term \(\cos(\theta) \eta^{+}_{n}\), since
      \begin{align*}
          \partial_{\theta}^{m} \big(\cos(\theta) \eta^{+}_{n}\big) = \sum_{k \geq 0}^{m} \partial_{\theta}^{m-k} \cos(\theta) \partial_{\theta}^{k} \eta^{+}_{n}.
      \end{align*}
      If \(k = 0\), we apply Poincar\'{e}'s inequality to obtain
      \begin{align*}
          \|\partial_{\theta}^{m} \cos(\theta)  \eta^{+}_{n}\|_{L^{2}} & \lesssim \|  \eta^{+}_{n}\|_{L^{\infty}} \leq \|  \partial_{\theta} \eta^{+}_{n}\|_{L^{2}}\\
          & \leq \| \rho_{2}^{-1/2} \rho_{2}^{1/2} \partial_{\theta} \eta^{+}_{n}\|_{L^{2}} \lesssim \| u^{+}_{n}\|_{L^{2}}.
      \end{align*}
      If \(1 \leq k \leq m\), since \(u_{n}^{+} = -\sqrt{\pi} \rho_{2}^{1/2} \partial_{\theta} \eta_{n}^{+}\), we have
      \begin{align*}
          \|\partial_{\theta}^{m-k} \cos(\theta)  \partial_{\theta}^{k} \eta^{+}_{n}\|_{L^{2}} & = \|\partial_{\theta}^{m-k} \cos(\theta)  \partial_{\theta}^{k-1} (\partial_{\theta}\eta^{+}_{n})\|_{L^{2}}\\
          & = \|\partial_{\theta}^{m-k} \cos(\theta)  \partial_{\theta}^{k-1} (-2\sin(\theta) u^{+}_{n})\|_{L^{2}}\\
          & \lesssim \| u^{+}_{n}\|_{H^{k-1}}.
   \end{align*}
   Similarly, we obtain the other \textit{lower-order terms} whose \(L^{2}\)-norms can be bounded by \(\| u_{n}^{\+}\|_{H^{m-1}}\). 

    Applying H\"{o}lder's inequality and \cref{le:hilbert estimate}, we estimate
    \begin{align*}
        \left\langle 2 \cos(\theta) H(\sin(\theta) \partial_{\theta}^{m} u_{n}^{\pm}), \partial_{\theta}^{m} u_{n}^{\pm} \right\rangle &\leq  2\| H (\sin(\theta) \partial_{\theta}^{m} u_{n}^{\pm} ) \|_{L^{2}} \| \partial_{\theta}^{m} u_{n}^{\pm} \|_{L^{2}}\\ 
        &\leq C \| \partial_{\theta}^{m} u_{n}^{\pm} \|_{L^{2}}^2.
    \end{align*}
    Therefore, we deduce
    \begin{align*}
	\frac{\mathrm d}{\mathrm dt} \left( \| \partial_{\theta}^{m} u_{n}^{+}(t)\|_{L^{2}} + \|\partial_{\theta}^{m} u_{n}^{-}(t)\|_{L^{2}} \right) \leq C \left( \| u_{n}^{+}(t)\|_{H^{m}} + \| u_{n}^{-}(t)\|_{H^{m}} \right),
    \end{align*}
    which leads to the uniform estimate
    \begin{align}\label{estimate:unHm}
        \sup_{0 \leq t \leq T} \|u_{n}^{\pm}(t)\|_{H^{m}} \leq C,
    \end{align}
    where \(C\) is a constant depending only on \(m\), \(\T\), initial data \( u_{0}^{\pm} \), and \(T > 0 \).
    
    \uline{Step 2\,c.} \emph{{\(L^{2}\)-estimates of \(\partial_{t} u_{n}^{\pm}\).}} 
    We multiply \cref{eq:un positive} by \( \frac{\mathrm d}{\mathrm dt} \eta_{2,k}^{+} (t) \) and equation \cref{eq:un negative} by \( \frac{\mathrm d}{\mathrm dt} \eta_{2,k}^{-} (t) \), and then sum up over \(k=1,2,\ldots,n\), we obtain 
    \begin{align*}
	\|\partial_{t} u_{n}^{\pm} \|_{L^{2}} = \left\langle L_{1}^{+} (u_{n}^{\pm}), \partial_{t}u_{n}^{\pm} \right\rangle.
    \end{align*}
    Due to \(\eta_{n}^{\pm} (t) \) are odd functions, it follows from H\"{o}lder's inequality and Poincar\'{e}'s inequality that
    \begin{align}\label{estimate:pt unL2}
        \sup_{0 \leq t \leq T} \| \partial_{t} u_{n}^{\pm} (t) \|_{L^{2}} \leq C \| u_{n}^{\pm} (t) \|_{H^{1}} \leq  C(\T, u_{0}^{\pm}, T),
    \end{align}
    which yields the desired uniform estimate for \( \| \partial_{t} u_{n}^{\pm} (t) \|_{L^{2}}\).

    \uline{Step 2\,d.} \emph{{\(L^{2}\)-estimates of \(\partial_{t} \partial_{\theta} u_{n}^{\pm}\).}}
Before deriving the \(L^{2}\)-estimates for \(\partial_{t}^{2} u_{n}^{\pm}\), we first establish the corresponding \(L^{2}\)-estimates for \(\partial_{t} \partial_{\theta} u_{n}^{\pm}\). We multiply equation \cref{eq:un positive} by \((k+1)^{2} \frac{\mathrm d}{\mathrm dt}\eta_{2,k}^{+} (t)\), and equation \cref{eq:un negative} by \((k+1)^{2} \frac{\mathrm d}{\mathrm dt}\eta_{2,k}^{-} (t)\), sum over \(k = 1, 2, \ldots, n\), and then apply integration by parts to find
    \begin{align*}
        \| \partial_{t} \partial_{\theta} u_{n}^{+} (t) \|_{L^{2}}^{2} + \| \partial_{t} \partial_{\theta} u_{n}^{-} (t) \|_{L^{2}}^{2} & = \left\langle \partial_{\theta} L_{1}^{+} (u_{n}^{+}), \partial_{t} \partial_{\theta} u_{n}^{+} \right\rangle + \left\langle \partial_{\theta} L_{1}^{-} (u_{n}^{-}), \partial_{t} \partial_{\theta} u_{n}^{-} \right\rangle\\
        & \leq \| \partial_{\theta} L_{1}^{+} (u_{n}^{+}) \|_{L^{2}} \| \partial_{t} \partial_{\theta} u_{n}^{+} \|_{L^{2}} +  \| \partial_{\theta} L_{1}^{-} (u_{n}^{-}) \|_{L^{2}} \| \partial_{t} \partial_{\theta} u_{n}^{-} \|_{L^{2}}\\
        & \lesssim \| u_{n}^{+} \|_{H^{2}} \| \partial_{t} \partial_{\theta}u_{n}^{+} \|_{L^{2}} + \| u_{n}^{-} \|_{H^{2}} \| \partial_{t} \partial_{\theta}u_{n}^{-} \|_{L^{2}}.
    \end{align*}
   Consequently, Young’s inequality yields
   \begin{align}\label{estimate:pt ptheta unL2}
       \sup_{0\leq t \leq T} \| \partial_{t} \partial_{\theta} u_{n}^{\pm}(t)\|_{L^{2}} \leq C(\T, \| u_{0}^{\pm} \|_{H^{2}}, T).
   \end{align}
    ~\\
    
    \uline{Step 2\,e.} \emph{{\(L^{2}\)-estimates of \(\partial_{t}^{2} u_{n}^{\pm}\).}} 
    We apply \(\partial_{t}\) to the equation \crefrange{eq:un positive}{eq:un negative}, multiply the resulting expressions by \( \frac{\mathrm d^{2}}{\mathrm dt^{2}} \eta_{2,k}^{+} (t) \) and \( \frac{\mathrm d^{2}}{\mathrm dt^{2}} \eta_{2,k}^{-} (t) \), respectively, and sum over \( k = 1, 2, \ldots, n \), to obtain
    \begin{align*}
        \| \partial_{t}^{2} u_{n}^{+} (t) \|_{L^{2}} + \| \partial_{t}^{2} u_{n}^{-} (t) \|_{L^{2}} = \left\langle \partial_{t} L_{1}^{+} (u_{n}^{+}), \partial_{t}^{2} u_{n}^{+} \right\rangle + \left\langle \partial_{t} L_{1}^{-} (u_{n}^{-}), \partial_{t}^{2} u_{n}^{-} \right\rangle.
    \end{align*}
    We have 
    \begin{align*}
       \partial_{t} L_{1}^{+} (u_{n}^{+}) & = - \frac{1}{2} \sin(2\theta) \partial_{t} \partial_{\theta} u_{n}^{+} - \cos^{2}(\theta) \partial_{t} u_{n}^{+}\\ & \qquad - 2 \cos(\theta) H (\sin(\theta) \partial_{t} u_{n}^{+}) + 2 \cos(\theta) \partial_{t} \eta_{n}^{+} + 4 \cos(\theta) \partial_{t} v_{n}^{+}, \\
       \partial_{t} L_{1}^{-} (u_{n}^{-}) & = - \frac{1}{2} \sin(2\theta) \partial_{t} \partial_{\theta} u_{n}^{-} - \cos^{2}(\theta) \partial_{t} u_{n}^{-} \\ & \qquad + 2 \cos(\theta) H (\sin(\theta) \partial_{t} u_{n}^{-}) + 2 \cos(\theta) \partial_{t} \eta_{n}^{-} - 4 \cos(\theta) \partial_{t} v_{n}^{-}.
    \end{align*}
    To estimate the term \(\left\langle \cos(\theta) \partial_{t} \eta_{n}^{\pm}, \partial_{t}^{2} u_{n}^{\pm} \right\rangle\), since \(\partial_{t} \eta_{n}^{\pm} (0) = 0\), we apply H\"{o}lder's inequality together with Poincar\'{e}'s inequality to obtain
    \begin{align*}
        \left\langle \cos(\theta) \partial_{t} \eta_{n}^{\pm}, \partial_{t}^{2} u_{n}^{\pm} \right\rangle & \leq \|\partial_{t} \eta_{n}^{\pm} \|_{L^{\infty}} \| \cos(\theta) \|_{L^{2}} \| \partial_{t}^{2} u_{n}^{\pm} \|_{L^{2}}\\	
        & \lesssim \| \partial_{t} \partial_{\theta} \eta_{n}^{\pm} \|_{L^{2}} \| \partial_{t}^{2} u_{n}^{\pm} \|_{L^2}\\\
        & \lesssim \| u_{n}^{\pm} \|_{H^{2}} \| \partial_{t}^{2} u_{n}^{\pm} \|_{L^{2}}.
    \end{align*}
    For the term \( \left\langle \cos(\theta) \partial_{t}v_{n}^{\pm}, \partial_{t}^{2} u_{n}^{\pm} \right\rangle \), we invoke H\"{o}lder's inequality, Sobolev's embedding, and \cref{le:hilbert estimate} to deduce that
    \begin{align*}
        \left\langle \cos(\theta) \partial_{t} v_{n}^{\pm}, \partial_{t}^{2} u_{n}^{\pm} \right\rangle & \leq \| \partial_{t} v_{n}^{\pm} \|_{L^{\infty}} \| \cos(\theta) \|_{L^{2}} \| \partial_{t}^{2} u_{n}^{\pm} \|_{L^{2}}\\
        & \lesssim \| H (\partial_{t} \eta_{n}^{\pm} ) \|_{L^{2}} \| \partial_{t}^{2} u_{n}^{\pm} \|_{L^{2}}\\
        & \lesssim \| \partial_{t} \eta_{n}^{\pm} \|_{L^{2}} \| \partial_{t}^{2} u_{n}^{\pm} \|_{L^{2}}\\
        & \lesssim \|\partial_{t} \partial_{\theta} \eta_{n}^{\pm} \|_{L^{2}} \|\partial_{t}^{2} u_{n}^{\pm} \|_{L^{2}}\\
        & \lesssim \| \partial_{t} u_{n}^{\pm} \|_{L^{2}} \| \partial_{t}^{2} u_{n}^{\pm} \|_{L^{2}}.
    \end{align*}
    As a result, combining this with \cref{estimate:unHm} and \cref{estimate:pt ptheta unL2}, we obtain 
    \begin{align*}
        \sup_{0 \leq t \leq T} \| \partial_{t}^{2} u_{n}^{\pm} (t) \|_{L^{2}} \leq C(\T, u_{0}^{\pm}, T),
    \end{align*}
    which yields a uniform bound for \( \partial_{t}^{2} u_{n}^{\pm} (t) \).

    \uline{Step 2\,f.} \emph{{Higher order estimates of \(\partial_{t} u_{n}^{\pm}\) and \(\partial_{t}^{2} u_{n}^{\pm}\).}}
    To derive the \(H^{m-1}\)-estimate for \(\partial_{t} u_{n}^{\pm} (t) \) \( (m \geq 2)\), we multiply equation \cref{eq:un positive} by \((k+1)^{2m-2} \frac{\mathrm d}{\mathrm dt} \eta_{2,k}^{+} (t)\), equation \cref{eq:un negative} by \((k+1)^{2m-2} \frac{\mathrm d}{\mathrm dt} \eta_{2,k}^{-} (t)\), and then sum up
    over \( k = 1, 2, \ldots, n\). It then follows from  integration by parts that
    \begin{align*}
        \| \partial_{t} \partial_{\theta}^{m-1} u_{n}^{+} \|_{L^{2}} + \| \partial_{t} \partial_{\theta}^{m-1} u_{n}^{-} \|_{L^{2}} = \left\langle \partial_{\theta}^{m-1} L_{1}^{+} (u_{n}^{+}), \partial_{t} \partial_{\theta}^{m-1} u_{n}^{+} \right\rangle + \left\langle \partial_{\theta}^{m-1} L_{1}^{-} (u_{n}^{+}), \partial_{t} \partial_{\theta}^{m-1} u_{n}^{-} \right\rangle.
    \end{align*}
    We have
    \begin{align*}
        \partial_{\theta}^{m-1} L_{1}^{\pm} ( u_{n}^{\pm} ) = - \frac{1}{2} \sin(2\theta) \partial_{\theta}^{m} u_{n}^{\pm} + \text{l.\,o.\,t.}
    \end{align*}
    Combining with \cref{estimate:unHm}, we get
    \begin{align*}
        \| \partial_{t} \partial_{\theta}^{m-1} u_{n}^{\pm} \|_{L^{2}} \leq \| u_{n}^{\pm} \|_{H^{m}},
    \end{align*}
    where we have employed Sobolev's embedding and Young's inequality and \cref{le:hilbert estimate}.
    Thus, we conclude that
    \begin{align}\label{estimate:pt unHm}
        \sup_{0 \leq t \leq T} \| \partial_{t} u_{n}^{\pm} (t) \|_{H^{m-1}} \leq C ( \T, u_{0}^{\pm}, m, T).
    \end{align}
     
     Similarly, to derive the \(H^{m-2}\)-estimate for \( \partial_{t}^{2} u_{n}^{\pm} \) \((m \geq 2)\). We apply \(\partial_{t}\) to the equations \cref{eq:un positive} and \cref{eq:un negative}, multiply the resulting expressions by \( (k+1)^{2m-4} \frac{\mathrm d^{2}}{\mathrm dt^{2}} \eta_{2,k}^{+} (t) \) and \( (k+1)^{2m-4} \frac{\mathrm d^{2}}{\mathrm dt^{2}} \eta_{2,k}^{-} (t) \), respectively, and add up to \( k = 1, 2, \ldots, n \) to obtain 
     \begin{align*}
         &\| \partial_{t}^{2} \partial_{\theta}^{m-2} u_{n}^{+} (t) \|_{L^{2}}^{2} + \| \partial_{t}^{2} \partial_{\theta}^{m-2} u_{n}^{-} (t) \|_{L^{2}}^{2} \\ &= \left\langle \partial_{t}^{2} \partial_{\theta}^{m-2} L_{1}^{+} (u_{n}^{+}), \partial_{t} \partial_{\theta}^{m-2} u_{n}^{+} \right\rangle + \left\langle \partial_{t}^{2} \partial_{\theta}^{m-2} L_{1}^{-} (u_{n}^{-}), \partial_{t} \partial_{\theta}^{m-2} u_{n}^{-} \right\rangle.
     \end{align*}
     Combining with \cref{estimate:unL2}, \cref{estimate:unHm}, \cref{estimate:pt unL2}, and \cref{estimate:pt unHm}, we deduce
     \begin{align}\label{estimate:ptt unHm}
         \sup_{0 \leq t \leq T} \| \partial_{t} u_{n}^{\pm} (t) \|_{H^{m-2}} \leq C(\T, m, u_{0}^{\pm}, T),
     \end{align}
     which yields a uniform estimate for \( \| \partial_{t}^{2} u_{n}^{\pm} (t) \|_{H^{m-2}}\).

    Since \(u_{n}^{\pm} = -\sqrt{\pi} \partial_{\theta} \eta_{n}^{\pm}\), estimates \cref{estimate:unHm}, \cref{estimate:pt unHm}, and \cref{estimate:ptt unHm} can be equivalently rewritten as
    \begin{align}
    \sup_{0 \leq t \leq T} \| \rho_2^{1/2} \partial_{\theta} \eta_{n}^{\pm} (t) \|_{H^{m}} &\leq C(\T, m, \rho_2^{1/2} \partial_{\theta} \eta_{0}^{\pm}, T),\label{eq:etan Hm}\\
     \sup_{0 \leq t \leq T} \| \rho_2^{1/2} \partial_{t} \partial_{\theta} \eta_{n}^{\pm} (t) \|_{H^{m-1}} &\leq C(\T, m, \rho_2^{1/2} \partial_{\theta} \eta_{0}^{\pm}, T),\label{eq:ptetan Hm}\\
     \sup_{0 \leq t \leq T} \| \rho_2^{1/2} \partial_{t}^{2} \partial_{\theta} \eta_{n}^{\pm} (t) \|_{H^{m-2}} &\leq C(\T, m, \rho_2^{1/2} \partial_{\theta} \eta_{0}^{\pm}, T).\label{eq:pttetan Hm}
    \end{align}
for any \(T > 0\). 

Owing to \crefrange{eq:etan Hm}{eq:pttetan Hm}, we obtain the following decay estimates for the Fourier coefficients of \(\eta_{2,k}^{\pm} (t)\) in \cref{eq:etan expression}:
    \begin{align}
       \label{eq:decay-coeff}  \sup_{0\leq t \leq T} |\eta_{2,k}^{\pm} (t)| &\leq \frac{C(\T, m ,\eta_{0}^{\pm}, T)}{k^{m}},\\
       \label{eq:decay-coeff-dt}  \sup_{0\leq t \leq T} \left|\frac{\mathrm d}{\mathrm dt} \eta_{2,k}^{\pm} (t)\right| &\leq \frac{C(\T, m ,\eta_{0}^{\pm}, T)}{k^{m-1}},\\
       \label{eq:decay-coeff-ddt} \sup_{0\leq t \leq T} \left|\frac{\mathrm d^{2}}{\mathrm dt^{2}} \eta_{2,k}^{\pm} (t)\right| &\leq \frac{C(\T, m ,\eta_{0}^{\pm}, T)}{k^{m-2}}.
    \end{align}

\uline{Step 3.} \emph{Convergence of the approximate solutions.}
     Combining \cref{estimate:unHm}, \cref{estimate:pt unHm}, and \cref{estimate:ptt unHm}, we deduce that there exists a subsequence \( \{u_{n_{l}}^{\pm} \}_{n_{l}=1}^{\infty} \subset \{u_{n}^{\pm} \}_{n=1}^{\infty} \) such that
     \begin{align*}
         u_{n_{l}}^{\pm} &\rightharpoonup u^{\pm}\ &&\text{weakly in } L^{2}( [0, T] ; H^{m}(\T));\\
         \partial_{t} u_{n_{l}}^{\pm} &\rightharpoonup \partial_{t} u^{\pm}\ &&\text{weakly in } L^{2}( [0, T] ; H^{m-1}(\T));\\
         \partial_{t}^{2} u_{n_{l}}^{\pm} &\rightharpoonup \partial_{t}^{2} u^{\pm}\ &&\text{weakly in } L^{2}( [0, T] ; H^{m-2}(\T)).  
     \end{align*}
     In fact, since \( \partial_{t} \eta_{n}^{\pm} = - 2 \sin(\theta) u_{n}^{\pm} \) and \( \eta^{\pm} \) are odd functions, we deduce that
     \begin{align*}
         \eta_{n_{l}}^{\pm} &\rightharpoonup \eta^{\pm} &&\text{weakly in } L^{2}( [0, T] ; H^{m+1}(\T));\\
         \partial_{t} \eta_{n_{l}}^{\pm} &\rightharpoonup \partial_{t} \eta^{\pm} &&\text{weakly in } L^{2}( [0, T] ; H^{m}(\T));\\
         \partial_{t}^{2} \eta_{n_{l}}^{\pm} &\rightharpoonup \partial_{t}^{2} \eta^{\pm} &&\text{weakly in } L^{2}( [0, T] ; H^{m-1}(\T)).  
     \end{align*}
      For simplicity, we set \( n_{l} = n \). Let \(N\) be a fixed positive integer, and choose functions \( \phi^{\pm} \in C^{1} ( [0,T]; C^{2}(\T)) \) such that
	\begin{align}\label{eq:phi}
		\phi^{\pm} (t,\theta) = \sum_{k=1}^{N}\phi_{2,k}^{\pm} (t) \sin(k\theta),
	\end{align}
	where \( \phi_{2,k}^{\pm} (t)\), for  \(k = 1, 2, \ldots, N\), are smooth functions. We choose \( n \geq N\), multiply equation \cref{eq:un positive} by \( \phi_{2,k}^{+} (t) \) and equation \cref{eq:un negative} by \( \phi_{2,k}^{-} (t) \), sum over \( k = 0, 1, 2, \ldots, n\), and integrate with respect to \(t\). We then deduce that 
    \begin{align*}
        \int_{0}^{T} \Big(\left\langle \partial_{t} u_{n}^{\pm} (t), \phi^{\pm} (t) \right\rangle - \left\langle L_{1}^{\pm} (u_{n}^{\pm})(t), \phi^{\pm} (t) \right\rangle \Big)\mathrm d t = 0.
    \end{align*}
    Since the sequence \( \{ u_{n}^{\pm} \}_{n=1}^{\infty} \) converges to \( u^{\pm} \) weakly in \( L^{2} ( [0,T]; H^{m} (\T)) \) and \( \{\partial_{t} u_{n}^{\pm} \}_{n=1}^{\infty} \) converges to \( \partial_{t} u^{\pm} \) weakly in \( L^{2} ( [0,T]; H^{m-1} (\T)) \), it follows that
    \begin{align}\label{eq:int phi}
        \int_{0}^{T} \Big(\left\langle \partial_{t} u^{\pm} (t), \phi^{\pm} (t) \right\rangle - \left\langle L_{1}^{\pm} (u^{\pm}) (t), \phi^{\pm} (t) \right\rangle\Big) \mathrm d t = 0.
    \end{align}
    Since functions of the form \cref{eq:phi} are dense in \( C^{1} ( [0,T]; C^{2} (\T)) \), equation \cref{eq:int phi} holds for all \( \phi^{\pm} \in C^{1} ( [0,T]; C^{2} (\T)) \). Combining the coefficient estimates \crefrange{eq:decay-coeff}{eq:decay-coeff-ddt} for \( \eta_{n}^{\pm} \), we deduce that \( \{ u_{n}^{\pm} \}_{n=1}^{\infty} \) strongly converges to \( u^{\pm} = - \sqrt{\pi} \rho_2^{1/2} \partial_{\theta} \eta^{\pm} \) in \( C ([0,T]; H^{m} (\T)) \). 
    
\uline{Step 4.} \emph{Uniqueness.}
     If both \( \eta_{1}^{\pm} \) and \( \eta_{2}^{\pm} \) are solutions to \cref{eq:linearized} with the same initial data \( \eta_{0}^{\pm} \), then applying Gr\"{o}nwall's inequality yields
     \begin{align*}
         \|\rho_2^{1/2} \partial_{\theta} \eta_{1}^{\pm}-\rho_2^{1/2} \partial_{\theta} \eta_{2}^{\pm} \|_{L^{2}} \leq e^{CT} \| \rho_2^{1/2} \partial_{\theta} \eta_{0}^{\pm} -\rho_2^{1/2} \partial_{\theta} \eta_{0}^{\pm} \|_{L^{2}} = 0,
     \end{align*}
     which implies the uniqueness of the solution to  \cref{eq:linearized}, thereby completing the proof of \cref{the:existence}.
\end{proof}

\section{Instability for the linearized problem}
\label{sec:instability-linearized}

In this section, we study the instability of the linearized problem \cref{eq:linearized} and prove \cref{the:instability_linearized}. 
The solution \(\eta^{\pm}\) to \cref{eq:linearized} obtained in \cref{the:existence} can be expressed as \(\eta^{\pm} (t,\theta) = \sum_{k\geq 1} \eta_{2,k}^{\pm} (t) e_{2,k}\),  using the orthonormal basis \(\{e_{2,k}\}_{k\ge1}\) of the weighted space \(\mathcal{H}_{2}\)  (which is recalled in \cref{ssec:spaces}). The key advantage of this representation is that the inner product \(\left\langle L^{\pm} \eta^{\pm}, \eta^{\pm}\right\rangle_{\mathcal{H}_{2}}\)
 can be expressed as a series involving only the sums of squares of the expansion coefficients \(\eta^{\pm}_{2,k} (t)\), without cross terms such as \(\eta^{\pm}_{2,k} (t) \eta^{\pm}_{2,k+2} (t)\). 

\subsection{Linear estimate on \texorpdfstring{$\eta^{-}$}{eta-}}
\label{ssec:lineareta-}

Direct computation yields
\begin{align*}
    L^{-} \sin(k\theta) = a_{2,k}^{-} \sin((k+2)\theta) + b_{2,k}^{+} \sin((k-2)\theta), \qquad k > 2,
\end{align*}
where
\begin{align*}
    a_{2,k}^{-} \coloneqq -\frac{(k+2)(k-2)}{4k}, \qquad b_{2,k}^{-} \coloneqq \frac{(k+2)^{2}}{4k}.
\end{align*}
For \(k = 1\), we have
\begin{align*}
    L^{-} \sin(\theta) = \frac{3}{4} \sin(3\theta) - \frac{9}{4} \sin(\theta);
\end{align*}
for \(k = 2\), we have
\begin{align*}
    L^{-} \sin(2\theta) = 0.
\end{align*}
It follows that
\begin{align*}
    L^{-} e_{2,k} & =  \frac{1}{k+2} L^{-} \sin((k+2)\theta) - \frac{1}{k} L^{-} \sin(k\theta)\\
     & = \frac{1}{k+2} \left( a_{2,k+2}^{-} \sin((k+4)\theta) + b_{2,k+2}^{-} \sin(k\theta) \right) - \frac{1}{k} \left( a_{2,k}^{-} \sin((k+2)\theta) + b_{2,k}^{-} \sin((k-2)\theta) \right)\\
     & = - \frac{(k+4)^{2}k}{4(k+2)^{2}} \frac{\sin((k+4)\theta)}{k+4} + \frac{(k+4)^{2}k}{4(k+2)^{2}} \frac{\sin(k\theta)}{k} + \frac{(k+2)^{2}(k-2)}{4 k^{2}} \frac{\sin((k+2)\theta)}{k+2}\\
     & \quad - \frac{(k+2)^{2}(k-2)}{4 k^{2}} \frac{\sin((k-2)\theta)}{k-2}\\
     & = - \frac{(k+4)^{2}k}{4(k+2)^{2}} \bigg( \frac{\sin((k+4)\theta)}{k+4} - \frac{\sin((k+2)\theta)}{k+2} \bigg) + \frac{(k+2)^{2}(k-2)}{4 k^{2}} \bigg( \frac{\sin(k\theta)}{k} - \frac{\sin((k-2)\theta)}{k-2}\bigg)\\
     & \quad + \bigg( - \frac{(k+4)^{2}k}{4(k+2)^{2}} + \frac{(k+2)^{2}(k-2)}{4 k^{2}} \bigg) \bigg( \frac{\sin((k+2)\theta)}{k+2} - \frac{\sin(k\theta)}{k} \bigg)\\
     & = - \frac{(k+4)^{2}k}{4(k+2)^{2}} e_{2, k + 2} + \bigg( - \frac{(k+4)^{2}k}{4(k+2)^{2}} + \frac{(k+2)^{2}(k-2)}{4 k^{2}} \bigg) e_{2, k} + \frac{(k+2)^{2}(k-2)}{4 k^{2}} e_{2, k -2},
\end{align*}
i.\,e.,
\begin{align*}
    L^{-} e_{2,k} = - d_{2,k + 2}^{-} e_{2, k + 2} + ( - d_{2,k + 2}^{-} + d_{2,k}^{-} ) e_{2, k } + d_{2,k}^{-} e_{2, k -2},
\end{align*}
where
\begin{align*}
    d_{2,k}^{-} \coloneqq \frac{(k+2)^{2}(k-2)}{4 k^{2}}.
\end{align*}
Since \(d_{2,2}^{-} = 0\), the above equality holds for all \(k \geq 2\). For \(k = 1\), we have
\begin{align*}
    L^{-} e_{2, 1} = - d_{2,3}^{-} e_{2,3} + ( - d_{2,3}^{-} + d_{2,1}^{-} ) e_{2, 1}.
\end{align*}
Expanding
\begin{align*}
    \eta^{-}(t,\theta) = \sum_{k\geq 1} \eta_{2,k}^{-}(t) e_{2,k},
\end{align*}
 then \(\partial_{t}\eta^{-} = L^{-}\eta^{-}\) reduces to the following infinite-dimensional system of ordinary differential equations
\begin{align*}
    \frac{\mathrm{d}}{\mathrm{d}t} \eta_{2, k}^{-} (t) = - d_{2,k}^{-} \eta_{2, k - 2}^{-} (t) + ( d_{2,k}^{-} - d_{2,k+2}^{-} ) \eta_{2, k}^{-} (t) + d_{2,k + 2}^{-} \eta_{2, k + 2}^{-} (t), \quad \text{for \(k\geq 1\).}
\end{align*}
where \( \eta_{2, - 1}^{-} (t)\) and \( \eta_{2, 0}^{-} (t)\) are understood to be zero. Consequently, we deduce (formally)
\begin{align}\label{eq:L- decay}
    \frac{1}{2} \frac{\mathrm{d}}{\mathrm{d}t} \sum_{k\geq 1} (\eta_{2, k}^{-} )^{2} & \leq \sum_{k \geq 1} - d_{2,k}^{-} \eta_{2, k-2}^{-} \eta_{2, k}^{-} + ( d_{2,k}^{-} - d_{2,k + 2}^{-} )(\eta_{2, k}^{-} )^{2} + d_{2,k + 2} \eta_{2, k}^{-} \eta_{2, k + 2}^{-}\nonumber\\
    & = \sum_{k \geq 1} ( d_{2,k}^{-} - d_{2,k + 2}^{-} )(\eta_{2, k}^{-} )^{2}\\
    & \leq - \frac{1}{2} \sum_{k \geq 1} (\eta_{2, k}^{-} )^{2}.\nonumber
\end{align}
Here, we have used the fact that
\begin{align*}
    d_{2,k}^{-} - d_{2,k+2}^{-} & = \frac{(k+2)^{2}(k-2)}{4 k^{2}} - \frac{(k+4)^{2}k}{4(k+2)^{2}}\\
    & = - \frac{1}{2} - \frac{2 k^{2} + 12 k + 8}{(k + 2)^{2} k^{2}}\\
    & < - \frac{1}{2}.
\end{align*}

Several approaches are available to justify the following formal computations (the summations involved may not converge). One option is to invoke the decay estimates for the coefficient \(\eta_{2,k}^{-} (t)\) to ensure convergence. Alternatively, one may appeal to standard linear semigroup theory, as outlined below. 

Let us consider the real Hilbert space \(\mathcal{W}\), formally spanned by the orthonormal basis 
\(\{e_{2,k}\}_{k \geq 1}\), defined as
\begin{align*}
    \mathcal{W} \coloneqq \{ \eta = \eta_{2,k} (t) e_{2,k}: \{e_{2,k}\}_{k \geq 1} \in \ell^{2} \} .
\end{align*}
Then \(L^{-}\) defines an unbounded, closed operator on  \(\mathcal{W}\). By a direct application of  Hille--Yosida's theorem (see \cite[Theorem II.3.8]{MR1721989}), it follows that \(L^{-}\) generates a strongly continuous semigroup with the desired decay estimate:
\begin{align*}
    \| e^{t L^{-}} \eta (0) \|_{\mathcal{W}} \leq e^{- \frac{1}{2} t} \| \eta (0) \|_{\mathcal{W}}, \qquad \text{for \(t \ge 0\)}.
\end{align*}

\subsection{Linear estimate on \texorpdfstring{$\eta^{+}$}{eta+}}
\label{ssec:lineareta+}

Similarly to \cref{ssec:lineareta-}, we have 
\begin{align*}
    L^{+} \sin(k\theta) = a_{2,k}^{+} \sin((k + 2)\theta) + b_{2,k}^{+} \sin((k-2)\theta), \qquad k > 2,
\end{align*}
where 
\begin{align*}
    a_{2,k}^{+} \coloneqq - \frac{(k - 2)^{2}}{4 k}, \qquad b_{2,k}^{+} \coloneqq \frac{(k + 2)(k - 2)}{4 k}.
\end{align*}
For \(k = 1\), we have
\begin{align*}
    L^{+} \sin(\theta) = - \frac{1}{4} \sin(3\theta) + \frac{3}{4} \sin(\theta).
\end{align*}
For \(k = 2\), we have
\begin{align*}
    L^{+} \sin(2\theta) = 0.
\end{align*}
It follows that
\begin{align*}
    L^{+} e_{2, k} & = \frac{a_{2,k + 2}^{+}}{k + 2} \sin((k + 4)\theta) + \frac{b_{k + 2}^{+}}{k + 2} \sin(k\theta) - \frac{a_{k}^{+}}{k} \sin((k + 2)\theta) - \frac{b_{k}^{+}}{k} \sin((k - 2)\theta)\\
    & = - \frac{\kappa^{2} (k + 4)}{4 (k + 2)^{2}} e_{2, k + 2} + \bigg(\frac{(k - 2)^{2} (k + 2)}{4 \kappa^{2}} - \frac{- \kappa^{2} (k + 4)}{4 (k + 2)^{2}}\bigg) e_{2, k} + \frac{(k - 2)^{2} (k + 2)}{4 \kappa^{2}} e_{2, k -2},
\end{align*}
which is
\begin{align*}
    L^{+} e_{2, k} = - d_{2,k + 2}^{+} e_{2, k + 2} + ( - d_{2,k + 2}^{+} + d_{2,k}^{+} ) e_{2, k } + d_{k}^{+} e_{2, k -2}, \qquad \text{for \(k > 2\),}
\end{align*}
where
\begin{align*}
    d_{2,k}^{+} \coloneqq \frac{(k - 2)^{2} (k + 2)}{4 k^{2}}.
\end{align*}
Since \(d_{2,2}^{+} = 0\), the above identity holds for all \(k \geq 2\). In the case \(k = 1\), we obtain the following. 
\begin{align*}
    L^{+} e_{2, 1} = - d_{2,3} e_{2, 3} + (- d_{2,3} + d_{2,1}) e_{2, 1}.
\end{align*}
We note that
\begin{align*}
    d_{2,k}^{+} - d_{2,k + 2}^{+} & =  - \frac{- k^{2} (k + 4)}{4 (k + 2)^{2}} + \frac{(k - 2)^{2} (k + 2)}{4 k^{2}}\\
    & = - \frac{k^{4} + 4 k^{3} + 8 k^{2} - 8 k - 16}{2 (k + 2)^{2} k ^{2}}\\
    & = - \frac{1}{2} - \frac{2 \kappa^{2} - 4 k - 8}{(k + 2)^{2} k^{2}}.
\end{align*}
This implies that \(- d_{2,3}^{+} + d_{2,1}^{+} = \frac{11}{18} > 0\), while, for \(k \geq 2\), we have \(- d_{2,k + 2}^{+} + d_{2,k}^{+} \leq - \frac{3}{8}\).

  \begin{remark}[Stability for a suitable class of initial data]\label{rk:eta+stab}
       Following the discussion in \cite[Section 4]{LLR2020}, let us suppose the initial data \(\eta^{+}_{0} \) is of the form \(\eta^{+}_{0} = \sum_{k\geq 1} \eta^{+}_{2,2k} (0) e_{2,2k}\). In this particular case, we have the following estimate:
        \begin{align*}
            \frac{1}{2} \frac{\mathrm{d}}{\mathrm{d}t} \sum_{k\geq 1} (\eta^{+}_{2,2k})^{2} & \leq \sum_{k \geq 1} - d_{2,2k}^{+} \eta_{2, 2k-2}^{+} \eta_{2, 2k}^{+} + ( d_{2,2k}^{+} - d_{2,2k + 2}^{+} )(\eta_{2, 2k}^{+} )^{2} + d_{2,2k + 2}^{+} \eta_{2, 2k}^{+} \eta_{2, 2k + 2}^{+}\\
            & = \sum_{k \geq 1} ( d_{2,2k}^{+} - d_{2,2k + 2}^{+} )(\eta_{2, 2k}^{+} )^{2}\\
            & \leq - \frac{3}{8} \sum_{k \geq 1} (\eta_{2, 2k}^{-} )^{2},
        \end{align*}
        where we have used 
        \begin{align*}
            d_{2,2k}^{+} - d_{2,2k + 2}^{+} & = - \frac{(2k)^{2} (2k + 4)}{4 (2k + 2)^{2}} + \frac{(2k - 2)^{2} (2k + 2)}{4 (2k)^{2}}\\
            & = - \frac{k^{2} (k + 2)}{2 (k + 1)^{2}} + \frac{(k - 1)^{2} (k + 1)}{2 k^{2}} \\
            & = - \frac{1}{2} - \frac{ k^{2} -  k - 1}{2 (k + 1)^{2} k^{2}} \leq -\frac{3}{8}.
        \end{align*}
        This implies exponential decay of \(\eta^{+}\) in \(\mathcal{H}_{2}\) for such initial data.
        
    \end{remark}

\subsection{Proof of the instability result for the linearized problem}

Building on \crefrange{ssec:lineareta-}{ssec:lineareta+}, we now establish \cref{the:instability_linearized}. 

While $\eta^-$ satisfies an exponential-decay estimate (see \cref{ssec:lineareta-}), the situation of $\eta^+$ is more delicate since the coefficients \(d_{k}^{+} - d_{k+2}^{+}\) arising in \cref{ssec:lineareta+} change sign. If the odd Fourier coefficients of $\eta_0^+$ vanish, then exponential stability holds (as proven in \cref{rk:eta+stab}). In the general case, we first show that the expansion coefficients \(\eta^{+}_{2,k} (t)\) of the solution \(\eta^{+} (t,\theta) = \sum_{k\geq 1} \eta_{2,k}^{+}(t) e_{2,k}\), obtained in \cref{the:existence}, satisfy the infinite-dimensional ODE system \cref{eq:eta positive ODE}; we then derive a second-order differential inequality \cref{eq:eta diff ineq} for the norm \(\|\eta^{+} (t) \|_{\mathcal{H}_{2}}\) which yields exponential-in-time growth  for a broad class of initial data, implying the instability of \cref{eq:linearized}. 

Depending on these different choices of the initial data \(\eta_{0}^{+}\), we later analyze the stability and instability of the nonlinear system \cref{eq:perturbations2} as well (in \cref{sec:instability-nonlinear} and \cref{sec:stability-nonlinear}, respectively).

\begin{proof}[Proof of \cref{the:instability_linearized}]
 
 Consider the solution \( ( \eta^{+}, \eta^{-} ) \) obtained in \cref{the:existence}. Here, \( \eta^{+} \) admits the expansion 
 \begin{align*}
     \eta^{+} (t, \theta) = \sum_{k \geq 1} \eta_{2, k}^{+} (t) e_{2, k}.
 \end{align*}
  Note that the coefficients \(\eta_{2, k}^{+} (t)\) appearing here may differ from those in \cref{eq:etan expression}, the decay estimates stated in \cref{eq:decay-coeff}–\cref{eq:decay-coeff-ddt} remain valid for both sets of coefficients, due to the convergence of the approximate solutions and the uniqueness of the solutions.  
  
  \uline{Step 1.} \emph{The infinite dimensional ODE system for \(\eta^{+}\).}
  Owing to \crefrange{eq:decay-coeff}{eq:decay-coeff-ddt}, then \(\partial_{t} \eta^{+} = L^{+} \eta^{+}\) reduces to the following infinite-dimensional ordinary differential equation (ODE) system
  \begin{align}\label{eq:eta positive ODE}
      \frac{\mathrm d}{\mathrm dt} \eta_{2,k}^{+} (t) = - d_{2,k}^{+} \eta_{2, k-2}^{+} (t) + ( d_{2,k}^{+} - d_{2,k+2}^{+} ) \eta_{2,k}^{+} (t) + d_{2,k+2}^{+} \eta_{2, k+2}^{+} (t), \qquad \text{for \(k \geq 1 \).}
  \end{align}
 Here, \( \eta_{2, -1}^{+} (t) \) and \( \eta_{2, 0}^{+} (t) \) are understood to be 0.

  \uline{Step 2.} \emph{The quadratic form on \( \eta_{2, k}^{+}\) and \( \eta_{2, k + 2}^{+} \).}
  In this step, we show that the non-trivial solution \(\eta^{+} \) obtained in \cref{the:existence} satisfies a second-order ordinary differential inequality. By applying the decay estimate for the coefficients of \( \eta^{+} (t, \theta) \) given in \crefrange{eq:decay-coeff}{eq:decay-coeff-ddt}, we obtain
  \begin{align*}
       \left| \frac{\mathrm d}{\mathrm dt} \left( d_{2,k}^{+} - d_{2,k+2}^{+} \right) \left( \eta_{2, k}^{+} (t) \right)^{2} \right| \leq \frac{C}{k^{2m-1}}, \qquad \text{for all \(t>0\), \( k \geq 3 \), and any integer \( m >3 \). }
  \end{align*}
   Consequently, we have
  \begin{align*}
      \frac{\mathrm d}{\mathrm dt} \left\langle L^{+} \eta^{+}, \eta^{+} \right\rangle_{\mathcal{H}_{2}} & = \frac{\mathrm d}{\mathrm dt} \sum_{k\geq 1} \left( d_{2,k}^{+} - d_{2,k+2}^{+} \right) \left( \eta_{2,k}^{+} (t) \right)^{2}\\
      & = \sum_{k\geq 1} \frac{\mathrm d}{\mathrm dt} \left( d_{2,k}^{+} - d_{2,k+2}^{+} \right) \left( \eta_{2, k}^{+} (t)\right)^{2}.
  \end{align*}
  We multiply \cref{eq:eta positive ODE} by \( ( d_{2,k}^{+} - d_{2,k+2}^{+} ) \eta_{2,k}^{+} (t)\) to obtain
  \begin{align}\label{eq:eta positive ODEd}
     \begin{aligned}
      \frac{1}{2} \frac{\mathrm d}{\mathrm dt} ( d_{2,k}^{+} - d_{2,k+2}^{+} ) \left( \eta_{2,k}^{+} (t) \right)^{2} & = - d_{2,k}^{+} ( d_{2,k}^{+} - d_{2,k+2}^{+} ) \eta_{2, k-2}^{+} (t) \eta_{2, k}^{+} (t) \\
			& \quad + ( d_{2,k}^{+} - d_{2,k+2}^{+} )^{2} \left( \eta_{2, k}^{+} (t) \right)^{2}\\
			& \quad + d_{2,k+2}^{+} ( d_{2,k}^{+} - d_{2,k+2}^{+} ) \eta_{2, k}^{+} (t) \eta_{2, k+2}^{+} (t).
            \end{aligned}
  \end{align}
  Considering the sum
  \begin{align*}
      S_{n}^{+} (t) \coloneqq \sum_{k=1}^{n} \frac{\mathrm d}{\mathrm dt} \left( d_{2,k}^{+} - d_{2,k+2}^{+} \right) \left(\eta_{2,k}^{+} (t) \right)^{2}
  \end{align*}
  and using \cref{eq:eta positive ODEd}, we get
  \begin{align}\label{eq:Sn}
  \begin{aligned}
      S_{n}^{+} (t) & \coloneqq \left( d_{2,1}^{+} - d_{2,3}^{+} \right)^{2} \left( \eta_{2, 1}^{+} (t) \right)^{2}  + \left( d_{2,2}^{+} - d_{2,4}^{+} \right)^{2} \left(\eta_{2, 2}^{+} (t) \right)^{2} + \sum_{k=1}^{n-2} Q_{k} (t)  \\
      & \quad + \left( d_{2,n-1}^{+} - d_{2,n+1}^{+} \right)^{2} \left( \eta_{2, n-1}^{+} (t) \right)^{2}  + \left( d_{2,n} - d_{2,n+2}^{+} \right)^{2} \left( \eta_{2,n}^{+} (t) \right)^{2} \\
      & \quad + R_{n-1} (t) + R_{n} (t).
      \end{aligned}
  \end{align}
  Here, the quadratic form  \( Q_{k} (t) \) is defined as
  \begin{align*}
      Q_{k}(t) & \coloneqq \left( d_{2,k}^{+} - d_{2,k+2}^{+} \right)^{2} \left( \eta_{2, k}^{+} (t) \right)^{2} + \left( d_{2,k+2}^{+} - d_{2,k+4}^{+} \right)^{2} \left( \eta_{2, k+2}^{+} (t) \right)^{2}\\
      & \quad + 2 \left( d_{2,k}^{+} d_{2,k+2}^{+} + d_{2,k+2}^{+} d_{2,k+4}^{+} - 2 (d_{2,k+2}^{+})^{2} \right ) \eta_{2, k}^{+} (t) \eta_{2, k+2}^{+} (t), \quad \text{for \( k = 1, 2, \ldots, n - 2\).}
  \end{align*}
   The remainder term \( R_{k} (t) \) is given by
  \begin{align*}
      R_{k}(t) \coloneqq 2 d_{2,k+2}^{+} \left( d_{2,k}^{+} - d_{2,k+2}^{+} \right) \eta_{2,k}^{+} (t) \eta_{2, k+2}^{+} (t), \quad \text{for \( k = n - 1,\, n\)}
  \end{align*}
 Applying the estimates for coefficients \( \eta_{2,k}^{+} (t) \) in \crefrange{eq:decay-coeff}{eq:decay-coeff-ddt}, we obtain
  \begin{align}\label{eq:estimate-Rk}
      |R_{k} (t) |\leq \frac{C(m,\T,\eta^+_0)}{k^{2m-1}}, \qquad \text{for all $t>0$.}
  \end{align}
     Consequently, both \( R_{n-1} (t) \) and \( R_{n} (t)\) tend to zero as \( n \rightarrow + \infty \). 
     
     For the estimate of the quadratic form  \( Q_{k} (t) \), we have
  \begin{align}\label{eq:estimate-Qf}
  \lambda_{\inf} \left[ \left( \eta_{2,k}^{+} (t) \right)^{2} + \left( \eta_{2, k+2}^{+} (t) \right)^{2} \right]\leq Q_{k} (t) \leq \lambda_{\sup} \left[ \left( \eta_{2, k}^{+} (t) \right)^{2} + \left( \eta_{2, k+2}^{+} (t) \right)^{2} \right], \quad \text{for \( k \geq 1 \).}
  \end{align}
  Substituting \cref{eq:estimate-Qf} into \cref{eq:Sn} yields 
  \begin{align}\label{eq:estimate-Snleq}
  \begin{aligned}
      S_{n}^{+} (t) & \leq \left( d_{2,1}^{+} - d_{2,3}^{+} \right)^{2} \left( \eta_{2,1}^{+} (t) \right)^{2} + \left( d_{2,2}^{+} - d_{2,4}^{+} \right)^{2} \left( \eta_{2,2}^{+} (t) \right)^{2} \\ & \quad + \sum_{k=1}^{n-2} \lambda_{\sup} \left[ \left(\eta_{2,k}^{+} (t) \right)^{2} + \left( \eta_{2,k+2}^{+} (t) \right)^{2} \right] \\
      & \quad + \left( d_{2,n-1}^{+} - d_{2,n+1}^{+} \right)^{2} \left(\eta_{2,n-1}^{+} (t) \right)^{2} + \left( d_{2,n}^{+} - d_{2,n+2}^{+} \right)^{2} \left( \eta_{2,n}^{+} (t) \right)^{2} + R_{n-1} (t) + R_{n} (t)  \\
      & \leq 2 \lambda_{\sup} \sum_{k=1}^{n} \left(\eta_{2,k}^{+} (t) \right)^{2} + R_{n-1} (t) + R_{n} (t)
      \end{aligned}
  \end{align}
  and 
  \begin{align}\label{eq:estimate-Sngeq}
      S_{n}^{+} (t) \geq 2 \sum_{k=1}^{n} \lambda_{\inf} \left( \eta_{2,k}^{+} (t) \right)^{2}  + R_{n-1} (t) + R_{n} (t).
  \end{align}
  Combining with \cref{eq:estimate-Snleq} and \cref{eq:estimate-Sngeq}, and letting \( n \rightarrow + \infty \), we obtain a second-order ordinary differential inequality for \(t \mapsto \|\eta^{+} (t) \|_{\mathcal{H}_{2}}\):
  \begin{align}\label{eq:eta diff ineq}
      4 \lambda_{1} \| \eta^{+} (t) \|_{\mathcal{H}_{2}}^{2} < \frac{\mathrm d^{2}}{\mathrm dt^{2}} \| \eta^{+} (t) \|_{\mathcal{H}_{2}}^{2} < 4 \lambda_{2} \| \eta^{+} (t) \|_{\mathcal{H}_{2}}^{2},
  \end{align}
   where we have used the decay property of \( R_{k} (t) \) in \cref{eq:estimate-Rk} for \( k = 1, 2, \ldots \), and \( 0 < \lambda_{1} < \lambda_{2} \) are two absolute constants as in \cref{le:lambda}.

    \uline{Step 3.} \emph{The second-ordinary differential inequality for \(\eta^{+} \).}
    In this step, we apply the comparison result in \cref{le:comparison the} to the differential inequality \cref{eq:eta diff ineq}. To begin with, we first consider the following second-order ordinary differential equation
    \begin{align}\label{eq:u diff ode}
        \begin{cases}
             \frac{\mathrm d^{2}}{\mathrm dt^{2}} u (t) = 4 \lambda_{i} u (t), \qquad  & t \geq 0,\\
             u (0) = 0, \\ \frac{\mathrm d}{\mathrm dt} u (0) = 1, &
        \end{cases} \qquad \text{(for \( i = 1,\, 2\)).}
    \end{align}
 The solution to \cref{eq:u diff ode} is given explicitly by
    \begin{align*}
        u(t) = \frac{1}{4\sqrt{\lambda_{i}}}e^{2\sqrt{\lambda_{i}}t} - \frac{1}{4\sqrt{\lambda_{i}}}e^{-2\sqrt{\lambda_{i}}t} > 0, \quad \text{for all \( t > 0 \) and \( i = 1,\, 2\).}
    \end{align*}
    We now proceed to solve the following problem:
    \begin{align}\label{eq:y diff}
        \begin{cases}
           \frac{\mathrm d^{2}}{\mathrm dt^{2}} y (t) = 4 \lambda_{i} y(t),  \qquad  & t\geq 0,\\
           y (0) = \left\langle \eta_{0}^{+}, \eta_{0}^{+} \right\rangle_{\mathcal{H}_{2}}, \\ \frac{\mathrm d}{\mathrm dt} y(0) = 2 \left\langle L^{+} \eta_{0}^{+}, \eta_{0}^{+} \right\rangle_{\mathcal{H}_{2}} &
        \end{cases} \qquad \qquad    \text{(for \( i = 1,\, 2\)).}
    \end{align}
 The solution to \cref{eq:y diff} is given explicitly by
    \begin{align*}
        y (t) = E_{i} (t) &\coloneqq \frac{\left\langle \eta_{0}^{+}, \eta_{0}^{+} \right\rangle_{\mathcal{H}_{2}} + \frac{1}{\sqrt{ \lambda_{i}}} \left\langle L^{+} \eta_{0}^{+}, \eta_{0}^{+} \right\rangle_{\mathcal{H}_{2}}}{2} e^{ 2 \sqrt{ \lambda_{i} }t} \\ & \quad + \frac{\left\langle \eta_{0}^{+}, \eta_{0}^{+} \right\rangle_{\mathcal{H}_{2}} - \frac{1}{\sqrt{ \lambda_{i}}}\left\langle L^{+} \eta_{0}^{+}, \eta_{0}^{+} \right\rangle_{\mathcal{H}_{2}}}{2} e^{-{2 \sqrt{ \lambda_{i} }t}},\qquad \text{(for \( i = 1,\, 2\)).}
    \end{align*}
    For \( t > 0\) and \( i = 1,\, 2\), we have
    \begin{align*}
        E_{i} (t) &= \frac{1}{2} \left\langle \eta_{0}^{+}, \eta_{0}^{+} \right\rangle_{ \mathcal{H}_{2} } \cosh( 2 \sqrt{\lambda_{i}} t ) + \frac{1}{2 \sqrt{\lambda_{i}}} \left\langle L^{+} \eta_{0}^{+}, \eta_{0}^{+} \right\rangle_{ \mathcal{H}_{2} } \sinh ( 2 \sqrt{\lambda_{i}} t ) \\ & > 0,
    \end{align*}
    which remains strictly positive provided that  
    \[ 
    \langle \eta_0^+, \eta_0^+\rangle_{\mathcal H_2} > 0 \quad \text{(i.\,e.,  $\eta_{0}^{+} \neq 0$)} \qquad \text{and} \qquad \left\langle L^{+} \eta_{0}^{+}, \eta_{0}^{+} \right\rangle_{ \mathcal{H}_{2} } \geq - \sqrt{ \lambda_{i} } \left\langle \eta_{0}^{+}, \eta_{0}^{+} \right\rangle_{ \mathcal{H}_{2}} >0.\]

Since 
    \begin{align*}
        \cosh (x) = \frac{ e^{x} + e^{- x} }{2},\qquad  \sinh (x) = \frac{ e^{x} - e^{- x} }{2},
    \end{align*}
    and both \( x \mapsto \cosh (x) \) and \(x \mapsto \tfrac{\sinh (x)}{x}\) are strictly increasing functions for \( x > 0\), we deduce that,  for all \( t > 0 \),
    \begin{align*}
        E_{2} (t) - E_{1} (t) & = \frac{1}{2} \left\langle \eta_{0}^{+}, \eta_{0}^{+} \right\rangle_{ \mathcal{H}_{2} } \cosh (2 \sqrt{\lambda_{2}} t) + \frac{1}{2 \sqrt{\lambda_{2} }} \left\langle L^{+} \eta_{0}^{+}, \eta_{0}^{+} \right\rangle_{ \mathcal{H}_{2} } \sinh (2 \sqrt{\lambda_{2} } t)\\
        & \quad - \frac{1}{2} \left\langle \eta_{0}^{+}, \eta_{0}^{+} \right\rangle_{ \mathcal{H}_{2} } \cosh (2 \sqrt{\lambda_{1}} t ) - \frac{1}{2 \sqrt{ \lambda_{1} } } \left\langle L^{+} \eta_{0}^{+}, \eta_{0}^{+} \right\rangle_{ \mathcal{H}_{2} } \sinh (2 \sqrt{ \lambda_{1} } t)\\
        & = \frac{1}{2} \left\langle \eta_{0}^{+}, \eta_{0}^{+} \right\rangle_{ \mathcal{H}_{2} } \left( \cosh (2 \sqrt{ \lambda_{2} } t) - \cosh ( 2 \sqrt{ \lambda_{1} } t )\right)\\
        & \quad + \frac{1}{2} \left\langle L^{+} \eta_{0}^{+}, \eta_{0}^{+} \right\rangle_{\mathcal{H}_{2}} \left( \frac{1}{\sqrt{\lambda_{2}}} \sinh (2 \sqrt{\lambda_{2}} t ) - \frac{1}{\sqrt{\lambda_{1}}} \sinh (2 \sqrt{\lambda_{1}} t ) \right) \\ &> 0,
    \end{align*}
    provided that \(\langle \eta_{0}^{+}, \eta_0^+\rangle_{\mathcal H_2} > 0\) (i.\,e., $\eta_0^+ \neq 0$) and  \( \left\langle L^{+} \eta_{0}^{+}, \eta_{0}^{+} \right\rangle_{ \mathcal{H}_{2} } \geq 0 \). As a consequence, we obtain
    \begin{align*}
        E_{1}^{1/2} (t) < \|\eta^{+} (t)\|_{\mathcal{H}_{2}} < E_{2}^{1/2} (t), \qquad \text{ for \(0 < t < +\infty\).}
    \end{align*}

    \end{proof}

  \begin{remark}[Higher excited states on the torus]\label{rk:arbitrarykappa}
Dealing with case of arbitrary \(\kappa \ge 3\) is much more involved. Since the linear structure of \(\eta^{+}\) is identical to that of the De Gregorio model, we illustrate the difficulties using \cref{eq:DG1}. 

Let us define
\begin{align*}
    \omega (t,\theta) &\coloneqq -\sin(\kappa\,\theta) + \eta (t,\theta) , && t >0, \ \theta \in \T, \\
    u(t,\theta) &\coloneqq \frac{1}{\kappa}\sin(\kappa\,\theta) + v (t,\theta), && t >0, \ \theta \in \T, 
\end{align*}
Then \cref{eq:DG1} can be rewritten as a partial differential equation in \(\eta\) as follows:
\begin{align}\label{eq:eta k}
\begin{cases}
\partial_t \eta  = \{\frac{1}{\kappa} \eta + v, \sin(\kappa\theta)\} +  \{\eta, v\}, & t >0, \ \theta \in \T, \\
\partial_{\theta}v  = H\eta, & t >0, \ \theta \in \T,
\end{cases} 
\end{align}
For any fixed integer \(\kappa\geq 1\), define
\begin{align*}
    \mathcal{H}_{\kappa} \coloneqq  \left\{\eta \in H^{1} (\mathbb{T}) : \ \eta \text{ is odd} \quad \text{ and } \quad  \int_{\mathbb{T}} \frac{|\partial_{\theta}\eta|^{2}}{|\sin\left(\frac{\kappa}{2}\theta\right)|^{2}} \, \mathrm d \theta < \infty\right\},
\end{align*}
which is a Hilbert space equipped with the inner product
\begin{equation}
\langle \xi, \eta \rangle_{\mathcal H_{\kappa}} \coloneqq \int_{-\pi}^{\pi} \rho_{\kappa} \, \partial_\theta \xi \, \partial_\theta \eta \, \mathrm d\theta, \qquad \text{where } \ \rho_{\kappa} \coloneqq \frac{1}{4\pi \sin^2 (\frac{\kappa}{2}\theta)}.
\end{equation}
The family \(\{e_{\kappa,l}\}_{l\in\mathbb{Z}_{+}}\) forms a complete orthonormal basis for \(\mathcal{H}_{\kappa}\), where
\begin{align*}
    e_{\kappa,l} \coloneqq \frac{\sin((l+\kappa)\theta)}{l+\kappa} - \frac{\sin(l\theta)}{l}, \qquad l\geq 1.
\end{align*}
For a fixed integer \(\kappa\,\geq 1\), direct computation yields 
\begin{align*}
    L_{\kappa} \sin(l\theta) = \begin{cases} A_{\kappa,l} \sin((l+\kappa)\theta) + B_{\kappa,l} \sin((l-\kappa)\theta), & \text{if } l > \kappa,\\
    0, &\text{if } l = \kappa, \\ 
     A_{\kappa,l} \sin((l+\kappa)\theta) - B_{\kappa,l} \sin((\kappa-l)\theta), & \text{if } l < \kappa,
    \end{cases}
\end{align*}
where
\begin{align*}
    A_{\kappa,l} \coloneqq - \frac{(l-\kappa)^{2}}{2\kappa \, l}, \qquad B_{\kappa,l} \coloneqq \frac{(l+\kappa)(l-\kappa)}{2\kappa \, l}.
\end{align*}
It follows that 
\begin{align*}
    L_{\kappa} e_{\kappa,l} & = \frac{1}{l+\kappa} L_{\kappa} \sin((l+\kappa)\theta) - \frac{1}{l} L_{\kappa} \sin(l\theta)\\
    & = \frac{A_{\kappa,l+\kappa}}{l+\kappa} \sin((l+2\kappa)\theta) + \frac{B_{\kappa,l+\kappa}}{l+\kappa} L_{\kappa} \sin(l\theta) - \frac{A_{\kappa,l}}{l} \sin((l+\kappa)\theta) - \frac{B_{\kappa,l}}{l} L_{\kappa} \sin((l-\kappa)\theta)\\
    & = - \frac{l^{2} (l+2\kappa)}{2k(l+\kappa)^{2}} \frac{\sin((l+2\kappa)\theta)}{l+2\kappa} + \frac{l^{2} (l+2\kappa)}{2k(l+\kappa)^{2}} \frac{\sin(l\theta)}{l} + \frac{(l-\kappa)^{2} (l+\kappa)}{2\kappa \, l^{2}} \frac{\sin((l+\kappa)\theta)}{l+\kappa} \\
    & \quad - \frac{(l-\kappa)^{2} (l+\kappa)}{2\kappa \, l^{2}} \frac{\sin((l-\kappa)\theta)}{l-\kappa}\\
    & =  - \frac{l^{2} (l+2\kappa)}{2k(l+\kappa)^{2}} \bigg(\frac{\sin((l+2\kappa)\theta)}{l+2\kappa} - \frac{\sin((l+\kappa)\theta)}{l+\kappa}\bigg) - \frac{l^{2} (l+2\kappa)}{2k(l+\kappa)^{2}}\frac{\sin((l+\kappa)\theta)}{l+\kappa}\\
    & \quad + \frac{l^{2} (l+2\kappa)}{2k(l+\kappa)^{2}} \frac{\sin(l\theta)}{l} + \frac{(l-\kappa)^{2} (l+\kappa)}{2\kappa \, l^{2}} \frac{\sin((l+\kappa)\theta)}{l+\kappa}\\
    & \quad + \frac{(l-\kappa)^{2} (l+\kappa)}{2\kappa \, l^{2}} \bigg( \frac{\sin(l\theta)}{l} - \frac{\sin((l-\kappa)\theta)}{l-\kappa}\bigg) - \frac{(l-\kappa)^{2} (l+\kappa)}{2\kappa \, l^{2}} \frac{\sin(l\theta)}{l}\\
    & = - \frac{l^{2} (l+2\kappa)}{2k(l+\kappa)^{2}} e_{\kappa,l+\kappa} + \bigg[\frac{(l-\kappa)^{2} (l+\kappa)}{2\kappa \, l^{2}} - \frac{l^{2} (l+2\kappa)}{2k(l+\kappa)^{2}}\bigg] e_{\kappa,l} + \frac{(l-\kappa)^{2} (l+\kappa)}{2\kappa \, l^{2}} e_{\kappa,l-\kappa}
\end{align*}
i.\,e., 
\begin{align*}
    L_{\kappa} e_{\kappa,l} & \coloneqq d_{\kappa,l+\kappa} e_{\kappa,l+\kappa} + \big( d_{\kappa,l} - d_{\kappa,l+\kappa}\big) e_{\kappa,l} + d_{\kappa,l} e_{\kappa,l-\kappa}, \qquad \text{for \(l\geq 1\),}
\end{align*}
 where
\begin{align*}
    d_{\kappa,l} \coloneqq \frac{(l-\kappa)^{2} (l+\kappa)}{2\kappa \, l^{2}}.
\end{align*}
Formally, if the solution \(\eta  \) to \cref{eq:eta k} can be written as \(\eta (t,\theta) = \sum_{l\geq 1} \eta_{\kappa,l} (t) e_{\kappa,l}\), then
\begin{align*}
    \frac{1}{2} \frac{\mathrm{d}}{\mathrm{d} t} \sum_{l \geq 1} \left(\eta_{\kappa,l} (t)\right)^{2} = \left\langle L_{\kappa} \eta, \eta\right\rangle_{\mathcal{H}_{\kappa}} = \sum_{l\geq 1} \left(d_{\kappa,l} - d_{\kappa,l+\kappa} \right) \left( \eta_{\kappa,l} (t)\right)^{2},
\end{align*}
where
\begin{align*}
    d_{\kappa,l} - d_{\kappa,l+\kappa} & = \frac{(l-\kappa)^{2} (l+\kappa)}{2\kappa \, l^{2}} - \frac{l^{2} (l+2\kappa)}{2k(l+\kappa)^{2}}\\
    & = \frac{-l^{4} - 2\kappa \, l^{3} - 2\kappa^{2}l^{2} + \kappa^{3}l + \kappa^{4}}{2l^{2}(l+\kappa)^{2}}\\
    & = - \frac{1}{2} + \frac{\kappa^{2}(-l^{2} + \kappa \, l + \kappa^{2})}{2l^{2}(l+\kappa)^{2}}. 
\end{align*}
Since the coefficients \(d_{\kappa,l} - d_{\kappa,l+\kappa}\) change signs for fixed integer \(\kappa\geq 3\), a finer analysis is required. 

We will address this problem (for both the De Gregorio model \cref{eq:DG1} and the one-dimensional MHD model \cref{eq:MHD1}) in a forthcoming work.  
  \end{remark}

  \begin{remark}[The case $q \neq0$]\label{rk:qneq0} 
  
  When considering the range of parameters \cref{ass:q=0} (namely,  the case $q \neq 0$), the linear equation for \(\eta^{-}\) in \cref{eq:perturbations} and the nonlinearity in \cref{def:nonlinear2} have additional terms to deal with.
  
  The main difficulties arise from the study of the operator \(Q\) in \cref{def:operator2}.  We stress that a similar difficulty arises in the proof of \cite[Theorem 1.2]{sun2025} (cf.~\cite[Section 5]{sun2025}); however, unlike the approach in \cite{sun2025}, where weighted estimates are performed in the function space 
   \begin{align*}
       \mathcal{H}_1 \coloneqq \bigg\{f\in H^{1} (\T) : \ f(0) = 0 \quad \text{and} \quad\int_{\T} \frac{|\partial_\theta f|^{2}}{|\sin(\theta/2)|^{2}} \, \mathrm{d}\theta < \infty \bigg\},
   \end{align*}
   (where the singular weight is given by \((\sin(\theta/2))^{-2}\)),
   our study avoids the classical kernel space decomposition \(\operatorname{span}_{\mathbb{R}}\{\sin(\theta)\} \oplus \mathcal{H}_1\) and instead considers the boundedness of the operator \(Q\) in the weighted space \(\mathcal{H}_{2}\) (which is endowed with a stronger singular weight function, \((\sin(\theta))^{-2}\), featuring singularities at both endpoints \(\theta = 0\) and \(\theta = \pm \pi\)), requiring a more delicate analysis of the nonlocal operator structure.

  We begin by rewriting the operator $Q$ in a form that makes its nonlocal structure explicit
  \begin{align*}
      Qf & \coloneqq \sin(2\theta) Hf + \cos(2\theta) f\\
      & = \cos(2\theta) f + \sin(2\theta) \left[Hf(\theta) - Hf(-\theta)\right] + \sin(2\theta) Hf(-\theta).
  \end{align*}
  To handle the difference \(Hf(\theta) - Hf(-\theta)\), we observe that
  \begin{align*}
      Hf(\theta) - Hf(-\theta) & = \mathrm{p.\,v.\,}\frac{1}{2\pi}\int_{-\pi}^{\pi} \cot\left(\frac{\theta - \varphi}{2}\right) f(\varphi) \mathrm{d}\varphi - \mathrm{p.\,v.\,} \frac{1}{2\pi}\int_{-\pi}^{\pi} \cot\left(\frac{-\theta - \varphi}{2}\right) f(\varphi) \, \mathrm{d}\varphi \\
      & = \mathrm{p.\,v.\,} \frac{1}{2\pi}\int_{-\pi}^{\pi} \cot\left(\frac{\theta - \varphi}{2}\right) f(\varphi) \mathrm{d}\varphi + \mathrm{p.\,v.\,} \frac{1}{2\pi}\int_{-\pi}^{\pi} \cot\left(\frac{\theta + \varphi}{2}\right) f(\varphi) \,\mathrm{d}\varphi\\
      & = \mathrm{p.\,v.\,}\frac{1}{2\pi} \int_{-\pi}^{\pi} \left[\cot\left(\frac{\theta - \varphi}{2}\right) + \cot\left(\frac{\theta + \varphi}{2}\right)\right] f(\varphi) \,\mathrm{d}\varphi\\
      & = \frac{1}{2\pi} \sin(\theta) \, \mathrm{p.\,v.\,} \int_{-\pi}^{\pi} \frac{\sin(\varphi)}{\sin\left(\frac{\theta + \varphi}{2}\right) \sin\left(\frac{\theta - \varphi}{2}\right)} \frac{f(\varphi)}{\sin(\varphi)} \,\mathrm{d}\varphi.
  \end{align*}
   Applying this identity, together with H\"{o}lder's inequality, \cref{le:hardy type}, and \cref{le:hilbert estimate}, we obtain 
  \begin{align*}
      \int_{\T} \bigg|\frac{\partial_{\theta} (Qf)}{\sin(\theta)}\bigg|^{2} \mathrm{d}\theta & \leq \int_{\T} \bigg|\frac{\cos(2\theta) \partial_{\theta}f - 2\sin(2\theta) f}{\sin(\theta)}\bigg|^{2} \, \mathrm{d}\theta + \int_{\T} \bigg|\frac{2\cos(2\theta) H f(-\theta) + \sin(2\theta) H(\partial_{\theta} f) (-\theta)}{\sin(\theta)}\bigg|^{2} \, \mathrm{d}\theta\\
      & \quad + \int_{\T} \bigg|\frac{2\cos(2\theta)\left[Hf(\theta) - Hf(-\theta)\right] + \sin(2\theta) \left[H (\partial_{\theta} f) (\theta) - H (\partial_{\theta}f) (-\theta)\right]}{\sin(\theta)}\bigg|^{2} \, \mathrm{d}\theta\\
      & \lesssim \int_{\T} \bigg|\frac{\cos(2\theta) \partial_\theta f - 2\sin(2\theta) f}{\sin(\theta)}\bigg|^{2} \mathrm{d}\theta + \int_{\T} \bigg| \frac{\sin(2\theta) H \partial_\theta f}{\sin(\theta)}\bigg|^{2} \,\mathrm{d}\theta\\
      & \quad + \int_{\T} \bigg|\frac{\cos(2\theta) \left[Hf(\theta) - Hf(-\theta)\right]}{\sin(\theta)}\bigg|^{2} \,\mathrm{d}\theta \\
      & \lesssim \|f\|_{\mathcal{H}_{2}}^{2} + \|H \partial_\theta f\|_{L^{2}}^{2} + \bigg\|\frac{f}{\sin(\theta)}\bigg\|_{L^{\infty}}^{2} + \int_{\T} \bigg| \mathrm{p.\,v.\,} \int_{-\pi}^{\pi} \frac{\sin(\varphi)}{\sin\left(\frac{\theta + \varphi}{2}\right) \sin\left(\frac{\theta - \varphi}{2}\right)} \frac{f(\varphi)}{\sin(\varphi)} \, \mathrm{d}\varphi \bigg|^{2} \, \mathrm{d}\theta \\
       & \lesssim \|f\|_{\mathcal{H}_{2}}^{2} + \|H \partial_\theta f\|_{L^{2}}^{2} + \bigg\|\frac{f}{\sin(\theta)}\bigg\|_{L^{\infty}}^{2} \\
      & \lesssim \|f\|_{\mathcal{H}_{2}}^{2}, \qquad \text{for any \(f\in \mathcal{H}_{2}\)},
  \end{align*}
where we used that 
  \begin{align*}
      \int_{\T} \bigg| \mathrm{p.\,v.\,} \int_{-\pi}^{\pi} \frac{\sin(\varphi)}{\sin\left(\frac{\theta + \varphi}{2}\right) \sin\left(\frac{\theta - \varphi}{2}\right)} \frac{f(\varphi)}{\sin(\varphi)} \, \mathrm{d}\varphi \bigg|^{2} \, \mathrm{d}\theta \lesssim \bigg\|\frac{f}{\sin(\theta)}\bigg\|_{L^{\infty}}^{2}
  \end{align*}
  noticing that
  \begin{align*}
      [-\pi,\pi] \ni\theta \mapsto \mathrm{p.\,v.\,} \int_{-\pi}^{\pi} \frac{\sin(\varphi)}{\sin\left(\frac{\theta + \varphi}{2}\right) \sin\left(\frac{\theta - \varphi}{2}\right)} \, \mathrm{d}\varphi
  \end{align*}
is well-defined and finite.\footnote{~Indeed, the integrand exhibits singularities at \(\varphi = \pm \theta\), where the denominator vanishes, and from these points, it is smooth and bounded. A local expansion near these points shows that the integrand behaves like a Cauchy-type kernel of order \(\frac{1}{(\varphi \mp \theta)}\), but such singularities are integrable in the \textit{Cauchy principal value} sense (cf.~\cite[Chapter 3, Section 8]{MR4703940}).}
  
  Combining the above estimates, we conclude that
  \begin{align*}
      \|Q f\|_{\mathcal{H}_{2}} \lesssim \|f\|_{\mathcal{H}_{2}},
  \end{align*}
  which implies that \(Q: \mathcal{H}_{2} \to \mathcal{H}_{2}\) is bounded.
   
Using the boundedness of the operator \(Q\), we can obtain an exponential decay estimate for $\eta^-$:
\begin{align*}
    \frac{1}{2} \frac{\mathrm{d}}{\mathrm{d}t} \|\eta^{-}(t)\|_{\mathcal{H}_{2}}^{2} & = \left\langle L^{-}\eta^{-},\eta^{-}\right\rangle_{\mathcal{H}_{2}} +  2q \left\langle Q\eta^{-},\eta^{-}\right\rangle_{\mathcal{H}_{2}}\\
    & < - \frac{1}{2} \|\eta^{-}(t)\|_{\mathcal{H}_{2}}^{2} + 2q \left\langle Q\eta^{-},\eta^{-}\right\rangle_{\mathcal{H}_{2}}\\
    & < - \left(\frac{1}{2} - 2C_{Q}q\right) \|\eta^{-}(t)\|_{\mathcal{H}_{2}}^{2} = -\frac{\delta}{2} \|\eta^{-}(t)\|_{\mathcal{H}_{2}}^{2},
\end{align*}
where \(\delta \coloneqq 1 - 4C_{Q} q > 0\) and \(C_{Q} > 0\) is a constant depending on the operator \(Q\). 

The linear analysis for $\eta^+$ remains unchanged when $q \neq 0$. The nonlinearity in \cref{def:nonlinear2} will include some additional terms \(\eta^{-} H\eta^{+} - \eta^{+} H\eta^{-}\). To analyze the local well-posedness of \cref{eq:perturbations}, we again adopt the Galerkin's method. The additional nonlinear terms give rise to the following expression
\begin{align*}
    -\sqrt{\pi} \partial_{\theta} (\eta^{-}_{n} H\eta^{+}_{n} - \eta^{+}_{n} H\eta^{-}_{n} ) & = \frac{1}{2\sin(\theta)} \big(\partial_{\theta}\eta^{+}_{n} \partial_{\theta}^{2} v^{-}_{n} +\eta^{+}_{n} \partial_{\theta}v^{-}_{n} - \partial_{\theta}\eta^{-}_{n} \partial_{\theta}v^{+}_{n} - \eta^{-}_{n} \partial_{\theta}^{2} v^{+}_{n}\big)\\
    & = u_{n}^{+} \partial_{\theta}^{2} v_{n}^{-} + f_{n}^{+} \partial_{\theta} v_{n}^{-} - u_{n}^{-} \partial_{\theta}^{2} v_{n}^{+} + f_{n}^{-} \partial_{\theta} v_{n}^{+},
\end{align*}
where \(u_{n}^{\pm} \coloneqq - \sqrt{\pi} \rho_2^{1/2} \partial_{\theta} \eta_{n}^{\pm}\) and \(f_{n}^{\pm} \coloneqq \sqrt{\pi} \rho_2^{1/2} \eta_{n}^{\pm}\). These satisfy
\begin{align*}
    \|\partial_{\theta} f_{n}^{\pm} \|_{L^{2}} \leq \|\partial_{\theta}^{m} u_{n}^{\pm}\|_{L^{2}}, \qquad \text{for \(m\geq 1\).}
\end{align*}
this can be done similarly to the estimates performed later in \cref{ssec:nonlinear-wp}.

In summary, if 
\[
a=1, \qquad p = 1 - q, \qquad 0 < q < \frac{1}{4 C_{Q}},
\]
the results and methods developed in this paper remain applicable. 

However, the question of stability or instability for the more general parameter range in \cref{ass:parameters} remains open and is left for future investigation.
\end{remark}

\section{(Local) existence and uniqueness for the nonlinear problem}
\label{ssec:nonlinear-wp}

Having analyzed the linear problem, now we turn to the proof of the local well-posedness of the classical solution to \cref{eq:perturbations2} with odd initial data. The proof of \cref{le:nonlinear existence} follows from Galerkin’s method, similar to the approach used in \cref{the:existence}. For convenience, we continue to denote the approximate solutions by \(\eta_{n}^{\pm}, v_{n}^{\pm} \). The main task is to establish uniform estimates for nonlinear terms. Since the derivative of the nonlinear term is
\begin{align*}
    & - \sqrt{\pi} \rho_2^{1/2} \partial_{\theta} N_{1} (\eta_{n}^{\pm}) = \frac{1}{2 \sin(\theta)} \left( \partial_{\theta}^{2} \eta_{n}^{+} v_{n}^{+} - \eta_{n}^{+} \partial_{\theta}^{2} v_{n}^{+} + \eta_{n}^{-} \partial_{\theta}^{2} v_{n}^{-} - \partial_{\theta}^{2} \eta_{n}^{-} v_{n}^{-} \right),\\
    & - \sqrt{\pi} \rho_2^{1/2} \partial_{\theta} N_{2} (\eta_{n}^{\pm}) = \frac{1}{2 \sin(\theta)} \left( \partial_{\theta}^{2} \eta_{n}^{-} v_{n}^{+} - \eta_{n}^{-} \partial_{\theta}^{2} v_{n}^{+} + \eta_{n}^{+} \partial_{\theta}^{2} v_{n}^{-} - \partial_{\theta}^{2} \eta_{n}^{+} v_{n}^{-} \right),
\end{align*}	
which exhibit a stronger singularity, to derive uniform \(H^{m}\) estimates for \(\rho_2^{1/2} \partial_{\theta} \eta_{n}^{\pm} \), we rewrite \(\rho_2^{1/2} \partial_{\theta} N_{i} (\eta_{n}^{\pm})\)  (with \(i = 1,\, 2\)) using the explicit form of approximate solutions \(\eta_{n}^{\pm}\).

\begin{proof}[Proof of \cref{le:nonlinear existence}]
\uline{Step 1.} \emph{Construction of approximate solutions.}
    Following the notation introduced in \cref{sec:existence-linearized}, for a fixed positive integer $n$, we define
    \begin{align}
        & \eta_{n}^{\pm} (t, \theta) = \sum_{k = 1}^{n}\eta_{2, k}^{\pm} (t) e_{2, k},\label{eq:nonlinear etan}\\
        & \partial_{\theta} v_{n}^{\pm} (t, \theta) = H \eta_{n}^{\pm} (t, \theta).\label{eq:nonlinear vn}
    \end{align}
Our objective is to determine the coefficients \(\eta_{2,k}^{\pm} (t)\) such that
\begin{align}
       \left\langle \partial_{t} \partial_{\theta} \eta_{n}^{+}, \partial_{\theta} e_{2, k} \rho_2 \right\rangle - \left\langle \partial_{\theta} L^{+} \eta_{n}^{+}, \partial_{\theta} e_{2, k} \rho_2 \right\rangle &= \left\langle \partial_{\theta} N_{1}(\eta_{n}^{\pm}), \partial_{\theta} e_{2, k} \rho_2 \right\rangle ,\label{eq:nonlinear etak positive}\\
        \left\langle \partial_{t} \partial_{\theta} \eta_{n}^{-}, \partial_{\theta} e_{2, k} \rho_2 \right\rangle - \left\langle \partial_{\theta} L^{-} \eta_{n}^{-}, \partial_{\theta} e_{2, k} \rho_2 \right\rangle &= \left\langle \partial_{\theta} N_{2}(\eta_{n}^{\pm}), \partial_{\theta} e_{2, k} \rho_2 \right\rangle,\label{eq:nonlinear etak negative}
\\
         \left\langle \partial_{\theta} \eta_{0}^{+}, \partial_{\theta} e_{2, k} \right\rangle &= \eta_{2, k}^{+} (0),\label{eq:nonlinear etan positive initial data}\\
         \left\langle \partial_{\theta} \eta_{0}^{-}, \partial_{\theta} e_{2, k} \right\rangle &= \eta_{2, k}^{-} (0).\label{eq:nonlinear etan negative initial data}
    \end{align}
    We define
    \begin{align*}
       & u_{n}^{\pm} (t, \theta) \coloneqq - \sqrt{\pi} \rho_2^{1/2} \partial_{\theta} \eta_{n}^{\pm} = \sum_{k=1}^{n} \eta_{2,k}^{\pm} (t) \sin((k+1) \theta),\\
       & L_{1}^{\pm} (u_{n}^{\pm}) \coloneqq - \sqrt{\pi} \rho_2^{1/2} \partial_{\theta} L^{\pm} \eta_{n}^{\pm},
    \end{align*}
    and 
    \begin{align*}
            \mathcal{N}_{1} (u_{n}^{\pm}) & \coloneqq - \sqrt{\pi} \rho_2^{1/2} \partial_{\theta} N_{1}(\eta_{n}^{\pm}) \\
         & = \sqrt{\pi} \rho_2^{1/2} \left( \partial_{\theta}^{2} \eta_{n}^{+} v_{n}^{+} - \eta_{n}^{+} \partial_{\theta}^{2} v_{n}^{+} + \eta_{n}^{-} \partial_{\theta}^{2} v_{n}^{-} - \partial_{\theta}^{2} \eta_{n}^{-} v_{n}^{-} \right),     \\
            \mathcal{N}_{2} (u_{n}^{\pm}) & \coloneqq - \sqrt{\pi} \rho_2^{1/2} \partial_{\theta} N_{2}(\eta_{n}^{\pm})\\
        & = \sqrt{\pi} \rho_2^{1/2} \left( \partial_{\theta}^{2} \eta_{n}^{-} v_{n}^{+} - \eta_{n}^{-} \partial_{\theta}^{2} v_{n}^{+} + \eta_{n}^{+} \partial_{\theta}^{2} v_{n}^{-} - \partial_{\theta}^{2} \eta_{n}^{+} v_{n}^{-} \right).
    \end{align*}
Since \(\partial_{\theta} v_{n}^{\pm} = H \eta_{n}^{\pm} \), we have
\begin{align}\label{eq:def gn}
    \begin{aligned}
        g_{n}^{\pm} (t, \theta) & \coloneqq \sqrt{\pi} \rho_2^{1/2} \partial_{\theta}^{2} v_{n}^{\pm} = \frac{1}{2 \sin(\theta)} \sum_{k=1}^{n} \eta_{2,k}^{\pm} (t) \left[ \sin((k+2)\theta) - \sin(k\theta) \right]\\
        & = \sum_{k=1}^{n} \eta_{2,k}^{\pm} (t) \cos((k+1)\theta),
    \end{aligned}
\end{align}
which implies
\begin{align}\label{eq:estimate-gn}
    \| g_{n}^{\pm} (t) \|_{H^{m}} = \| u_{n}^{\pm} (t) \|_{H^{m}}, \qquad \text{for every \(m \geq 0\).}
\end{align}
Next, we represent \(\sqrt{\pi} \rho_2^{1/2} v_{n}^{\pm} \) as an explicit Fourier series.	Since \(\eta_{n}^{\pm}\) and \(v_{n}^{\pm}\) are odd functions, combining with \cref{eq:nonlinear etan} and \cref{eq:nonlinear vn}, we obtain
\begin{align}\label{eq:nonlinear vn expression}
    v_{n}^{\pm} (t, \theta) = \sum_{k=1}^{n+2} \frac{1}{k^{2}} \left( - \eta_{2,k-2}^{\pm} (t) + \eta_{2,k}^{\pm} (t) \right) \sin (k\theta),
\end{align}
where \( \eta_{2,-1}^{\pm} (t), \, \eta_{2,0}^{\pm} (t),\, \eta_{2,n+1}^{\pm} (t)\), and \(\eta_{2,n+2}^{\pm} (t) \) are understood to be zero. 

Next, we analyze the expression \(\frac{\sin(k\theta)}{\sin(\theta)}\).
If \(k\) is odd, that is, \( k = 2l - 1\), with \(l = 1, 2, \ldots\), we have
\begin{align*}
    \frac{\sin ((2l - 1)\theta)}{\sin(\theta)} & =\frac{\sin(\theta) - \sin(\theta) + \sin(3\theta) - \cdots - \sin((2l - 3)\theta) + \sin((2l - 1)\theta)}{\sin(\theta)}\\
    & = 1 + 2\sum_{j=1}^{l-1} \cos(2j\theta).
\end{align*}
On the other hand, if \( k = 2l\), with \(l = 1, 2, \ldots \), then
\begin{align*}
    \frac{\sin(k\theta)}{\sin(\theta)} & = \frac{\sin(2\theta) - \sin(2\theta) + \sin(4\theta) - \cdots - \sin((2l - 2)\theta) + \sin(2l\theta)}{\sin(\theta)}\\
    & = 2 \sum_{j=1}^{l} \cos((2j-1)\theta).
\end{align*}
Without loss of generality, we consider the case \(n = 2l -1\) in \cref{eq:nonlinear vn expression} and define
\begin{align*}
       h_{n}^{\pm} (t, \theta) & \coloneqq \sqrt{\pi} \rho_2^{1/2} v_{n}^{\pm}\\
    & = \frac{1}{2} \sum_{k=1}^{l+1} \frac{- \eta_{2, 2k - 3}^{\pm} (t) + \eta_{2, 2k - 1}^{\pm} (t) }{(2k - 1)^{2}} + \sum_{k=1}^{l} \cos(2k\theta) \left( \sum_{j\geq k}^{l} \frac{ - \eta_{2, 2j - 1}^{\pm} (t) + \eta_{2, 2j + 1}^{\pm} (t) }{(2j + 1)^{2}}\right)\\
    & \quad + \sum_{k=1}^{l} \cos((2k - 1)\theta) \left( \sum_{j\geq k}^{l} \frac{ - \eta_{2, 2j - 2}^{\pm} (t) + \eta_{2, 2j}^{\pm} (t) }{(2j)^{2}}\right). 
\end{align*}
Thus, \(\mathcal{N}_{1}\) and \(\mathcal{N}_{2}\) can be expressed as
\begin{align*}
    \mathcal{N}_{1} &= - 2 \sin(\theta) h_{n}^{+} \partial_{\theta} u_{n}^{+} - 2 \cos(\theta) h_{n}^{+} u_{n}^{+} - \eta_{n}^{+} g_{n}^{+} + \eta_{n}^{-} g_{n}^{-} + 2 \cos(\theta) u_{n}^{-} h_{n}^{-} + 2 \sin(\theta) h_{n}^{-} \partial_{\theta} u_{n}^{-},\\
    \mathcal{N}_{2} &= - 2 \sin(\theta) h_{n}^{+} \partial_{\theta} u_{n}^{-} - 2 \cos(\theta) h_{n}^{+} u_{n}^{-} - \eta_{n}^{-} g_{n}^{+} + \eta_{n}^{+} g_{n}^{-} + 2 \cos(\theta) u_{n}^{+} h_{n}^{-} + 2 \sin(\theta) h_{n}^{-} \partial_{\theta} u_{n}^{+}.
\end{align*}
We note that
\begin{align*}
    &\| \partial_{\theta}^{m} h_{n}^{\pm} \|_{L^{2}}^{2} \\ & = \sum_{k=1}^{l}(2k)^{2m} \left( \sum_{j\geq k}^{l} \frac{ -\eta_{2, 2j - 1 }^{\pm} (t)  + \eta_{2, 2j + 1}^{\pm} (t) }{(2j + 1)^{2}} \right)^{2} + \sum_{k=1}^{l} (2k - 1)^{2m} \left( \sum_{j\geq k}^{l} \frac{- \eta_{2, 2j - 2}^{\pm} (t) + \eta_{2, 2j}^{\pm} (t) }{(2j)^{2}}\right)^{2}\\
    & \leq \sum_{k=1}^{l} (2k)^{2m} \sum_{j \geq k}^{l} \left( (\eta_{2, 2j -1})^{2}\cdot\frac{1}{(2j+1)^{4}}\right) + \sum_{k=1}^{l}(2k-1)^{2m}\sum_{j\geq k}^{l}\left((\tilde{\eta}_{2j-2})^{2}\cdot\frac{1}{(2j)^{4}}\right)\\
    & \leq \sum_{j \geq 1}^{l} \sum_{k \leq j} \left( ( \eta_{2, 2j - 1}^{\pm} (t) )^{2} \cdot \frac{(2k)^{2m}}{(2j + 1)^{4}} \right) + \sum_{j \geq 1}^{l} \sum_{k \leq j} \left( ( \eta_{2, 2j - 2 }^{\pm} (t) )^{2} \cdot \frac{(2k - 1)^{2m}}{(2j)^{4}}\right)\\
    & \leq \sum_{j \geq 1}^{l} ( \eta_{2, 2j - 1}^{\pm} (t) )^{2} \cdot j \cdot \frac{(2j)^{2m}}{(2j + 1)^{4}} + \sum_{j \geq 1}^{l}( \eta_{2, 2j - 2}^{\pm} (t) )^{2} \cdot j \cdot \frac{(2j - 1)^{2m}}{(2j)^{4}}\\
    & \leq \sum_{k=1}^{n} (k+1)^{2m - 4 + 1} (\eta_{2,k}^{\pm} (t) )^{2} \\
    & \leq \sum_{k=1}^{n} (k+1)^{2m - 2}( \eta_{2, k}^{\pm} (t) )^{2}\\
    & = \| \partial_{\theta}^{m-1} u_{n}^{\pm} \|^2_{L^{2}},
\end{align*}
which yields 
\begin{align}\label{eq:estimate-hn}
    \| \partial_{\theta}^{m} h_{n}^{\pm} \|_{L^{2}} \leq \| \partial_{\theta}^{m-1} u_{n}^{\pm} \|_{L^{2}}, \qquad \text{for \( m \geq 1 \).}
\end{align}
The ordinary differential equations \crefrange{eq:nonlinear etak positive}{eq:nonlinear etak negative} are equivalent to
\begin{align}
    \left\langle \partial_{t} u_{n}^{+}, \sin((k+1)\theta) \right\rangle - \left\langle L_{1}^{+} (u_{n}^{+}), \sin((k+1)\theta) \right\rangle &= \left\langle \mathcal{N}_{1}(u_{n}^{+}), \sin((k+1)\theta) \right\rangle,\label{eq:nonlinear un positive ode}\\
    \left\langle \partial_{t} u_{n}^{-}, \sin((k+1)\theta) \right\rangle - \left\langle L_{1}^{-} (u_{n}^{\pm}), \sin((k+1)\theta) \right\rangle &= \left\langle \mathcal{N}_{2}(u_{n}^{\pm}), \sin((k+1)\theta) \right\rangle.\label{eq:nonlinear un negative ode}
\end{align}
By the standard existence theory for ordinary differential equations, there exists a \( T (k) > 0\), which may depend on \(k\), and a unique set of absolutely continuous functions \( \eta_{2,k}^{\pm} (t)\), with \( k = 1, 2, \ldots, n\), satisfying \crefrange{eq:nonlinear etak positive}{eq:nonlinear etak negative} and \crefrange{eq:nonlinear etan negative initial data}{eq:nonlinear etan positive initial data} for \( 0 \leq t \leq T (k)\).

   \uline{Step 2.} \emph{Energy estimates.} We need several estimates on $u_n^\pm$ and their derivatives, which are proven below.
   
\uline{Step 2\,a.} \emph{\(L^{2}\)-estimates of \(u_{n}^{\pm}\).} Multiplying \cref{eq:nonlinear un positive ode} by \(\eta_{2,k}^{+} (t)\) and \cref{eq:nonlinear un negative ode} by \(\eta_{2,k}^{-} (t)\), and then summing up over \(k = 1, 2, \ldots, n\), we obtain
    \begin{align*}
        \frac{1}{2} \frac{\mathrm d}{\mathrm dt} \left( \| u_{n}^{+} \|_{L^{2}}^{2} + \| u_{n}^{-} \|_{L^{2}}^{2} \right)  = \left\langle L_{1}^{+} (u_{n}^{+}), u_{n}^{+} \right\rangle + \left\langle L_{1}^{-} (u_{n}^{-}), u_{n}^{-} \right\rangle + \left\langle \mathcal{N}_{1}(u_{n}^{\pm}), u_{n}^{+} \right\rangle + \left\langle \mathcal{N}_{2}(u_{n}^{\pm}), u_{n}^{-} \right\rangle.
    \end{align*}
 We note that \(v_{n}^{\pm}\) are odd and periodic functions, and hence satisfy \(v_{n}(\pi)^{\pm} = v_{n}^{\pm} (0) \).  Moreover, since \(\sin(\theta) \geq \frac{2}{\pi} \min\{\theta, \pi - \theta\}\) for \(\theta \in [0, \pi]\), it follows that 
 \begin{align*}
     \|h_{n}^{\pm}\|_{L^{\infty}} = \bigg\| \frac{v_{n}^{\pm}}{\sin(\theta)} \bigg\|_{L^{\infty}} \lesssim \|\partial_{\theta} v_{n}^{\pm} \|_{L^{\infty}}
 \end{align*} 
 by Lagrange's mean value theorem. Then, applying Poincar\'{e}'s inequality and \cref{le:hilbert estimate}, we obtain
 \begin{align}\label{eq:estimate-hn infty}
 \|h_{n}^{\pm}\|_{L^{\infty}} \lesssim \|\partial_{\theta} v_{n}^{\pm}\|_{L^{\infty}} \| H (\partial_{\theta} \eta_{n}^{\pm})\|_{L^{2}} \lesssim \|\partial_{\theta} \eta_{n}^{\pm}\|_{L^{2}} \lesssim \|u_{n}^{\pm}\|_{L^{2}}
 \end{align}
 and
 \begin{align}\label{eq:estimate-sinhn}
     \|2 \partial_{\theta} (\sin(\theta) h_{n}^{\pm})\|_{L^{\infty}} = \bigg\| \partial_{\theta} v_{n}^{\pm} - \cos(\theta) \frac{v_{n}^{\pm}}{\sin(\theta)}\bigg\|_{L^{\infty}} \lesssim \|\partial_{\theta} v_{n}^{\pm}\|_{L^{\infty}} \lesssim \|u_{n}^{\pm}\|_{L^{2}}.
 \end{align}
Applying \cref{eq:estimate-gn}, \cref{eq:estimate-hn}, together with \crefrange{eq:estimate-hn infty}{eq:estimate-sinhn}, we obtain
\begin{align*}
   & \left\langle \mathcal{N}_{1}(u_{n}^{\pm}), u_{n}^{+} \right\rangle + \left\langle \mathcal{N}_{2}(u_{n}^{\pm}), u_{n}^{-} \right\rangle \\
   & \lesssim \left(\|\partial_{\theta} (\sin(\theta) h_{n}^{+}) \|_{L^{\infty}} + \| h_{n}^{+}\|_{L^{\infty}} + \| h_{n}^{-}\|_{L^{\infty}} + \|\eta_{n}^{+}\|_{L^{\infty}} + \|\eta_{n}^{-}\|_{L^{\infty}} \right) \left( \| u_{n}^{+} \|_{L^{2}}^{2} + \| u_{n}^{-} \|_{L^{2}}^{2} \right) \\
   & \quad + \left\langle 2 \sin(\theta) h_{n}^{-} \partial_{\theta} u_{n}^{-}, u_{n}^{+} \right\rangle + \left\langle 2 \sin(\theta) h_{n}^{-} \partial_{\theta} u_{n}^{+}, u_{n}^{-} \right\rangle \\
   & \lesssim \left(\| u_{n}^{+} \|_{L^{2}} + \| u_{n}^{-} \|_{L^{2}} + \| \partial_{\theta} \eta_{n}^{+} \|_{L^{2}} + \| \partial_{\theta} \eta_{n}^{-} \|_{L^{2}} \right) \left( \| u_{n}^{+} \|_{L^{2}}^{2} + \| u_{n}^{-} \|_{L^{2}}^{2} \right)\\
   & \quad + \left\langle 2 \sin(\theta) h_{n}^{-} \partial_{\theta} u_{n}^{-}, u_{n}^{+} \right\rangle + \left\langle 2 \sin(\theta) h_{n}^{-} \partial_{\theta} u_{n}^{+}, u_{n}^{-} \right\rangle \\
   & \lesssim  \left( \| u_{n}^{+} \|_{L^{2}}^{2} + \| u_{n}^{-} \|_{L^{2}}^{2} \right)^{3/2} + \left\langle 2 \sin(\theta) h_{n}^{-} \partial_{\theta} u_{n}^{-}, u_{n}^{+} \right\rangle + \left\langle 2 \sin(\theta) h_{n}^{-} \partial_{\theta} u_{n}^{+}, u_{n}^{-} \right\rangle. 
\end{align*}
For the last two terms in the above estimate, we compute
\begin{align*}
    & \quad \left\langle 2 \sin(\theta) h_{n}^{-} \partial_{\theta} u_{n}^{-}, u_{n}^{+} \right\rangle + \left\langle 2 \sin(\theta) h_{n}^{-} \partial_{\theta} u_{n}^{+}, u_{n}^{-} \right\rangle \\
    & = \int_{\T} \left( 2 \sin(\theta) h_{n}^{-} \partial_{\theta} u_{n}^{-} u_{n}^{+} + 2 \sin(\theta) h_{n}^{-} \partial_{\theta} u_{n}^{+}, u_{n}^{-} \right) \mathrm d \theta\\
    & = \int_{\T} 2 \sin(\theta) h_{n}^{-} \partial_{\theta} \left( u_{n}^{-} u_{n}^{+}\right) \mathrm d \theta = - \int_{\T} 2 \partial_{\theta} \left(\sin(\theta) h_{n}^{-}\right) u_{n}^{-} u_{n}^{+} \mathrm d \theta \\
    & \lesssim \|u_{n}^{+}\|_{L^{2}} \|u_{n}^{-}\|_{L^{2}}^{2}.  
\end{align*}
Thus, we obtain
\begin{align*}
    \frac{\mathrm d}{\mathrm dt} \left( \| u_{n}^{+} \|_{L^{2}}^{2} + \| u_{n}^{-} \|_{L^{2}}^{2} \right) \leq C  \left( \| u_{n}^{+} \|_{L^{2}}^{2} + \| u_{n}^{-} \|_{L^{2}}^{2} \right) + C  \left( \| u_{n}^{+} \|_{L^{2}}^{2} + \| u_{n}^{-} \|_{L^{2}}^{2} \right)^{3/2},
\end{align*}
where \(C > 0\) is a constant independent of \(n\).

\uline{Step 2\,b.} \emph{{\(H^{m}\)-estimates of \(u_{n}^{\pm}\) (for \(m \geq 1\)).}}
Multiplying equation \cref{eq:nonlinear un positive ode} by \( (k + 1)^{2m} \eta_{2,k}^{+} (t)\) and equation \cref{eq:nonlinear un negative ode} by \( (k + 1)^{2m} \eta_{2,k}^{-} (t)\), summing up over \( k = 1, 2,\ldots, n \), and applying integration by parts, we obtain
    \begin{align*}
        \frac{1}{2} \frac{\mathrm d}{\mathrm dt} \left( \| \partial_{\theta}^{m} u_{n}^{+} (t) \|_{L^{2}}^{2} + \| \partial_{\theta}^{m} u_{n}^{-} (t)\|_{L^{2}}^{2} \right) & = \left\langle \partial_{\theta}^{m} L_{1}^{+} (u_{n}^{+}), \partial_{\theta}^{m} u_{n}^{+} \right\rangle + \left\langle \partial_{\theta}^{m} L_{1}^{-} (u_{n}^{-}), \partial_{\theta}^{m} u_{n}^{-} \right\rangle\\
        & \quad + \left\langle \partial_{\theta}^{m} \mathcal{N}_{1} (u_{n}^{\pm}), \partial_{\theta}^{m} u_{n}^{+} \right\rangle + \left\langle \partial_{\theta}^{m} \mathcal{N}_{2} (u_{n}^{\pm}), \partial_{\theta}^{m} u_{n}^{-} \right\rangle.
    \end{align*}
 We decompose the nonlinear terms into three components
 \begin{align*}
     J_{1} &\coloneqq \left\langle \partial_{\theta}^{m} \left( - 2 \sin(\theta) h_{n}^{+} \partial_{\theta} u_{n}^{+} - 2 \cos(\theta) h_{n}^{+} u_{n}^{+} - \eta_{n}^{+} g_{n}^{+} + \eta_{n}^{-} g_{n}^{-} + 2 \cos(\theta) u_{n}^{-} h_{n}^{-} \right), \partial_{\theta}^{m} u_{n}^{+} \right\rangle,\\ 
     J_{2} &\coloneqq \left\langle \partial_{\theta}^{m} \left(  - 2 \sin(\theta) h_{n}^{+} \partial_{\theta} u_{n}^{-} - 2 \cos(\theta) h_{n}^{+} u_{n}^{-} - \eta_{n}^{-} g_{n}^{+} + \eta_{n}^{+} g_{n}^{-} + 2 \cos(\theta) u_{n}^{+} h_{n}^{-} \right), \partial_{\theta}^{m} u_{n}^{-} \right\rangle,\\
     J_{3} &\coloneqq 2 \left\langle \partial_{\theta}^{m} \left(\sin(\theta) h_{n}^{-} \partial_{\theta} u_{n}^{-} \right), \partial_{\theta}^{m} u_{n}^{+}\right\rangle + 2 \left\langle \partial_{\theta}^{m} \left(\sin(\theta) h_{n}^{-} \partial_{\theta} u_{n}^{+} \right), \partial_{\theta}^{m} u_{n}^{-}\right\rangle.\\
 \end{align*}
 We now turn to the estimates of \(J_{3}\).  Each of the terms in \(J_{3}\) can be written in the form \(\left\langle \partial_{\theta}^{m_{1}} u_{n}^{\pm} \partial_{\theta}^{m_{2}} h_{n}^{-} \partial_{\theta}^{m_{3}} \sin(\theta), \partial_{\theta}^{m} u_{n}^{\mp}\right\rangle\) with \(m_{1} + m_{2} + m_{3} = m\). In the case \(m_{1} = m\), we integrate by parts and apply H\"{o}lder's inequality together with \cref{eq:estimate-sinhn} to deduce
 \begin{align*}
     & \left\langle\sin(\theta) h_{n}^{-} \partial_{\theta}^{m+1} u_{n}^{-} , \partial_{\theta}^{m} u_{n}^{+} \right\rangle + \left\langle\sin(\theta) h_{n}^{-} \partial_{\theta}^{m+1} u_{n}^{+} , \partial_{\theta}^{m} u_{n}^{-} \right\rangle\\
     & = \int_{\T} \left(\sin(\theta) h_{n}^{-} \partial_{\theta}^{m+1} u_{n}^{-} \partial_{\theta}^{m} u_{n}^{+} + \sin(\theta) h_{n}^{-} \partial_{\theta}^{m+1} u_{n}^{+} \partial_{\theta}^{m} u_{n}^{-} \right) \mathrm d \theta\\
     & = \int_{\T} \sin(\theta) h_{n}^{-} \partial_{\theta} \left( \partial_{\theta}^{m} u_{n}^{+} \partial_{\theta}^{m} u_{n}^{-}\right) \\
     & = - \int_{\T} \partial_{\theta} \left( \sin(\theta) h_{n}^{-} \right) \partial_{\theta}^{m} u_{n}^{+} \partial_{\theta}^{m} u_{n}^{-}  \mathrm d \theta \\
     & \leq \| \partial_{\theta} \left( \sin(\theta) h_{n}^{-} \right) \|_{L^{\infty}} \| \partial_{\theta}^{m} u_{n}^{+} \|_{L^{2}} \| \partial_{\theta}^{m} u_{n}^{-} \|_{L^{2}} \\
     & \lesssim \|u_{n}^{-}\|_{L^{2}} \| \partial_{\theta}^{m} u_{n}^{+} \|_{L^{2}} \| \partial_{\theta}^{m} u_{n}^{-} \|_{L^{2}}.
 \end{align*}
 If the indices satisfy \(m_{1} + m_{2} + m_{3} = m \), with \(0 \leq m_{1} \leq m-1, 1 \leq m_{2}\), then, by H\"{o}lder's inequality, Poincar\'{e}'s inequality, and \cref{eq:estimate-hn}, we obtain
 \begin{align*}
     \left\langle \partial_{\theta}^{m_{1}} u_{n}^{\pm} \partial_{\theta}^{m_{2}} h_{n}^{-} \partial_{\theta}^{m_{3}} \sin(\theta), \partial_{\theta}^{m} u_{n}^{\mp}\right\rangle & \lesssim \| \partial_{\theta}^{m_{1}} u_{n}^{\pm}\|_{L^{\infty}} \| \partial_{\theta}^{m_{2}} h_{n}^{-} \|_{L^{2}} \| \partial_{\theta}^{m} u_{n}^{\mp}\|_{L^{2}}\\
     & \lesssim \| \partial_{\theta}^{m_{1} + 1} u_{n}^{\pm}\|_{L^{2}} \| \partial_{\theta}^{m_{2} - 1} u_{n}^{-} \|_{L^{2}} \| \partial_{\theta}^{m} u_{n}^{\mp}\|_{L^{2}}\\
     & \lesssim \| u_{n}^{\pm} \|_{H^{m}} \| u_{n}^{-} \|_{H^{m}} \| u_{n}^{\mp} \|_{H^{m}} .
 \end{align*} 
 If the indices satisfy \(m_{1} + m_{2} = m\) with \(m_{1} \leq m-1\), then by \cref{eq:estimate-hn infty}, we have
 \begin{align*}
     \left\langle h_{n}^{-} \partial_{\theta}^{m_{1}} u_{n}^{\pm} \partial_{\theta}^{m_{2}} \sin(\theta), \partial_{\theta}^{m} u_{n}^{\mp} \right\rangle & \lesssim \| h_{n}^{-}\|_{L^{\infty}} \|\partial_{\theta}^{m_{1}} u_{n}^{\pm}\|_{L^{2}} \|\partial_{\theta}^{m} u_{n}^{\mp}\|_{L^{2}} \\
     & \lesssim \|u_{n}^{-}\|_{L^{2}} \|\partial_{\theta}^{m_{1}} u_{n}^{\pm}\|_{L^{2}} \|\partial_{\theta}^{m} u_{n}^{\mp}\|_{L^{2}}.
 \end{align*}
 Combining the above estimates, we conclude that
 \begin{align*}
     J_{3} \leq C \left( \| u_{n}^{+} \|_{H^{m}}^{2} + \| u_{n}^{-} \|_{H^{m}}^{2} \right)^{3/2}.
 \end{align*}
  As for the estimates of \(J_{1}\) and \(J_{2}\), using H\"{o}lder's inequality, Poincar\'{e}'s inequality, together with \cref{eq:estimate-gn}, \cref{eq:estimate-hn}, \cref{eq:estimate-hn infty}, and \cref{eq:estimate-sinhn}, we obtain
  \begin{align*}
      J_{1} + J_{2} \leq C \left( \| u_{n}^{+} \|_{H^{m}}^{2} + \| u_{n}^{-} \|_{H^{m}}^{2} \right)^{3/2}.
  \end{align*}
 Define 
 \begin{align*}
     I_{m,n}^{2} (t) \coloneqq \| u_{n}^{+} (t) \|_{H^{m}}^{2} + \| u_{n}^{-} (t) \|_{H^{m}}^{2},
 \end{align*}
 Then it follows that
 \begin{align}\label{eq:C12}
    \frac{\mathrm d}{\mathrm dt} I_{m,n} (t) \leq C_{1} I_{m,n} (t) + C_{2} I_{m,n}^{2} (t),
\end{align}
where \(C_{1}, C_{2} > 0 \) are constants independent of \(n\). A direct computation yields
\begin{align*}
    I_{m,n} (t) \leq \frac{C_{{1}}}{C_{{2}}(1-e^{C_{{1}}t})I_{m} (0) + C_{{1}}} I_{m} (0) e^{C_{{1}}t}, \qquad \text{for \( 0 \leq t < T^{*}\),}
\end{align*}
where
\[T^{*} \coloneqq \frac{1}{C_{{1}}} \ln\left(1 + \frac{C_{{1}}}{C_{{2}} I_{m} (0) }\right),\]
and \(m\) is any fixed integer with \( m \geq 1\). We choose
\begin{align*}
    T_{0} \coloneqq \frac{1}{C_{{1}}} \ln\left(1 + \frac{C_{{1}}}{2C_{{2}} I_{m} (0)}\right),
\end{align*}
where \(I_{m}^{2} (0) = \| u_{0}^{+} \|_{H^{m}}^{2} + \| u_{0}^{-} \|_{H^{m}}^{2}\). It then follows that
\begin{align}\label{eq:estimate-nonlinear un}
    \sup_{0 \leq t \leq T_{0}} \| u_{n}^{\pm} (t) \|_{H^{m}} \leq\ 4 e^{C_{{1}}T_{{0}}} \left( \| u_{0}^{+} \|_{H^{m}} + \| u_{0}^{-} \|_{H^{m}}\right).
\end{align}

\uline{Step 2\,c.} \emph{{\(H^{m-1}\)-estimates of \(\partial_{t} u_{n}^{\pm}\) (for \(m \geq 1\)).}}
To derive the \(H^{m-1}\)-estimate of \(\partial_{t}u_{n}^{\pm}\), we multiply equation \cref{eq:nonlinear un positive ode} by \((k+1)^{2m-2} \frac{\mathrm d}{\mathrm dt} \eta_{2,k}^{+} (t)\), equation \cref{eq:nonlinear etak negative} by \((k+1)^{2m-2} \frac{\mathrm d}{\mathrm dt} \eta_{2,k}^{-} (t)\), then sum up \( k = 1, 2, \ldots,n\). After applying integration by parts, we obtain
\begin{align*}
    \| \partial_{t} \partial_{\theta}^{m-1} u_{n}^{+} \|_{L^{2}} + \| \partial_{t} \partial_{\theta}^{m-1} u_{n}^{+} \|_{L^{2}} & = \left\langle \partial_{\theta}^{m-1} L_{1}^{+}(u_{n}^{+}), \partial_{t} \partial_{\theta}^{m-1} u_{n}^{+} \right\rangle + \left\langle \partial_{\theta}^{m-1} L_{1}^{-}(u_{n}^{-}), \partial_{t} \partial_{\theta}^{m-1} u_{n}^{-} \right\rangle\\
    & \quad + \left\langle \partial_{\theta}^{m-1} \mathcal{N}_{1} (u_{n}^{\pm}), \partial_{t} \partial_{\theta}^{m-1} u_{n}^{+} \right\rangle + \left\langle \partial_{\theta}^{m-1} \mathcal{N}_{2} (u_{n}^{\pm}), \partial_{t} \partial_{\theta}^{m-1} u_{n}^{-} \right\rangle.
\end{align*}
For the nonlinear terms \(\partial_{\theta}^{m-1} \mathcal{N}_{i}(u_{n}^{\pm})\), with \(i = 1, \, 2\), applying \cref{eq:estimate-nonlinear un} yields
\begin{align*}
    \|\partial_{\theta}^{m-1} \mathcal{N}_{i}(u_{n}^{\pm})\|_{L^{2}} \lesssim I_{m}^{2} (t),
\end{align*}
which implies
\begin{align}\label{eq:estimate-nonlinear ptun}
     \sup_{0 \leq t \leq T_{0}} \| \partial_{t} \partial_{\theta}^{m-1} u_{n}^{\pm} (t) \|_{L^{2}} \leq C e^{2C_{{1}}T_{{0}}} \left( \| u_{0}^{+} \|_{H^{m}}^{2} + \| u_{0}^{-} \|_{H^{m}}^{2} \right),
\end{align}
where \(C\) is a positive constant independent of \(n\).

\uline{Step 2\,d.} \emph{{\(H^{m-2}\)-estimates of \(\partial_{t} u_{n}^{\pm}\) (for \(m \geq 2\)).}}
Proceeding similarly, we multiply equation \cref{eq:nonlinear un positive ode} by \((k+1)^{2m-4} \frac{\mathrm d^{2}}{\mathrm dt^{2}} \eta_{2,k}^{+} (t)\), equation \cref{eq:nonlinear etak negative} by \((k+1)^{2m-4} \frac{\mathrm d^{2}}{\mathrm dt^{2}} \eta_{2,k}^{-} (t)\), then sum up over \( k = 1, 2, \ldots,n\). After applying integration by parts, we obtain
\begin{align*}
    \| \partial_{t} \partial_{\theta}^{m-2} u_{n}^{+} \|_{L^{2}} + \| \partial_{t} \partial_{\theta}^{m-2} u_{n}^{+} \|_{L^{2}} & = \left\langle \partial_{\theta}^{m-2} L_{1}^{+}(u_{n}^{+}), \partial_{t} \partial_{\theta}^{m-2} u_{n}^{+} \right\rangle + \left\langle \partial_{\theta}^{m-2} L_{1}^{-}(u_{n}^{-}), \partial_{t} \partial_{\theta}^{m-2} u_{n}^{-} \right\rangle\\
    & \quad + \left\langle \partial_{\theta}^{m-1} \mathcal{N}_{1} (u_{n}^{\pm}), \partial_{t} \partial_{\theta}^{m-2} u_{n}^{+} \right\rangle + \left\langle \partial_{\theta}^{m-2} \mathcal{N}_{2} (u_{n}^{\pm}), \partial_{t} \partial_{\theta}^{m-2} u_{n}^{-} \right\rangle.
\end{align*}
By combining this with \cref{eq:estimate-nonlinear un}, we obtain
\begin{align}\label{eq:estimate-nonlinear pttun}
    \sup_{0 \leq t \leq T_{0}} \| \partial_{t}^{2} \partial_{\theta}^{m-2} u_{n}^{\pm} (t) \|_{L^{2}} \leq C e^{2C_{{1}}T_{{0}}} \left( \| u_{0}^{+} \|_{H^{m}}^{2} + \| u_{0}^{-} \|_{H^{m}}^{2} \right),
\end{align}
where \(C\) is a positive constant independent of \(n\).

\uline{Step 3.} \emph{{Existence and uniqueness.}}
By combing the uniform estimates established in  \cref{eq:estimate-nonlinear un}, \cref{eq:estimate-nonlinear ptun} and \cref{eq:estimate-nonlinear pttun}, the local existence interval \([0,T(k)]\) derived in Step 1 of \cref{ssec:nonlinear-wp} can be extended to \([0, T_{0}]\). In particular, there exists a subsequence \(\{u_{n_{l}}\}_{n_{l} = 1}^{\infty} \subset \{u_{n}\}_{n = 1}^{\infty}\), such that
\begin{align*}
         u_{n_{l}}^{\pm} &\rightharpoonup u^{\pm}\ &&\text{weakly in } L^{2}( [0, T_{0}] ; H^{m}(\T));\\
         \partial_{t} u_{n_{l}}^{\pm} &\rightharpoonup \partial_{t} u^{\pm}\ &&\text{weakly in } L^{2}( [0, T_{0}] ; H^{m-1}(\T));\\
         \partial_{t}^{2} u_{n_{l}}^{\pm} &\rightharpoonup \partial_{t}^{2} u^{\pm}\ &&\text{weakly in } L^{2}( [0, T_{0}] ; H^{m-2}(\T)).  
     \end{align*}
     Since \( \partial_{\theta} \eta_{n}^{\pm} = - 2 \sin(\theta) u_{n}^{\pm} \) and \( \eta^{\pm} \) are odd, it follows that
     \begin{align*}
         \eta_{n_{l}}^{\pm} &\rightharpoonup \eta^{\pm} &&\text{weakly in } L^{2}( [0, T_{0}] ; H^{m+1}(\T));\\
         \partial_{t} \eta_{n_{l}}^{\pm} &\rightharpoonup \partial_{t} \eta^{\pm} &&\text{weakly in } L^{2}( [0, T_{0}] ; H^{m}(\T));\\
         \partial_{t}^{2} \eta_{n_{l}}^{\pm} &\rightharpoonup \partial_{t}^{2} \eta^{\pm} &&\text{weakly in } L^{2}( [0, T_{0}] ; H^{m-1}(\T)).  
     \end{align*}
     The remainder of the argument follows along the same lines as in Step 3 of the proof of \cref{the:existence} and is therefore omitted. Furthermore, combining \cref{eq:estimate-nonlinear un}, \cref{eq:estimate-nonlinear ptun}, and \cref{eq:estimate-nonlinear pttun}, we conclude that
     \begin{align*}
         \sup_{0 \leq t \leq T_{0}} \| u^{\pm} (t)\|_{H^{m}}, \ \sup_{0 \leq t \leq T_{0}} \| \partial_{t} u^{\pm} (t)\|_{H^{m}}, \ \sup_{0 \leq t \leq T_{0}} \| \partial_{t}^{2} u^{\pm} (t)\|_{H^{m}} \leq C (T_{0}) \delta,
     \end{align*}
thus completing the proof of \cref{le:nonlinear existence}.
\end{proof}

\section{Instability for the nonlinear problem}
\label{sec:instability-nonlinear}

In this section, we prove \cref{the:instability_nonlinear}. To this end, we first construct a solution to the linearized problem, together with a family of solutions to the nonlinear system, each satisfying carefully chosen structural properties. Based on these constructions, we then employ a contradiction argument to demonstrate the existence of a nonlinear solution exhibiting the instability described in \cref{the:instability_nonlinear}. Our proof follows the strategy of  \cite{JJN2013}.

\begin{proof}[Proof of \cref{the:instability_nonlinear}]
     \uline{Step 1.} \emph{{Construction of a solution to the linearized problem.}}
    By \crefrange{the:existence}{the:instability_linearized}, we can construct a classical solution \(\eta^{\pm}\) to the linearized system \cref{eq:linearized}. We assume that the initial data satisfy the assumptions of \cref{the:instability_linearized}, and, in addition, satisfy
    \begin{align*}
        0 \leq \left\langle L^{+} \eta_{0}^{+}, \eta_{0}^{+} \right\rangle_{\mathcal{H}_{2}} \leq \sqrt{\lambda_{1}} \left\langle\eta_{0}^{+}, \eta_{0}^{+} \right\rangle_{\mathcal{H}_{2}}.
    \end{align*}
    It follows that
    \begin{align*}
        \| \eta^{+} (t) \|_{\mathcal{H}_{2}}  > \frac{1}{2} \| \eta_{0}^{+} \|_{\mathcal{H}_{2}} e^{\sqrt{\lambda_{1}} t}.
    \end{align*}
    To simplify the notation, we set \( u (t, \theta) \coloneqq \rho_2^{1/2} \partial_{\theta} \eta^{\pm} (t, \theta) \), so that \( \| u \|_{L^{2}} = \| \rho_2^{1/2} \partial_{\theta} \eta^{\pm} \|_{L^{2}} = \| \eta^{\pm} \|_{\mathcal{H}_{2}}\). We further introduce the rescaled functions
    \begin{align*}
	\check{\eta}^{\pm} (t, \theta) \coloneqq \frac{\delta \eta^{\pm} (t, \theta)}{I_{m} (0)},
    \end{align*}
    and
    \begin{align}\label{eq:def-tilde u}
        \check{u}^{\pm} (t, \theta) \coloneqq \rho_2^{1/2} \partial_{\theta} \check{\eta}^{\pm} (t, \theta) = \frac{\delta u^{\pm} (t, \theta)}{I_{m} (0)},
    \end{align}
    it implies that
    \begin{align*}
        \| \check{u}^{\pm} (0) \|_{H^{m}} = \delta.
    \end{align*}
    It follows directly from the linearity of the system \cref{eq:linearized} that
     \(\check{u}^{\pm} 
     \) remains a classical solution of the linearized problem, inheriting all the regularity and structural properties of the original solution \(\eta^{\pm} \), which means that
     \begin{align*}
         \| \check{\eta}^{+}(t)\|_{\mathcal{H}_{2}} > \frac{1}{2} \|\check{\eta}^{+} (0) \|_{\mathcal{H}_{2}} e^{\sqrt{\lambda}_{1} t}, \qquad \text{for $t>0$.}
     \end{align*}
 Let
     \begin{align*}
         t_{K} \coloneqq \frac{1}{\sqrt{\lambda_{1}}} \ln \left( \frac{4K\delta}{\|\tilde{u}^{+} (0) \|_{L^{2}}} \right).
     \end{align*}
     Then, we obtain
     \begin{align}\label{eq:assumption}
         \| \check{u}^{+} (t_{K})\|_{L^{2}} = \| \check{\eta}^{+} (t_{K})\|_{\mathcal{H}_{2}} > \frac{1}{2} \|\check{\eta}^{+} (0) \|_{\mathcal{H}_{2}} e^{\sqrt{\lambda}_{1} t_{K}} = 2K \delta.
     \end{align}

     \uline{Step 2.} \emph{{Construction of a solution to the nonlinear problem.}}
     Based on the initial data \( \check{u} (0)\) of the solution \(\check{u}  \) defined in \cref{eq:def-tilde u}, we proceed to construct a family of solutions to the perturbed nonlinear problem \cref{eq:perturbations2}. Specifically, we define
     \begin{align*}
         \bar{\eta}_{\varepsilon}^{+} (0, \theta) \coloneqq \varepsilon \tilde{\eta}^{+} (0, \theta)
     \end{align*}
     and set
    \begin{align*}
        \bar{u}_{\varepsilon}^{+} (0, \theta) \coloneqq \rho_2^{1/2} \partial_{\theta} \bar{\eta}_{\varepsilon}^{+} (0, \theta) = \varepsilon \tilde{u} (0, \theta), \qquad \text{for }  0 < \varepsilon < 1,
    \end{align*}
    so that
    \begin{align*}
        \| \bar{u}_{\varepsilon}^{+} (0, \theta) \|_{H^{m}} = \varepsilon \delta < \delta, \qquad \text{for } m > 3.
    \end{align*}
    By applying \cref{le:nonlinear existence}, we obtain that there exists a constant \(\varepsilon_{1} > 0\) such that, for all \( \varepsilon \in (0, \varepsilon_1)\), the nonlinear system \cref{eq:perturbations2} admits a classical solution \(\bar{\eta}_{\varepsilon}^{\pm}\) on \((0, T_{\varepsilon})\), where
    \begin{align*}
        T_{\varepsilon} \coloneqq \frac{1}{C_{{1}}} \ln \left( 1 + \frac{C_{{1}}}{2C_{{2}} \varepsilon \delta} \right) > t_{K}.
    \end{align*}
    Moreover, obtain the following uniform-in-time estimates:
    \begin{align}\label{eq:estimate-nonlinear solution}
        \sup_{0\leq t\leq t_{K}} \|\bar{u}^{\pm}_{\varepsilon} (t) \|_{H^{m}}, \sup_{0\leq t\leq t_{K}} \|\partial_{t} \bar{u}^{\pm}_{\varepsilon} (t) \|_{H^{m-1}}, \sup_{0\leq t\leq t_{K}} \|\partial_{t}^{2} \bar{u}^{\pm}_{\varepsilon} (t) \|_{H^{m-2}}  \leq C (t_{K}) \delta,
    \end{align}
    where we defined \(\bar{u}^{\pm}_{\varepsilon} (t,\theta) \coloneqq \rho_2^{1/2} \partial_{\theta} \bar{\eta}^{\pm}_{\varepsilon} (t,\theta)\). 

         \uline{Step 3.} \emph{Contradiction argument.} We prove \cref{the:instability_nonlinear} by contradiction. Let us suppose that for any \(\varepsilon \in (0,\varepsilon_{1})\), the classical solution \((\bar{\eta}_{\varepsilon}^{\pm}, \bar{v}^{\pm}_{\varepsilon}) \), emanating from the initial data \(\bar{\eta}_{\varepsilon}^{\pm} (0) \), satisfies
    \begin{align*}
    \|\bar{\eta}^{+}_{\varepsilon} (t) \|_{\mathcal{H}_{2}} = \|\bar{u}_{\varepsilon}^{+} (t) \|_{L^{2}} \leq  F(I_{m,\varepsilon} (0) ),
    \end{align*}
    where \(I_{m,\varepsilon}^{2} (0) \coloneqq \|\bar{u}_{\varepsilon}^{+}(0)\|_{H^{m}}^{2} + \|\bar{u}_{\varepsilon}^{-}(0)\|_{H^{m}}^{2} \).
    Then, by applying the definition of \(F\) in \cref{eq:def_F}, we obtain 
    \begin{align*}
        \sup_{0\leq t \leq t_{k}} \|\bar{u}_{\varepsilon}^{+} (t) \|_{L^{2}} \leq \varepsilon K \delta.
    \end{align*}
    Next, we define
    \begin{align*}
        (\tilde{\eta}^{\pm}, \tilde{v}^{\pm} ) \coloneqq \frac{1}{\varepsilon} (\bar{\eta}^{\pm}_{\varepsilon}, \bar{v}^{\pm}_{\varepsilon})
    \end{align*}
    Then, \(\tilde{\eta}^{\pm}, \tilde{v}^{\pm} \) satisfy the following system:
    \begin{align}\label{eq:tilde eta}
\begin{cases}
\partial_t \tilde{\eta}^{+}  = \{\frac{1}{2} \tilde{\eta}^{+} + \tilde{v}^{+}, \sin(2\theta)\} + \varepsilon N_{1} (\tilde{\eta}^{\pm} ), & t >0, \ \theta \in \T, \\
\partial_t \tilde{\eta}^{-}  = \{\frac{1}{2} \tilde{\eta}^{-} - \tilde{v}^{-}, \sin(2\theta)\} + \varepsilon N_{2} (\tilde{\eta}^{\pm} ) , & t >0, \ \theta \in \T,\\
\partial_{\theta} \tilde{v}^{\pm}  = H\tilde{\eta}^{\pm}, & t >0, \ \theta \in \T,
\end{cases} 
\end{align}
with the initial value 
\begin{align*}
\tilde{\eta}^{\pm}_{\varepsilon}(0) = \frac{1}{\varepsilon} \bar{\eta}^{\pm}_{\varepsilon}(0) = \check{\eta} (0),
\end{align*}
where
\begin{align*}
     N_{1} (\tilde{\eta}^{\pm} ) &\coloneqq \{\tilde{\eta}^{+}, \tilde{v}^{+}\} - \{\tilde{\eta}^{-}, \tilde{v}^{-}\},\\
     N_{2} (\tilde{\eta}^{\pm} ) &\coloneqq \{\tilde{\eta}^{-}, \tilde{v}^{+}\} - \{\tilde{\eta}^{+}, \tilde{v}^{-}\}.
\end{align*}
Moreover, by \cref{eq:estimate-nonlinear solution}, the following uniform estimates hold:
\begin{align*}
    \sup_{0 \leq t \leq t_{K}} \| \tilde{u}^{\pm} (t)\|_{H^{m}}, \ \sup_{0 \leq t \leq t_{K}} \| \partial_{t} \tilde{u}^{\pm} (t)\|_{H^{m}}, \ \sup_{0 \leq t \leq t_{K}} \| \partial_{t}^{2} \tilde{u}^{\pm} (t)\|_{H^{m}} \leq C (t_{K}) \delta,
\end{align*}
which are independent of \(\varepsilon\). As a consequence, by standard compactness arguments (see \cite[Section 1.4.5]{NS2004}  for further details), we may extract a subsequence (not relabeled) from \(\{\bar{\eta}^{\pm}_{\varepsilon}\}_{\varepsilon>0}\) such that
     \begin{align*}
         \tilde{u}_{\varepsilon}^{\pm} &\rightharpoonup \tilde{u}^{\pm}\ &&\text{weakly in } L^{2}( [0, t_{K}] ; H^{m}(\T));\\
         \partial_{t} \tilde{u}_{\varepsilon}^{\pm} &\rightharpoonup \partial_{t} \tilde{u}^{\pm}\ &&\text{weakly in } L^{2}( [0, t_{K}] ; H^{m-1}(\T));\\
         \partial_{t}^{2} \tilde{u}_{\varepsilon}^{\pm} &\rightharpoonup \partial_{t}^{2} \tilde{u}^{\pm}\ &&\text{weakly in } L^{2}( [0, t_{K}] ; H^{m-2}(\T))
     \end{align*}
     and
     \begin{align}\label{eq:estimate_tilde u}
         \sup_{0\leq t\leq t_{K}} \| \tilde{u}^{\pm} (t)\|_{L^{2}} \leq K \delta, \quad \tilde{u}^{\pm} (t,\theta) \in C([0,t_{K}], H^{m}(\T)) \cap C^{2}([0,t_{K}],H^{m-2}(\T)).
     \end{align}
     In fact, since \( \partial_{\theta} \tilde{\eta}_{\varepsilon}^{\pm} = - 2 \sin(\theta) \tilde{u}_{\varepsilon}^{\pm} \) and \( \tilde{\eta}_{\varepsilon}^{\pm} \) are odd functions, we deduce that
     \begin{align*}
         \tilde{\eta}_{\varepsilon}^{\pm} &\rightharpoonup \tilde{\eta}^{\pm} &&\text{weakly in } L^{2}( [0, t_{K}] ; H^{m+1}(\T));\\
         \partial_{t} \tilde{\eta}_{\varepsilon}^{\pm} &\rightharpoonup \partial_{t} \tilde{\eta}^{\pm} &&\text{weakly in } L^{2}( [0, t_{K}] ; H^{m}(\T));\\
         \partial_{t}^{2} \tilde{\eta}_{\varepsilon}^{\pm} &\rightharpoonup \partial_{t}^{2} \tilde{\eta}^{\pm} &&\text{weakly in } L^{2}( [0, t_{K}] ; H^{m-1}(\T)).  
     \end{align*}
     Passing to the limit \(\varepsilon \to 0\) in the \cref{eq:tilde eta}, we obtain the linearized system
     \begin{align*}
       \begin{cases}
         \partial_t \tilde{\eta}^{+} = \{\frac{1}{2} \tilde{\eta}^{+} + \tilde{v}^{+}, \sin(2\theta)\}, & t >0, \ \theta \in \T, \\
         \partial_t \tilde{\eta}^{-}  = \{\frac{1}{2} \tilde{\eta}^{-} - \tilde{v}^{-}, \sin(2\theta)\}, & t >0, \ \theta \in \T,\\
         \partial_{\theta} \tilde{v}^{\pm}  = H\tilde{\eta}^{\pm}, & t >0, \ \theta \in \T,
       \end{cases} 
     \end{align*}
     Therefore, \(\tilde{\eta}^{\pm}\) is a classical solution to the linearized problem \cref{eq:linearized} for \(t \in [0,t_{K}]\), with initial data coinciding with \(\check{\eta}^{\pm} (0) \).  By the uniqueness result in \cref{the:existence}, it follows that
     \begin{align}\label{eq:tildecheck}
         \tilde{\eta}^{\pm} (t,\theta) = \check{\eta}^{\pm} (t,\theta), \qquad (t,\theta) \in   [0,t_{K}] \times \T.
     \end{align}
     Thus, combining \cref{eq:tildecheck} with \cref{eq:assumption} and \cref{eq:estimate_tilde u}, we deduce
     \begin{align*}
         2K \delta \leq \|\check{\eta}^{+} (t_{K}) \|_{\mathcal{H}_{2}} = \|\tilde{\eta}^{+} (t_{K})\|_{\mathcal{H}_{2}} \leq K \delta,
     \end{align*}
     which yields a contradiction, thus completing the proof of \cref{the:instability_nonlinear}. 
\end{proof}

\section{Global well-posedness and stability for the nonlinear problem}
\label{sec:stability-nonlinear}

In this section, we prove the stability result in \cref{the:stability_nonlinear}. We build on the discussion in \cref{rk:eta+stab}.

\begin{proof}[Proof of \cref{the:stability_nonlinear}]
    Taking the weighted \(\rho_2\)-inner product with \(\eta^{+}\) on the both sides of the first equation of the system \cref{eq:linearized}, and the weighted \(\rho_2\)-inner product with \(\eta^{-}\) on the both sides of the second equation of the system \cref{eq:linearized}, we obtain     (as in  \cref{ssec:lineareta-} and \cref{rk:eta+stab}) 
    \begin{align*}
        &\frac{1}{2} \frac{\mathrm{d}}{\mathrm{d}t} \left(\|\eta^{+} (t) \|_{\mathcal{H}_{2}}^{2} + \|\eta^{-} (t) \|_{\mathcal{H}_{2}}^{2} \right)\\
        & = \left\langle L_{1}^{+} \eta^{+}, \eta^{+} \right\rangle_{\mathcal{H}_{2}} + \left\langle L_{1}^{-} \eta^{-}, \eta^{-} \right\rangle_{\mathcal{H}_{2}} + \left\langle N_{1}, \eta^{+} \right\rangle_{\mathcal{H}_{2}} + \left\langle N_{2}, \eta^{-} \right\rangle_{\mathcal{H}_{2}}\\
        & < - \frac{3}{8} \left(\|\eta^{+} (t) \|_{\mathcal{H}_{2}}^{2} + \|\eta^{-} (t) \|_{\mathcal{H}_{2}}^{2} \right)  + \left\langle N_{1}, \eta^{+} \right\rangle_{\mathcal{H}_{2}} + \left\langle N_{2}, \eta^{-} \right\rangle_{\mathcal{H}_{2}}.
    \end{align*}

    For the nonlinear terms, we have
    \begin{align*}
        \left\langle N_{1}, \eta^{+} \right\rangle_{\mathcal{H}_{2}} + \left\langle N_{2}, \eta^{-} \right\rangle_{\mathcal{H}_{2}} 
         & = \frac{1}{4\pi} \int_{\T} \frac{\left(\eta^{+} \partial_{\theta}^{2} v^{+} - \partial_{\theta}^{2}\eta^{+} v^{+} - \eta^{-} \partial_{\theta}^{2} v^{-} + \partial_{\theta}^{2} \eta^{-} v^{-}\right)\partial_{\theta}\eta^{+}}{\sin^{2} (\theta)} \mathrm{d} \theta\\
         & \quad + \frac{1}{4\pi} \int_{\T} \frac{\left(\eta^{-} \partial_{\theta}^{2} v^{+} - \partial_{\theta}^{2}\eta^{-} v^{+} - \eta^{+} \partial_{\theta}^{2} v^{-} + \partial_{\theta}^{2} \eta^{+} v^{-}\right)\partial_{\theta}\eta^{-}}{\sin^{2} (\theta)} \mathrm{d} \theta\\
         & \coloneqq K_{1} + K_{2} + K_{3},
    \end{align*}
    where
    \begin{align*}
        & K_{1} \coloneqq - \frac{1}{4\pi} \int_{\T} \frac{v^{+} \partial_{\theta}^{2} \eta^{+} \partial_{\theta} \eta^{+}}{\sin^{2}(\theta)} - \frac{1}{4\pi} \int_{\T} \frac{v^{+} \partial_{\theta}^{2} \eta^{-} \partial_{\theta} \eta^{-}}{\sin^{2}(\theta)}, \\
        & K_{2} \coloneqq  \frac{1}{4\pi} \int_{\T} \frac{\eta^{+} \partial_{\theta}^{2} v^{+} \partial_{\theta}\eta^{+}}{\sin^{2}(\theta)} - \frac{1}{4\pi} \int_{\T} \frac{\eta^{-} \partial_{\theta}^{2} v^{-} \partial_{\theta}\eta^{+}}{\sin^{2}(\theta)} + \frac{1}{4\pi} \int_{\T} \frac{\eta^{-} \partial_{\theta}^{2} v^{+} \partial_{\theta}\eta^{-}}{\sin^{2}(\theta)} - \frac{1}{4\pi} \int_{\T} \frac{\eta^{+} \partial_{\theta}^{2} v^{-} \partial_{\theta}\eta^{-}}{\sin^{2}(\theta)},\\
        & K_{3} \coloneqq \frac{1}{4\pi} \int_{\T} \frac{v^{-} \partial_{\theta}^{2} \eta^{-} \partial_{\theta} \eta^{+}}{\sin^{2}(\theta)} + \frac{1}{4\pi} \int_{\T} \frac{v^{-} \partial_{\theta}^{2} \eta^{+} \partial_{\theta} \eta^{-}}{\sin^{2}(\theta)}.
    \end{align*}
    By applying H\"{o}lder's inequality and \cref{le:hilbert estimate}, we obtain
    \begin{align*}
        K_{1} & = \frac{1}{4\pi} \int_{\T} \partial_{\theta}\left(\frac{v^{+}}{\sin^{2}(\theta)}\right) \left( (\partial_{\theta} \eta^{+})^{2} + (\partial_{\theta} \eta^{-})^{2} \right) \mathrm{d} \theta\\
        & = \frac{1}{4\pi} \int_{\T} \partial_{\theta} v^{+} \frac{1}{\sin^{2}(\theta)} \left( (\partial_{\theta} \eta^{+})^{2} + (\partial_{\theta} \eta^{-})^{2} \right) \mathrm{d} \theta 
        \\ & \quad - \frac{1}{2\pi} \int_{\T} \frac{v^{+}}{\sin(\theta)} \frac{1}{\sin^{2}(\theta)} \left( (\partial_{\theta} \eta^{+})^{2} + (\partial_{\theta} \eta^{-})^{2} \right) \mathrm{d} \theta\\
        & \lesssim \|\partial_{\theta} v^{+}\|_{L^{\infty}} \left(\|\eta^{+} (t) \|_{\mathcal{H}_{2}}^{2} + \|\eta^{-} (t) \|_{\mathcal{H}_{2}}^{2} \right)\\
        & \lesssim \| H (\partial_{\theta} \eta^{+})\|_{L^{2}} \left(\|\eta^{+} (t) \|_{\mathcal{H}_{2}}^{2} + \|\eta^{-} (t) \|_{\mathcal{H}_{2}}^{2} \right)\\
        & \lesssim \|\partial_{\theta}\eta^{+} (t)\|_{L^{2}}\left(\|\eta^{+} (t) \|_{\mathcal{H}_{2}}^{2} + \|\eta^{-} (t) \|_{\mathcal{H}_{2}}^{2} \right)\\
        & \lesssim \left(\|\eta^{+} (t) \|_{\mathcal{H}_{2}}^{2} + \|\eta^{-} (t) \|_{\mathcal{H}_{2}}^{2} \right)^{3/2}.
    \end{align*}
    Similarly, we have
    \begin{align*}
        K_{3} & = \frac{1}{4\pi} \int_{\T} \frac{v^{-} \partial_{\theta}(\partial_{\theta} \eta^{+} \partial_{\theta} \eta^{-})}{\sin^{2}(\theta)} \mathrm{d} \theta = - \frac{1}{4\pi} \int_{\T} \partial_{\theta}\left(\frac{v^{-}}{\sin^{2}(\theta)}\right) \partial_{\theta} \eta^{+} \partial_{\theta} \eta^{-} \mathrm{d} \theta\\
        & \lesssim \|\partial_{\theta} v^{-}\|_{L^{\infty}} \left(\|\eta^{+} (t) \|_{\mathcal{H}_{2}}^{2} + \|\eta^{-} (t) \|_{\mathcal{H}_{2}}^{2} \right)\\
        & \lesssim \left(\|\eta^{+} (t) \|_{\mathcal{H}_{2}}^{2} + \|\eta^{-} (t) \|_{\mathcal{H}_{2}}^{2} \right)^{3/2}.
    \end{align*}
    Finally, applying \cref{le:hilbert estimate}, \cref{le:hardy type}, and H\"{o}lder's inequality, we obtain
    \begin{align*}
        K_{2} & \lesssim \bigg\| \frac{\eta^{+}}{\sin(\theta)}\bigg\|_{L^{\infty}} \left( \|\partial_{\theta}^{2} v^{+}\|_{L^{2}} \bigg\|\frac{\partial_{\theta}\eta^{+}}{\sin(\theta)}\bigg\|_{L^{2}} + \|\partial_{\theta}^{2} v^{-}\|_{L^{2}} \bigg\|\frac{\partial_{\theta}\eta^{-}}{\sin(\theta)}\bigg\|_{L^{2}}\right)\\
        & \quad + \bigg\| \frac{\eta^{-}}{\sin(\theta)}\bigg\|_{L^{\infty}} \left( \|\partial_{\theta}^{2} v^{-}\|_{L^{2}} \bigg\|\frac{\partial_{\theta}\eta^{+}}{\sin(\theta)}\bigg\|_{L^{2}} + \|\partial_{\theta}^{2} v^{+}\|_{L^{2}} \bigg\|\frac{\partial_{\theta}\eta^{-}}{\sin(\theta)}\bigg\|_{L^{2}}\right)\\
        & \lesssim \left(\|\eta^{+} (t) \|_{\mathcal{H}_{2}}^{2} + \|\eta^{-} (t) \|_{\mathcal{H}_{2}}^{2} \right)^{3/2}.
    \end{align*}
    Then, it follows that
    \begin{align*}
        \frac{1}{2} \frac{\mathrm{d}}{\mathrm{d}t} \left(\|\eta^{+} (t) \|_{\mathcal{H}_{2}}^{2} + \|\eta^{-} (t) \|_{\mathcal{H}_{2}}^{2} \right) \leq - \frac{1}{2} \left(\|\eta^{+} (t) \|_{\mathcal{H}_{2}}^{2} + \|\eta^{-} (t) \|_{\mathcal{H}_{2}}^{2} \right) + \left(\|\eta^{+} (t) \|_{\mathcal{H}_{2}}^{2} + C \|\eta^{-} (t) \|_{\mathcal{H}_{2}}^{2} \right)^{3/2},
    \end{align*}
    for some absolute constant \(C>0\). We proceed by employing a standard bootstrap argument. Let
    \begin{align*}
        I_{0}^{2} (t) \coloneqq \|\eta^{+} (t) \|_{\mathcal{H}_{2}}^{2} + \|\eta^{-} (t) \|_{\mathcal{H}_{2}}^{2},
    \end{align*}
    and fix \(\varepsilon > 0\) to be chosen sufficiently small later. We assume that
    \begin{align*}
        I_{0}^{2} (t) \leq 2\varepsilon^{2}, \qquad  \text{for \(t\geq 0\).}
    \end{align*}
    It then follows that
    \begin{align*}
        \left(\|\eta^{+} (t) \|_{\mathcal{H}_{2}}^{2} +  \|\eta^{-} (t) \|_{\mathcal{H}_{2}}^{2} \right)^{3/2} \leq (2C\varepsilon^{2})^{3/2} = 2^{3/2} C^{3/2} \varepsilon^{3},
    \end{align*}
    which implies
    \begin{align}\label{eq:bootstrap assumption}
        \frac{\mathrm{d}}{\mathrm{d} t} I_{0}^{2} (t)\leq - I_{0}^{2} (t) +  2^{3/2} C^{3/2} \varepsilon^{3}.
    \end{align}
   We now choose \(\varepsilon > 0\) sufficiently small such that 
    \begin{align*}
        2^{3/2} C^{3/2} \varepsilon^{3} \leq \frac{1}{2} \varepsilon^{2}.
    \end{align*}
    It then follows that whenever \(I_{0}^{2} (t) \geq \varepsilon^{2}\), we obtain
    \begin{align*}
        \frac{\mathrm{d}}{\mathrm{d} t} I_{0}^{2} (t) \leq - \frac{1}{2} I_{0}^{2} (t),
    \end{align*}
    which implies exponential decay of \(I_{0}^{2} (t)\) until it falls below \(\varepsilon^{2}\). Once in the regime where \(I_{0}^{2} (t) < \varepsilon^{2}\), we return to the general differential inequality \cref{eq:bootstrap assumption}, which yields
    \begin{align*}
        - I_{0}^{2} (t) +  2^{3/2} C^{3/2} \varepsilon^{3} \leq - \varepsilon^{2} + \frac{1}{2}\varepsilon^{2} = -\frac{1}{2}\varepsilon^{2} < 0.
    \end{align*}
    Thus, the bootstrap assumption is preserved by continuity. As a consequence, if the initial data satisfy
    \begin{align*}
      I_{0}^{2} (0) = \|\eta^{+} (0) \|_{\mathcal{H}_{2}}^{2} + \|\eta^{-} (0) \|_{\mathcal{H}_{2}}^{2} \leq \varepsilon^{2},    
    \end{align*}
     then we obtain
    \begin{align*}
        I_{0}^{2} (t) = \|\eta^{+} (t) \|_{\mathcal{H}_{2}}^{2} + \|\eta^{-} (t) \|_{\mathcal{H}_{2}}^{2}  \lesssim \varepsilon^{2}, \qquad \text{for all \(t>0\).}
    \end{align*}
 This completes the proof of \cref{the:stability_nonlinear}.
\end{proof}

\appendix

\section{Technical lemmas}
\label{app:tech}

In this appendix, we collect several preliminary results that have been used throughout the paper. 

    We start by recalling several basic properties of the Hilbert transform (whose proof can be found, for instance, in \cite[Example 9.2.3]{MR3244169}, \cite[Proposition 9.1.8]{BN1971}, \cite[Proposition 9.3.1]{BN1971}, and \cite[Theorem 9.1.3]{BN1971}, respectively).
    
    \begin{lemma}[Hilbert transform of sines and cosines] \label{le:Heo}
		For any \(a > 0\), the following formulas hold:
		\begin{align*}
			& H \sin(a\theta) = - \cos(a\theta),\\
			& H \cos(a\theta) = \sin(a\theta).
		\end{align*}
	\end{lemma}
    
	\begin{lemma}[Fourier transform and Hilbert transform]
		Let \(f \in L^{p}\), with \(1 < p < \infty\), and assume \(f\) is \(2 \pi\)-periodic. Then the Fourier coefficient of the Hilbert transform \(H f\) at frequency \(k\) is given by
		\begin{align*}
			\widehat{H f}(k) = \big(-i\, \sgn k\big)\hat{f} (k).
		\end{align*}
	\end{lemma}
	\begin{lemma}[Fourier series and Hilbert transform]
		Let \(f \in L^{1}\) and \(2 \pi\)-periodic function such that \(H f \in L^{1}\) and is also \(2 \pi\)-periodic. Then 
		\begin{align*}
			H f (x) \sim \sum_{k = - \infty}^{\infty}\big(-i\, \sgn k\big) \hat{f} (k) e^{ikx}.
		\end{align*}
	\end{lemma}
    
	\begin{lemma}[$L^p$-boundedness of the Hilbert transform] \label{le:hilbert estimate}
		The Hilbert transform \(H\) is a bounded linear operator from space \(L^{p}\) to \(L^{p}\) with \(1 < p < \infty\) and
		\begin{align*}
		    \| H f\|_{L^{p}}\leq C_{p} \| f \|_{L^{p}},
		\end{align*}
		where the constant \(C_{p} > 0\) depends only on \(p\).
	\end{lemma}

When dealing with the space $\mathcal H_2$ (whose definition is recalled in \cref{ssec:spaces}), the following Hardy-type inequality is useful (see~\cite[Lemma 6.1]{guo2025}).

    \begin{lemma}[Hardy-type inequality] \label{le:hardy type}
        Let us suppose that \(\rho_2^{1/2} f' \in L^{2} (\T)\), where \(f\) is an odd function satisfying \(f (0) = f (\pi) = 0\) and \(\rho_2 \coloneqq \frac{1}{4\pi \sin^2 \theta}\). Then the following estimate holds:
		\begin{align*}
			\bigg\| \frac{f}{\sin(\cdot)}\bigg\|_{L^{\infty}} \le C\ \bigg\| \frac{f'}{\sin(\cdot)} \bigg\|_{L^{2}},
		\end{align*}
        where the absolute constant $C>0$ depends on \(\T\).
    \end{lemma}

Next, we present a comparison theorem for second-order ordinary differential equations (see \cite[Section 6]{MM1949}; cf.~also \cite{MR19813,MR16827}).

	\begin{lemma}[Comparison lemma for second order ODE] \label{le:comparison the}
		Let us consider the ordinary differential equation
		\begin{align}\label{eq:ODE2}
			y'' = p_{1}y'+p_{2}y+q, \qquad x\geq x_{0},
		\end{align}
		where \(p_{1}\), \(p_{2}\), and \(q\) are continuous functions on \( [x_{0},+\infty)\), and let \( y\) be a solution of \cref{eq:ODE2} with initial conditions
		\begin{align*}
			y(x_{0})=y_{0},\qquad y'(x_{0})=y'_{0}.
		\end{align*}
		Let us suppose that there exists a solution of
		\begin{align}\label{eq:uode}
			u'' = p_{1} u' + p_{2} u,
		\end{align}
		such that
		\begin{align}\label{eq:uode-c}
			u(x)\neq 0,\qquad x_{0} < x < x_{1}.
		\end{align}
		Let \(u_{0}\) be the solution of \crefrange{eq:uode}{eq:uode-c} such that \(u_{0} (x_{0}) = 0\), \(u'_{0} (x_{0}) = 1\), and let \(X (x_{0})\) be the first zero of \( u_{0} (x)\) to the right of \(x_{0}\), if any such zero exists (otherwise let \(X (x_{0}) = + \infty\)).	
\begin{enumerate}[label=(\roman*)]
            \item If \(\phi\) satisfies
		\begin{align*}
        \begin{cases}
			\phi'' > p_{1}\, \phi' + p_{2}\, \phi + q, & x \geq x_{0},\\
			\phi(x_{0}) = y(x_{0}), \\ \phi'(x_{0}) = y'(x_{0}),
            \end{cases}
		\end{align*}
		then
		\begin{align}\label{ineq:phiy}
			\phi(x) > y(x),\qquad x_{0} <x \leq x_{1}.
		\end{align}
            \item The interval \(x_{0} < x < X(x_{0})\) is the largest one in which the inequality \cref{ineq:phiy} holds.
        \end{enumerate} 
	\end{lemma}

Finally, we state a lemma on symmetric matrices (proven in \cite[Lemma 4.1 and Lemma 7.1]{guo2025}).

    \begin{lemma}\label{le:lambda}
		For any positive integer \(k\), consider the symmetric matrix
		\begin{equation}\label{matrixAk}
			A_{k} = \left(\begin{array}{lc}
				a_{k}^{+} & \varepsilon_{k}^{+} \\
				\varepsilon_{k}^{+} & a_{k+2}^{+}
			\end{array}\right),
		\end{equation}
		where
		\begin{align*}
			a_{k}^{+} \coloneqq \left( d_{k}^{+} - d_{k+2}^{+} \right)^{2} = \left( -\frac{1}{2} - \frac{2 \kappa^{2} - 4 k- 8}{(k+2)^{2} \kappa^{2}} \right)^{2}
		\end{align*}
		and
		\begin{align}\label{epsk}
			\varepsilon_{k}^{+} \coloneqq d_{k}^{+} d_{k+2}^{+} + d_{k+2}^{+} d_{k+4}^{+} - 2 (d_{k+2}^{+})^{2}
			=\frac{-2 \kappa^{3} + 32 k + 32}{(k+2)^{4} (k+4)}.
		\end{align}
        Then \( A_{k} \) is positive definite for all \( k \geq 1\). Moreover, there exist two absolute  constants \( 0 < \lambda_{1} <  \lambda_{2} \)  such that the eigenvalues \( \lambda_{k}^{1}, \lambda_{k}^{2} \) of \( A_{k} \), as well as  coefficients \( a_{k}^{+} \) satisfy the uniform bounds
		\begin{align*}
			0 < \lambda_{1} < \lambda_{\inf} \leq a_{k}^{+}, \, \lambda_{k}^{1}, \, \lambda_{k}^{2} \leq \lambda_{\sup} < \lambda_{2}, \qquad \text{for \( k \geq 1 \),}
		\end{align*}
	 where
		\begin{align*}
			\lambda_{k}^{1} &\coloneqq \frac{a_{k}^{+} + a_{k+2}^{+} - \sqrt{(a_{k}^{+} - a_{k+2}^{+} )^{2} + 4 (\varepsilon_{k}^{+} )^{2} } } {2},\\
			\lambda_{k}^{2} &\coloneqq \frac{a_{k}^{+} + a_{k+2}^{+} + \sqrt{( a_{k}^{+} - a_{k+2}^{+} )^{2} + 4 (\varepsilon_{k}^{+} )^{2} } } {2},
		\end{align*}
		and
		\begin{align*}
			\lambda_{\sup} \coloneqq \sup_{k\geq1} \{ a_{k}^{+}, \lambda_{k}^{1}, \lambda_{k}^{2} \},\qquad  \lambda_{\inf} \coloneqq \inf_{ k \geq 1} \{ a_{k}^{+}, \lambda_{k}^{1}, \lambda_{k}^{2}\}.
		\end{align*}
        Furthermore, 
       \[ \frac{1}{50} < \lambda_{1} < \lambda_{2}<\frac{3}{5}.
        \]
	\end{lemma}

\vspace{0.5cm}

\section*{Acknowledgments}

N.~De Nitti is a member of the Gruppo Nazionale per l’Analisi Matematica, la Probabilità e le loro Applicazioni (GNAMPA) of the Istituto Nazionale di Alta Matematica (INdAM). 

We thank M.~Dai and Y.~Sun for some helpful discussions.

\vspace{0.5cm}

\printbibliography

@misc{guo2025,
  title={Stability and instability on the De Gregorio modification of the Constantin-Lax-Majda model}, 
  author={Guo, Jie and Jiu, Quansen},
  year={2025},
  eprint={2506.02800},
  archivePrefix={arXiv},
  primaryClass={math.AP},
  url={https://arxiv.org/abs/2506.02800},
}

@article{PhysRev,
  title = {The Hydromagnetic Equations},
  author = {Elsasser, Walter M.},
  journal = {Phys. Rev.},
  volume = {79},
  issue = {1},
  pages = {183--183},
  year = {1950},
  publisher = {American Physical Society},
  doi = {10.1103/PhysRev.79.183},
  url = {https://link.aps.org/doi/10.1103/PhysRev.79.183}
}

@book {MR4703940,
    AUTHOR = {Taylor, Michael E.},
     TITLE = {Partial differential equations {I}. {B}asic theory},
    SERIES = {Applied Mathematical Sciences},
    VOLUME = {115},
      NOTE = {Third edition [of  1395148]},
 PUBLISHER = {Springer, Cham},
      YEAR = {2023},
     PAGES = {xxiv+714},
      ISBN = {978-3-031-33858-8},
   MRCLASS = {35-01 (46N20 47F05 47N20)},
  MRNUMBER = {4703940},
       DOI = {10.1007/978-3-031-33859-5},
       URL = {https://doi.org/10.1007/978-3-031-33859-5},
}

@article{zbMATH08051718,
 author = {Bianchini, Roberta and Elgindi, Tarek M.},
 title = {Finite-time singularity formation for scalar stretching equations},
 fjournal = {Nonlinearity},
 journal = {Nonlinearity},
 issn = {0951-7715},
 volume = {38},
 number = {7},
 pages = {11},
 note = {Id/No 075003},
 year = {2025},
 doi = {10.1088/1361-6544/addcae},
 zbMATH = {8051718}
}

@article {MR4300258,
    AUTHOR = {Lushnikov, Pavel M. and Silantyev, Denis A. and Siegel,
              Michael},
     TITLE = {Collapse versus blow-up and global existence in the
              generalized {C}onstantin-{L}ax-{M}ajda equation},
   JOURNAL = {J. Nonlinear Sci.},
  FJOURNAL = {Journal of Nonlinear Science},
    VOLUME = {31},
      YEAR = {2021},
    NUMBER = {5},
     PAGES = {Paper No. 82, 56},
      ISSN = {0938-8974},
   MRCLASS = {35Q35 (76B03)},
  MRNUMBER = {4300258},
       DOI = {10.1007/s00332-021-09737-x},
       URL = {https://doi.org/10.1007/s00332-021-09737-x},
}

@article {MR2680191,
    AUTHOR = {Castro, A. and C\'{o}rdoba, D.},
     TITLE = {Infinite energy solutions of the surface quasi-geostrophic
              equation},
   JOURNAL = {Adv. Math.},
  FJOURNAL = {Advances in Mathematics},
    VOLUME = {225},
      YEAR = {2010},
    NUMBER = {4},
     PAGES = {1820--1829},
      ISSN = {0001-8708},
   MRCLASS = {35L65 (35B44 35C06 35Q35 44A15 76B03 86A10)},
  MRNUMBER = {2680191},
MRREVIEWER = {Catherine Choquet},
       DOI = {10.1016/j.aim.2010.04.018},
       URL = {https://doi.org/10.1016/j.aim.2010.04.018},
}

@article {MR2388660,
    AUTHOR = {Hou, Thomas Y. and Li, Congming},
     TITLE = {Dynamic stability of the three-dimensional axisymmetric
              {N}avier-{S}tokes equations with swirl},
   JOURNAL = {Comm. Pure Appl. Math.},
  FJOURNAL = {Communications on Pure and Applied Mathematics},
    VOLUME = {61},
      YEAR = {2008},
    NUMBER = {5},
     PAGES = {661--697},
      ISSN = {0010-3640},
   MRCLASS = {35Q30 (35B35 76D05)},
  MRNUMBER = {2388660},
MRREVIEWER = {Ivan Stra\v{s}kraba},
       DOI = {10.1002/cpa.20212},
       URL = {https://doi.org/10.1002/cpa.20212},
}

@article {MR3621816,
    AUTHOR = {Hoang, Vu and Radosz, Maria},
     TITLE = {Cusp formation for a nonlocal evolution equation},
   JOURNAL = {Arch. Ration. Mech. Anal.},
  FJOURNAL = {Archive for Rational Mechanics and Analysis},
    VOLUME = {224},
      YEAR = {2017},
    NUMBER = {3},
     PAGES = {1021--1036},
      ISSN = {0003-9527},
   MRCLASS = {35Q86 (35L60 86A05 86A10)},
  MRNUMBER = {3621816},
       DOI = {10.1007/s00205-017-1094-3},
       URL = {https://doi.org/10.1007/s00205-017-1094-3},
}

@article {MR4259877,
    AUTHOR = {Elgindi, Tarek M. and Ghoul, Tej-Eddine and Masmoudi, Nader},
     TITLE = {Stable self-similar blow-up for a family of nonlocal transport
              equations},
   JOURNAL = {Anal. PDE},
  FJOURNAL = {Analysis \& PDE},
    VOLUME = {14},
      YEAR = {2021},
    NUMBER = {3},
     PAGES = {891--908},
      ISSN = {2157-5045},
   MRCLASS = {35Q31 (35Q35)},
  MRNUMBER = {4259877},
MRREVIEWER = {Kyudong Choi},
       DOI = {10.2140/apde.2021.14.891},
       URL = {https://doi.org/10.2140/apde.2021.14.891},
}

@article {MR4284536,
    AUTHOR = {Chen, Jiajie},
     TITLE = {On the slightly perturbed {D}e {G}regorio model on {$S^1$}},
   JOURNAL = {Arch. Ration. Mech. Anal.},
  FJOURNAL = {Archive for Rational Mechanics and Analysis},
    VOLUME = {241},
      YEAR = {2021},
    NUMBER = {3},
     PAGES = {1843--1869},
      ISSN = {0003-9527},
   MRCLASS = {35Q31},
  MRNUMBER = {4284536},
       DOI = {10.1007/s00205-021-01685-w},
       URL = {https://doi.org/10.1007/s00205-021-01685-w},
}

@article {MR4630486,
    AUTHOR = {Huang, De and Tong, Jiajun and Wei, Dongyi},
     TITLE = {On self-similar finite-time blowups of the de {G}regorio model
              on the real line},
   JOURNAL = {Comm. Math. Phys.},
  FJOURNAL = {Communications in Mathematical Physics},
    VOLUME = {402},
      YEAR = {2023},
    NUMBER = {3},
     PAGES = {2791--2829},
      ISSN = {0010-3616},
   MRCLASS = {35B44 (35Q31)},
  MRNUMBER = {4630486},
       DOI = {10.1007/s00220-023-04784-9},
       URL = {https://doi.org/10.1007/s00220-023-04784-9},
}

@article {MR4721709,
    AUTHOR = {Huang, De and Qin, Xiang and Wang, Xiuyuan and Wei, Dongyi},
     TITLE = {Self-similar finite-time blowups with smooth profiles of the
              generalized {C}onstantin-{L}ax-{M}ajda model},
   JOURNAL = {Arch. Ration. Mech. Anal.},
  FJOURNAL = {Archive for Rational Mechanics and Analysis},
    VOLUME = {248},
      YEAR = {2024},
    NUMBER = {2},
     PAGES = {Paper No. 22, 65},
      ISSN = {0003-9527},
   MRCLASS = {35Q31 (35B44 76B47)},
  MRNUMBER = {4721709},
MRREVIEWER = {In-Jee Jeong},
       DOI = {10.1007/s00205-024-01971-3},
       URL = {https://doi.org/10.1007/s00205-024-01971-3},
}

@article {MR4065651,
    AUTHOR = {Elgindi, Tarek M. and Jeong, In-Jee},
     TITLE = {On the effects of advection and vortex stretching},
   JOURNAL = {Arch. Ration. Mech. Anal.},
  FJOURNAL = {Archive for Rational Mechanics and Analysis},
    VOLUME = {235},
      YEAR = {2020},
    NUMBER = {3},
     PAGES = {1763--1817},
      ISSN = {0003-9527},
   MRCLASS = {76B47 (35Q35 76B03)},
  MRNUMBER = {4065651},
       DOI = {10.1007/s00205-019-01455-9},
       URL = {https://doi.org/10.1007/s00205-019-01455-9},
}

@article {MR4242826,
    AUTHOR = {Chen, Jiajie and Hou, Thomas Y. and Huang, De},
     TITLE = {On the finite time blowup of the {D}e {G}regorio model for the
              3{D} {E}uler equations},
   JOURNAL = {Comm. Pure Appl. Math.},
  FJOURNAL = {Communications on Pure and Applied Mathematics},
    VOLUME = {74},
      YEAR = {2021},
    NUMBER = {6},
     PAGES = {1282--1350},
      ISSN = {0010-3640},
   MRCLASS = {35Q30 (76B03 76D03)},
  MRNUMBER = {4242826},
       DOI = {10.1002/cpa.21991},
       URL = {https://doi.org/10.1002/cpa.21991},
}

@article {MR2963789,
    AUTHOR = {Guo, Yan and Tice, Ian},
     TITLE = {Compressible, inviscid {R}ayleigh-{T}aylor instability},
   JOURNAL = {Indiana Univ. Math. J.},
  FJOURNAL = {Indiana University Mathematics Journal},
    VOLUME = {60},
      YEAR = {2011},
    NUMBER = {2},
     PAGES = {677--711},
      ISSN = {0022-2518},
   MRCLASS = {35Q35 (35Q31)},
  MRNUMBER = {2963789},
MRREVIEWER = {Rapha\"{e}l Danchin},
       DOI = {10.1512/iumj.2011.60.4193},
       URL = {https://doi.org/10.1512/iumj.2011.60.4193},
}

@book {MR3244169,
    AUTHOR = {Debnath, Lokenath and Bhatta, Dambaru},
     TITLE = {Integral transforms and their applications},
   EDITION = {Third edition},
 PUBLISHER = {CRC Press, Boca Raton, FL},
      YEAR = {2015},
     PAGES = {xxvi+792},
      ISBN = {978-1-4822-2357-6},
   MRCLASS = {44-01 (35A22 35A25)},
  MRNUMBER = {3244169},
}

@article {MR4875613,
    AUTHOR = {Chen, Jiajie},
     TITLE = {On the regularity of the {D}e {G}regorio model for the 3{D}
              {E}uler equations},
   JOURNAL = {J. Eur. Math. Soc. (JEMS)},
  FJOURNAL = {Journal of the European Mathematical Society (JEMS)},
    VOLUME = {27},
      YEAR = {2025},
    NUMBER = {4},
     PAGES = {1619--1677},
      ISSN = {1435-9855},
   MRCLASS = {35Q31 (35B65 35Q35)},
  MRNUMBER = {4875613},
       DOI = {10.4171/jems/1399},
       URL = {https://doi.org/10.4171/jems/1399},
}

@article{zbMATH04172622,
 author = {De Gregorio, Salvatore},
 title = {On a one-dimensional model for the three-dimensional vorticity equation},
 fjournal = {Journal of Statistical Physics},
 journal = {J. Stat. Phys.},
 issn = {0022-4715},
 volume = {59},
 number = {5-6},
 pages = {1251--1263},
 year = {1990},
 doi = {10.1007/BF01334750},
 zbMATH = {4172622},
 Zbl = {0712.76027}
}

@misc{daiArXiv,
      title={1D Model for the 3D magnetohydrodynamics}, 
      author={Mimi Dai and Bhakti Vyas and Xiangxiong Zhang},
      eprint={2107.02920v2},
      year={2021},
      archivePrefix={arXiv},
      primaryClass={math.AP},
      url={https://arxiv.org/pdf/2107.02920v2}, 
}

@article {MR16827,
    AUTHOR = {Petrov, V. N.},
     TITLE = {The limits of applicability of {S}. {T}chaplygin's theorem on
              differential inequalities to linear equations with usual
              derivatives of the second order},
   JOURNAL = {C. R. (Doklady) Acad. Sci. URSS (N.S.)},
  FJOURNAL = {C. R. (Doklady) Acad. Sci. URSS (N.S.)},
    VOLUME = {51},
      YEAR = {1946},
     PAGES = {255--258},
   MRCLASS = {36.0X},
  MRNUMBER = {16827},
MRREVIEWER = {J. E. Wilkins, Jr.},
}

@article {MR19813,
    AUTHOR = {Wilkins, Jr., J. Ernest},
     TITLE = {The converse of a theorem of {T}chaplygin on differential
              inequalities},
   JOURNAL = {Bull. Amer. Math. Soc.},
  FJOURNAL = {Bulletin of the American Mathematical Society},
    VOLUME = {53},
      YEAR = {1947},
     PAGES = {126--129},
      ISSN = {0002-9904},
   MRCLASS = {36.0X},
  MRNUMBER = {19813},
MRREVIEWER = {R. Bellman},
       DOI = {10.1090/S0002-9904-1947-08759-0},
       URL = {https://doi.org/10.1090/S0002-9904-1947-08759-0},
}

@book {MR1721989,
    AUTHOR = {Engel, Klaus-Jochen and Nagel, Rainer},
     TITLE = {One-parameter semigroups for linear evolution equations},
    SERIES = {Graduate Texts in Mathematics},
    VOLUME = {194},
      NOTE = {With contributions by S. Brendle, M. Campiti, T. Hahn, G.
              Metafune, G. Nickel, D. Pallara, C. Perazzoli, A. Rhandi, S.
              Romanelli and R. Schnaubelt},
 PUBLISHER = {Springer-Verlag, New York},
      YEAR = {2000},
     PAGES = {xxii+586},
      ISBN = {0-387-98463-1},
   MRCLASS = {47D06 (34G10 35K90 47N20)},
  MRNUMBER = {1721989},
MRREVIEWER = {Charles Batty},
}

@article {sun2025,
    AUTHOR = {Sun, Yunhao},
     TITLE = {Stability of a stationary solution to a 1-{D} model for the
              {MHD}},
   JOURNAL = {Discrete Contin. Dyn. Syst.},
  FJOURNAL = {Discrete and Continuous Dynamical Systems. Series A},
    VOLUME = {45},
      YEAR = {2025},
    NUMBER = {10},
     PAGES = {3545--3564},
      ISSN = {1078-0947},
   MRCLASS = {35B35 (35Q35 76B03 76W05)},
  MRNUMBER = {4897301},
       DOI = {10.3934/dcds.2025031},
       URL = {https://doi.org/10.3934/dcds.2025031},
}

@article {dai2023,
    AUTHOR = {Dai, Mimi and Vyas, Bhakti and Zhang, Xiangxiong},
     TITLE = {1{D} model for the 3{D} magnetohydrodynamics},
   JOURNAL = {J. Nonlinear Sci.},
  FJOURNAL = {Journal of Nonlinear Science},
    VOLUME = {33},
      YEAR = {2023},
    NUMBER = {5},
     PAGES = {Paper No. 87, 38},
      ISSN = {0938-8974},
   MRCLASS = {35Q35 (65M70 76B03 76D03 76W05)},
  MRNUMBER = {4621540},
MRREVIEWER = {Luc Paquet},
       DOI = {10.1007/s00332-023-09944-8},
       URL = {https://doi.org/10.1007/s00332-023-09944-8},
}

@article {JSS2019,
    AUTHOR = {Jia, H. and Stewart, S. and Sverak, V.},
     TITLE = {On the {D}e {G}regorio modification of the
              {C}onstantin-{L}ax-{M}ajda model},
   JOURNAL = {Arch. Ration. Mech. Anal.},
  FJOURNAL = {Archive for Rational Mechanics and Analysis},
    VOLUME = {231},
      YEAR = {2019},
    NUMBER = {2},
     PAGES = {1269--1304},
      ISSN = {0003-9527},
   MRCLASS = {35Q35},
  MRNUMBER = {3900823},
MRREVIEWER = {Alberto Valli},
       DOI = {10.1007/s00205-018-1298-1},
       URL = {https://doi.org/10.1007/s00205-018-1298-1},
}

@article {MR2439488,
    AUTHOR = {Okamoto, Hisashi and Sakajo, Takashi and Wunsch, Marcus},
     TITLE = {On a generalization of the {C}onstantin-{L}ax-{M}ajda
              equation},
   JOURNAL = {Nonlinearity},
  FJOURNAL = {Nonlinearity},
    VOLUME = {21},
      YEAR = {2008},
    NUMBER = {10},
     PAGES = {2447--2461},
      ISSN = {0951-7715},
   MRCLASS = {76B03 (35Q35)},
  MRNUMBER = {2439488},
MRREVIEWER = {Koji Ohkitani},
       DOI = {10.1088/0951-7715/21/10/013},
       URL = {https://doi.org/10.1088/0951-7715/21/10/013},
}

@article {MR2179734,
    AUTHOR = {C\'{o}rdoba, Antonio and C\'{o}rdoba, Diego and Fontelos, Marco A.},
     TITLE = {Formation of singularities for a transport equation with
              nonlocal velocity},
   JOURNAL = {Ann. of Math. (2)},
  FJOURNAL = {Annals of Mathematics. Second Series},
    VOLUME = {162},
      YEAR = {2005},
    NUMBER = {3},
     PAGES = {1377--1389},
      ISSN = {0003-486X},
   MRCLASS = {35A20 (35Q35 82C70)},
  MRNUMBER = {2179734},
MRREVIEWER = {Zoran Gruji\'{c}},
       DOI = {10.4007/annals.2005.162.1377},
       URL = {https://doi.org/10.4007/annals.2005.162.1377},
}

@article{zbMATH05531021,
 author = {Hou, Thomas Y. and Lei, Zhen},
 title = {On the stabilizing effect of convection in three-dimensional incompressible flows},
 fjournal = {Communications on Pure and Applied Mathematics},
 journal = {Commun. Pure Appl. Math.},
 issn = {0010-3640},
 volume = {62},
 number = {4},
 pages = {501--564},
 year = {2009},
 doi = {10.1002/cpa.20254},
 zbMATH = {5531021},
 Zbl = {1171.35095}
}

@article {MR812343,
    AUTHOR = {Constantin, P. and Lax, P. D. and Majda, A.},
     TITLE = {A simple one-dimensional model for the three-dimensional
              vorticity equation},
   JOURNAL = {Comm. Pure Appl. Math.},
  FJOURNAL = {Communications on Pure and Applied Mathematics},
    VOLUME = {38},
      YEAR = {1985},
    NUMBER = {6},
     PAGES = {715--724},
      ISSN = {0010-3640},
   MRCLASS = {76C05 (35Q99)},
  MRNUMBER = {812343},
MRREVIEWER = {W. O. Criminale, Jr.},
       DOI = {10.1002/cpa.3160380605},
       URL = {https://doi.org/10.1002/cpa.3160380605},
}

@article {DG1996,
    AUTHOR = {De Gregorio, Salvatore},
     TITLE = {A partial differential equation arising in a {$1$}{D} model
              for the {$3$}{D} vorticity equation},
   JOURNAL = {Math. Methods Appl. Sci.},
  FJOURNAL = {Mathematical Methods in the Applied Sciences},
    VOLUME = {19},
      YEAR = {1996},
    NUMBER = {15},
     PAGES = {1233--1255},
      ISSN = {0170-4214},
   MRCLASS = {76C05 (35Q53)},
  MRNUMBER = {1410208},
MRREVIEWER = {J. Thomas Beale},
       DOI =
              {10.1002/(SICI)1099-1476(199610)19:15<1233::AID-MMA828>3.3.CO;2-N},
       URL =
              {https://doi.org/10.1002/(SICI)1099-1476(199610)19:15<1233::AID-MMA828>3.3.CO;2-N},
}

@article {LLR2020,
    AUTHOR = {Lei, Zhen and Liu, Jie and Ren, Xiao},
     TITLE = {On the {C}onstantin-{L}ax-{M}ajda model with convection},
   JOURNAL = {Comm. Math. Phys.},
  FJOURNAL = {Communications in Mathematical Physics},
    VOLUME = {375},
      YEAR = {2020},
    NUMBER = {1},
     PAGES = {765--783},
      ISSN = {0010-3616},
   MRCLASS = {35Q35 (35Q31)},
  MRNUMBER = {4082178},
MRREVIEWER = {Matthew R. I. Schrecker},
       DOI = {10.1007/s00220-019-03584-4},
       URL = {https://doi.org/10.1007/s00220-019-03584-4},
}

@article {MM1949,
	AUTHOR = {Matos Peixoto, Mauricio},
	TITLE = {Generalized convex functions and second order differential
	inequalities},
	JOURNAL = {Bull. Amer. Math. Soc.},
	FJOURNAL = {Bulletin of the American Mathematical Society},
	VOLUME = {55},
	YEAR = {1949},
	PAGES = {563--572},
	ISSN = {0002-9904},
	MRCLASS = {27.0X},
	MRNUMBER = {29949},
	MRREVIEWER = {E.\ F.\ Beckenbach},
	DOI = {10.1090/S0002-9904-1949-09246-7},
	URL = {https://doi.org/10.1090/S0002-9904-1949-09246-7},
}

@article{zbMATH08048827,
 author = {Dai, Mimi},
 title = {Reduced models for electron magnetohydrodynamics: well-posedness and singularity formation},
 fjournal = {Transactions of the American Mathematical Society},
 journal = {Trans. Am. Math. Soc.},
 issn = {0002-9947},
 volume = {378},
 number = {6},
 pages = {3981--4009},
 year = {2025},
 doi = {10.1090/tran/9427},
 zbMATH = {8048827}
}

@article {MR2662457,
    AUTHOR = {Kiselev, A.},
     TITLE = {Regularity and blow up for active scalars},
   JOURNAL = {Math. Model. Nat. Phenom.},
  FJOURNAL = {Mathematical Modelling of Natural Phenomena},
    VOLUME = {5},
      YEAR = {2010},
    NUMBER = {4},
     PAGES = {225--255},
      ISSN = {0973-5348},
   MRCLASS = {35Q35 (35B44 76U05)},
  MRNUMBER = {2662457},
MRREVIEWER = {Masahito Ohta},
       DOI = {10.1051/mmnp/20105410},
       URL = {https://doi.org/10.1051/mmnp/20105410},
}

@article {MR840339,
    AUTHOR = {Schochet, Steven},
     TITLE = {Explicit solutions of the viscous model vorticity equation},
   JOURNAL = {Comm. Pure Appl. Math.},
  FJOURNAL = {Communications on Pure and Applied Mathematics},
    VOLUME = {39},
      YEAR = {1986},
    NUMBER = {4},
     PAGES = {531--537},
      ISSN = {0010-3640},
   MRCLASS = {35Q20 (76C05)},
  MRNUMBER = {840339},
MRREVIEWER = {W. O. Criminale, Jr.},
       DOI = {10.1002/cpa.3160390404},
       URL = {https://doi.org/10.1002/cpa.3160390404},
}

@article {MR2397459,
    AUTHOR = {Li, Dong and Rodrigo, Jose},
     TITLE = {Blow-up of solutions for a 1{D} transport equation with
              nonlocal velocity and supercritical dissipation},
   JOURNAL = {Adv. Math.},
  FJOURNAL = {Advances in Mathematics},
    VOLUME = {217},
      YEAR = {2008},
    NUMBER = {6},
     PAGES = {2563--2568},
      ISSN = {0001-8708},
   MRCLASS = {35F20 (35B40 76B03)},
  MRNUMBER = {2397459},
MRREVIEWER = {Hidetoshi Tahara},
       DOI = {10.1016/j.aim.2007.11.002},
       URL = {https://doi.org/10.1016/j.aim.2007.11.002},
}

@article {MR2865810,
    AUTHOR = {Wunsch, Marcus},
     TITLE = {The generalized {C}onstantin-{L}ax-{M}ajda equation revisited},
   JOURNAL = {Commun. Math. Sci.},
  FJOURNAL = {Communications in Mathematical Sciences},
    VOLUME = {9},
      YEAR = {2011},
    NUMBER = {3},
     PAGES = {929--936},
      ISSN = {1539-6746},
   MRCLASS = {35Q35 (35A20 35B44 76B03)},
  MRNUMBER = {2865810},
MRREVIEWER = {Takashi Suzuki},
       DOI = {10.4310/CMS.2011.v9.n3.a12},
       URL = {https://doi.org/10.4310/CMS.2011.v9.n3.a12},
}

@article {MR4105366,
    AUTHOR = {Chen, Jiajie},
     TITLE = {Singularity formation and global well-posedness for the
              generalized {C}onstantin-{L}ax-{M}ajda equation with
              dissipation},
   JOURNAL = {Nonlinearity},
  FJOURNAL = {Nonlinearity},
    VOLUME = {33},
      YEAR = {2020},
    NUMBER = {5},
     PAGES = {2502--2532},
      ISSN = {0951-7715},
   MRCLASS = {35C06 (35A20 35B30 76D03)},
  MRNUMBER = {4105366},
MRREVIEWER = {Dmitrii Legatiuk},
       DOI = {10.1088/1361-6544/ab74b0},
       URL = {https://doi.org/10.1088/1361-6544/ab74b0},
}

@article{zbMATH07789595,
 author = {Ambrose, David M. and Lushnikov, Pavel M. and Siegel, Michael and Silantyev, Denis A.},
 title = {Global existence and singularity formation for the generalized {Constantin}-{Lax}-{Majda} equation with dissipation: the real line vs. periodic domains},
 fjournal = {Nonlinearity},
 journal = {Nonlinearity},
 issn = {0951-7715},
 volume = {37},
 number = {2},
 pages = {43},
 note = {Id/No 025004},
 year = {2024},
 doi = {10.1088/1361-6544/ad140c},
 zbMATH = {7789595},
 Zbl = {1530.35205}
}

@misc{dai2025wellposednessblowup1delectron,
      title={Well-posedness and blowup of 1D electron magnetohydrodynamics}, 
      author={Mimi Dai},
      year={2025},
      eprint={2503.06383},
      archivePrefix={arXiv},
      primaryClass={math.AP},
      url={https://arxiv.org/abs/2503.06383}, 
}

@book {BN1971,
    AUTHOR = {Butzer, Paul L. and Nessel, Rolf J.},
     TITLE = {Fourier analysis and approximation. {Vol}. 1: {One}-dimensional theory},
    SERIES = {Pure and Applied Mathematics, Vol. 40},
 PUBLISHER = {Academic Press, New York-London},
      YEAR = {1971},
     PAGES = {xvi+553},
   MRCLASS = {42A08 (41-01)},
  MRNUMBER = {510857},
}

@article {JJN2013,
    AUTHOR = {Jiang, Fei and Jiang, Song and Ni, GuoXi},
     TITLE = {Nonlinear instability for nonhomogeneous incompressible
              viscous fluids},
   JOURNAL = {Sci. China Math.},
  FJOURNAL = {Science China. Mathematics},
    VOLUME = {56},
      YEAR = {2013},
    NUMBER = {4},
     PAGES = {665--686},
      ISSN = {1674-7283,1869-1862},
   MRCLASS = {35Q35 (35B65 76D03 76E30)},
  MRNUMBER = {3034832},
MRREVIEWER = {Kazuo\ Yamazaki},
       DOI = {10.1007/s11425-013-4587-z},
       URL = {https://doi.org/10.1007/s11425-013-4587-z},
}

@book {NS2004,
    AUTHOR = {Novotn\'y, A. and Stra\v skraba, I.},
     TITLE = {Introduction to the mathematical theory of compressible flow},
    SERIES = {Oxford Lecture Series in Mathematics and its Applications},
    VOLUME = {27},
 PUBLISHER = {Oxford University Press, Oxford},
      YEAR = {2004},
     PAGES = {xx+506},
      ISBN = {0-19-853084-6},
   MRCLASS = {35Q35 (35Q30 76N10)},
  MRNUMBER = {2084891},
MRREVIEWER = {Rapha\"el\ Danchin},
}

\vfill 

\end{document}